\documentclass{article}
\usepackage{amsmath}
\usepackage{amssymb}
\usepackage{amsthm}
\usepackage{bbm}
 \usepackage{cite}
\usepackage{xcolor}
\usepackage{stmaryrd}
\SetSymbolFont{stmry}{bold}{U}{stmry}{m}{n}
\usepackage{euscript}
\usepackage[arrow,curve,matrix,arc,2cell,frame]{xy}
\UseAllTwocells
\usepackage[utf8]{inputenc}
\usepackage[pdftex,breaklinks,unicode]{hyperref}
\usepackage{slashed}
\DeclareFontFamily{U}{rsfs}{} 
\DeclareFontShape{U}{rsfs}{n}{it}{<->
rsfs10}{} \DeclareSymbolFont{mscr}{U}{rsfs}{n}{it}
\DeclareSymbolFontAlphabet{\scr}{mscr}
\def\mathscr{\scr}
\begin{document}
\def\e#1\e{\begin{equation}#1\end{equation}}
\def\ea#1\ea{\begin{align}#1\end{align}}
\def\eq#1{{\rm(\ref{#1})}}
\theoremstyle{plain}
\newtheorem{thm}{Theorem}[section]
\newtheorem{lem}[thm]{Lemma}
\newtheorem{prop}[thm]{Proposition}
\newtheorem{cor}[thm]{Corollary}
\newtheorem{quest}[thm]{Question}
\newtheorem{prob}[thm]{Problem}
\newtheorem{claim}[thm]{Claim}
\theoremstyle{definition}
\newtheorem{dfn}[thm]{Definition}
\newtheorem{ex}[thm]{Example}
\newtheorem{rem}[thm]{Remark}
\newtheorem{conj}[thm]{Conjecture}
\newtheorem{princ}[thm]{Principle}
\newtheorem{ax}[thm]{Axiom}
\newtheorem{ass}[thm]{Assumption}
\newtheorem{property}[thm]{Property}
\newtheorem{cond}[thm]{Condition}
\newtheorem{conc}[thm]{Conclusion}
\numberwithin{figure}{section}
\numberwithin{equation}{section}
\def\depth{\mathop{\rm depth}\nolimits}
\def\dim{\mathop{\rm dim}\nolimits}
\def\bc{\mathop{\rm bc}\nolimits}
\def\ec{\mathop{\rm ec}\nolimits}
\def\codim{\mathop{\rm codim}\nolimits}
\def\vdim{\mathop{\rm vdim}\nolimits}
\def\sign{\mathop{\rm sign}\nolimits}
\def\Im{\mathop{\rm Im}\nolimits}
\def\det{\mathop{\rm det}\nolimits}
\def\Res{\mathop{\rm Res}\nolimits}
\def\Area{\mathop{\rm Area}\nolimits}
\def\Type{\mathop{\rm Type}\nolimits}
\def\Turn{\mathop{\rm Turn}\nolimits}
\def\Ker{\mathop{\rm Ker}}
\def\Coker{\mathop{\rm Coker}}
\def\Spec{\mathop{\rm Spec}}
\def\Perf{\mathop{\rm Perf}}
\def\Vect{\mathop{\rm Vect}}
\def\Iso{\mathop{\rm Iso}\nolimits}
\def\Aut{\mathop{\rm Aut}}
\def\End{\mathop{\rm End}}
\def\Ho{\mathop{\rm Ho}}
\def\PGL{\mathop{\rm PGL}\nolimits}
\def\GL{\mathop{\rm GL}\nolimits}
\def\SL{\mathop{\rm SL}}
\def\SO{\mathop{\rm SO}}
\def\SU{\mathop{\rm SU}}
\def\Sp{\mathop{\rm Sp}}
\def\Re{\mathop{\rm Re}}
\def\Pd{\mathop{\rm Pd}}
\def\ev{{\rm ev}}
\def\ror{{\rm or}}
\def\Orb{{\mathop{\bf Orb}}}
\def\Orbb{{\mathop{\bf Orb^b}}}
\def\Orbc{{\mathop{\bf Orb^c}}}
\def\Man{{\mathop{\bf Man}}}
\def\Manb{{\mathop{\bf Man^b}}}
\def\Manc{{\mathop{\bf Man^c}}}
\def\Mancin{{\mathop{\bf Man^c_{in}}}}  
\def\Mancbn{{\mathop{\bf Man^c_{bn}}}}
\def\Mancsi{{\mathop{\bf Man^c_{si}}}} 
\def\Mancst{{\mathop{\bf Man^c_{st}}}} 
\def\Mancis{{\mathop{\bf Man^c_{is}}}}
\def\cManc{{\mathop{\bf\check{M}an^c}}}
\def\cMancin{{\mathop{\bf\check{M}an^c_{in}}}} 
\def\cMancbn{{\mathop{\bf\check{M}an^c_{bn}}}} 
\def\cMancst{{\mathop{\bf\check{M}an^c_{st}}}}
\def\cMancis{{\mathop{\bf\check{M}an^c_{is}}}}
\def\cMancsi{{\mathop{\bf\check{M}an^c_{si}}}}
\def\Mangc{{\mathop{\bf Man^{gc}}}}
\def\cMangc{{\mathop{\bf\check{M}an^{gc}}}}
\def\Mangcin{{\mathop{\bf Man^{gc}_{in}}}}
\def\cMangcin{{\mathop{\bf\check{M}an^{gc}_{in}}}}
\def\Manac{{\mathop{\bf Man^{ac}}}}
\def\cManac{{\mathop{\bf\check{M}an^{ac}}}}
\def\Manacin{{\mathop{\bf Man^{ac}_{in}}}}
\def\cManacin{{\mathop{\bf\check{M}an^{ac}_{in}}}}
\def\SMan{{\mathop{\bf SMan}}}
\def\SManb{{\mathop{\bf SMan^b}}}
\def\SManc{{\mathop{\bf SMan^c}}}
\def\cSMan{{\mathop{\bf \check{S}Man}}}
\def\cSManb{{\mathop{\bf \check{S}Man^b}}}
\def\cSManc{{\mathop{\bf \check{S}Man^c}}}
\def\SManbin{{\mathop{\bf SMan^b_{in}}}}
\def\SManbbn{{\mathop{\bf SMan^b_{bn}}}}
\def\SMancin{{\mathop{\bf SMan^c_{in}}}}
\def\SMancbn{{\mathop{\bf SMan^c_{bn}}}}
\def\cSManbin{{\mathop{\bf\check{S}Man^b_{in}}}}
\def\cSManbbn{{\mathop{\bf\check{S}Man^b_{bn}}}}
\def\cSMancin{{\mathop{\bf\check{S}Man^c_{in}}}}
\def\cSMancbn{{\mathop{\bf\check{S}Man^c_{bn}}}}
\def\cSManbbn{{\mathop{\bf\check{S}Man^b_{bn}}}}
\def\CSch{{\mathop{\bf C^{\bs\iy}Sch}}}
\def\inc{\mathop{\rm inc}\nolimits}
\def\St{{\rm St}}
\def\rin{{\rm in}}
\def\ror{{\rm or}}
\def\ind{\mathop{\rm ind}\nolimits}
\def\glb{\mathop{\rm glb}\nolimits}
\def\lub{\mathop{\rm lub}\nolimits}
\def\bdim{{\mathbin{\bf dim}\kern.1em}}
\def\rk{\mathop{\rm rk}}
\def\colim{\mathop{\rm colim}\nolimits}
\def\Stab{\mathop{\rm Stab}\nolimits}
\def\supp{\mathop{\rm supp}}
\def\Or{\mathop{\rm Or}\nolimits}
\def\rank{\mathop{\rm rank}\nolimits}
\def\Hom{\mathop{\rm Hom}\nolimits}
\def\bHom{\mathop{\bf Hom}\nolimits}
\def\id{{\mathop{\rm id}\nolimits}}
\def\Id{{\mathop{\rm Id}\nolimits}}
\def\cs{{\rm cs}}
\def\lf{{\rm lf}}
\def\pl{{\rm pl}}
\def\eu{{\rm eu}}
\def\mix{{\rm mix}}
\def\rint{{\rm int}}
\def\fd{{\rm fd}}
\def\wt{{\rm wt}}
\def\ran{{\rm an}}
\def\fpd{{\rm fpd}}
\def\pfd{{\rm pfd}}
\def\coa{{\rm coa}}
\def\rsi{{\rm si}}
\def\ssi{{\rm ssi}}
\def\rst{{\rm st}}
\def\ss{{\rm ss}}
\def\vi{{\rm vi}}
\def\smq{{\rm smq}}
\def\rsm{{\rm sm}}
\def\cla{{\rm cla}}
\def\rArt{{\rm Art}}
\def\po{{\rm po}}
\def\pe{{\rm pe}}
\def\rp{{\rm rp}}
\def\spo{{\rm spo}}
\def\Kur{{\rm Kur}}
\def\dcr{{\rm dcr}}
\def\top{{\rm top}}
\def\fc{{\rm fc}}
\def\cla{{\rm cla}}
\def\num{{\rm num}}
\def\irr{{\rm irr}}
\def\red{{\rm red}}
\def\sing{{\rm sing}}
\def\virt{{\rm virt}}
\def\inv{{\rm inv}}
\def\fund{{\rm fund}}
\def\qcoh{{\rm qcoh}}
\def\coh{{\rm coh}}
\def\vect{{\rm vect}}
\def\lft{{\rm lft}}
\def\lfp{{\rm lfp}}
\def\cs{{\rm cs}}
\def\BM{{\sst\rm BM}}
\def\dR{{\rm dR}}
\def\Obj{{\rm Obj}}
\def\Top{{\mathop{\bf Top}\nolimits}}
\def\ul{\underline}
\def\bs{\boldsymbol}
\def\ge{\geqslant}
\def\le{\leqslant\nobreak}
\def\O{{\mathcal O}}
\def\bA{{\mathbin{\mathbb A}}}
\def\bF{{\mathbin{\mathbb F}}}
\def\bG{{\mathbin{\mathbb G}}}
\def\bH{{\mathbin{\mathbb H}}}
\def\bL{{\mathbin{\mathbb L}}}
\def\P{{\mathbin{\mathbb P}}}
\def\K{{\mathbin{\mathbb K}}}
\def\R{{\mathbin{\mathbb R}}}
\def\bT{{\mathbin{\mathbb T}}}
\def\Z{{\mathbin{\mathbb Z}}}
\def\bP{{\mathbin{\mathbb P}}}
\def\Q{{\mathbin{\mathbb Q}}}
\def\N{{\mathbin{\mathbb N}}}
\def\C{{\mathbin{\mathbb C}}}
\def\CP{{\mathbin{\mathbb{CP}}}}
\def\KP{{\mathbin{\mathbb{KP}}}}
\def\RP{{\mathbin{\mathbb{RP}}}}
\def\fC{{\mathbin{\mathfrak C}\kern.05em}}
\def\fD{{\mathbin{\mathfrak D}}}
\def\fE{{\mathbin{\mathfrak E}}}
\def\fF{{\mathbin{\mathfrak F}}}
\def\A{{\mathbin{\cal A}}}
\def\G{{{\cal G}}}
\def\M{{\mathbin{\cal M}}}
\def\B{{\mathbin{\cal B}}}  
\def\cC{{\mathbin{\cal C}}}
\def\cD{{\mathbin{\cal D}}}
\def\cE{{\mathbin{\cal E}}}
\def\cF{{\mathbin{\cal F}}}
\def\cG{{\mathbin{\cal G}}}
\def\cH{{\mathbin{\cal H}}}
\def\cI{{\mathbin{\cal I}}}
\def\cJ{{\mathbin{\cal J}}}
\def\cK{{\mathbin{\cal K}}}
\def\cL{{\mathbin{\cal L}}}
\def\cB{{\mathbin{\cal B}}}
\def\bcM{{\mathbin{\bs{\cal M}}}}
\def\cN{{\cal N}}
\def\cP{{\mathbin{\cal P}}}
\def\cQ{{\mathbin{\cal Q}}}
\def\cR{{\mathbin{\cal R}}}
\def\cS{{\mathbin{\cal S}}}
\def\T{{{\cal T}\kern .04em}}
\def\cU{{\mathbin{\cal U}}}
\def\cV{{\mathbin{\cal V}}}
\def\cW{{\mathbin{\cal W}}}
\def\cX{{\cal X}}
\def\cY{{\cal Y}}
\def\cZ{{\cal Z}}
\def\oM{{\mathbin{\smash{\,\,\overline{\!\!\mathcal M\!}\,}}\vphantom{\cal M}}}
\def\tiM{{\mathbin{\smash{\ti{\mathcal M}}}\vphantom{\cal M}}}
\def\cV{{\cal V}}
\def\cW{{\cal W}}
\def\sC{{{\mathscr C}}}
\def\sF{{{\mathscr F}}}
\def\sR{{{\mathscr R}}}
\def\sS{{{\mathscr S}}}
\def\sT{{{\mathscr T}}}
\def\sU{{{\mathscr U}}}
\def\sV{{{\mathscr V}}}
\def\sW{{{\mathscr W}}}
\def\b{{\mathfrak b}}
\def\fe{{\mathfrak e}}
\def\f{{\mathfrak f}}
\def\g{{\mathfrak g}}
\def\h{{\mathfrak h}}
\def\k{{\mathfrak k}}
\def\m{{\mathfrak m}}
\def\n{{\mathfrak n}}
\def\p{{\mathfrak p}}
\def\q{{\mathfrak q}}
\def\u{{\mathfrak u}}
\def\H{{\mathfrak H}}
\def\so{{\mathfrak{so}}}
\def\su{{\mathfrak{su}}}
\def\sp{{\mathfrak{sp}}}
\def\fW{{\mathfrak W}}
\def\fX{{\mathfrak X}}
\def\fY{{\mathfrak Y}}
\def\fZ{{\mathfrak Z}}
\def\bfb{{\bs{\mathfrak b}}}
\def\bfc{{\bs{\mathfrak c}}}
\def\bfd{{\bs{\mathfrak d}}}
\def\bfe{{\bs{\mathfrak e}}}
\def\bff{{\bs{\mathfrak f}}}
\def\bfg{{\bs{\mathfrak g}}}
\def\bfh{{\bs{\mathfrak h}}}
\def\bfU{{\bs{\mathfrak U}}}
\def\bfV{{\bs{\mathfrak V}}}
\def\bfW{{\bs{\mathfrak W}}}
\def\bfX{{\bs{\mathfrak X}}}
\def\bfY{{\bs{\mathfrak Y}}}
\def\bfZ{{\bs{\mathfrak Z}}}
\def\bE{{\bs E}}
\def\bM{{\bs M}}
\def\bN{{\bs N}}
\def\bO{{\bs O}}
\def\bQ{{\bs Q}}
\def\bS{{\bs S}}
\def\bU{{\bs U}}
\def\bV{{\bs V}}
\def\bW{{\bs W}\kern -0.1em}
\def\bX{{\bs X}}
\def\bY{{\bs Y}\kern -0.1em}
\def\bZ{{\bs Z}}
\def\al{\alpha}
\def\be{\beta}
\def\ga{\gamma}
\def\de{\delta}
\def\io{\iota}
\def\ep{\epsilon}
\def\la{\lambda}
\def\ka{\kappa}
\def\th{\theta}
\def\ze{\zeta}
\def\up{\upsilon}
\def\vp{\varphi}
\def\si{\sigma}
\def\om{\omega}
\def\De{\Delta}
\def\Ka{{\rm K}}
\def\La{\Lambda}
\def\Om{\Omega}
\def\Ga{\Gamma}
\def\Si{\Sigma}
\def\Th{\Theta}
\def\Up{\Upsilon}
\def\Chi{{\rm X}}
\def\Tau{{\rm T}}
\def\Nu{{\rm N}}
\def\pd{\partial}
\def\ts{\textstyle}
\def\st{\scriptstyle}
\def\sst{\scriptscriptstyle}
\def\w{\wedge}
\def\sm{\setminus}
\def\lt{\ltimes}
\def\bu{\bullet}
\def\sh{\sharp}
\def\di{\diamond}
\def\he{\heartsuit}
\def\od{\odot}
\def\op{\oplus}
\def\ot{\otimes}
\def\hot{\mathbin{\hat\otimes}}
\def\bt{\boxtimes}
\def\bp{\boxplus}
\def\ov{\overline}
\def\bigop{\bigoplus}
\def\bigot{\bigotimes}
\def\tr{\blacktriangle}
\def\iy{\infty}
\def\es{\emptyset}
\def\ra{\rightarrow}
\def\rra{\rightrightarrows}
\def\Ra{\Rightarrow}
\def\Longra{\Longrightarrow}
\def\ab{\allowbreak}
\def\longra{\longrightarrow}
\def\hookra{\hookrightarrow}
\def\dashra{\dashrightarrow}
\def\lb{\llbracket}
\def\pf{\pitchfork}
\def\rb{\rrbracket}
\def\ha{{\ts\frac{1}{2}}}
\def\t{\times}
\def\pr{\preceq}
\def\tl{\trianglelefteq}
\newcommand{\bl}{\mathrel{\mathpalette\blinn\relax}}
\newcommand{\blinn}[2]{%
  \ooalign{%
    \raisebox{.2ex}{$#1\blacktriangleleft$}\cr
    $#1\leq$\cr
  }%
}
\def\ci{\circ}
\def\ti{\tilde}
\def\ac{\acute}
\def\gr{\grave}
\def\d{{\rm d}}
\def\D{{\rm D}}
\def\md#1{\vert #1 \vert}
\def\ms#1{\vert #1 \vert^2}
\def\bmd#1{\big\vert #1 \big\vert}
\def\bms#1{\big\vert #1 \big\vert^2}
\def\an#1{\langle #1 \rangle}
\def\ban#1{\bigl\langle #1 \bigr\rangle}
\def\nm#1{\Vert #1 \Vert}
\def\bnm#1{\big\Vert #1 \big\Vert}
\title{Stratified manifolds with corners}
\author{Dominic Joyce}
\date{}
\maketitle

\begin{abstract}
We define categories of {\it stratified manifolds\/} ({\it s-manifolds\/}) and {\it stratified manifolds with corners\/} ({\it s-manifolds with corners\/}). An s-manifold $\bX$ of dimension $n$ is a Hausdorff, locally compact topological space $X$ with a stratification $X=\coprod_{i\in I}X^i$ into locally closed subsets $X^i$ which are smooth manifolds of dimension $\le n$, satisfying some conditions. 

S-manifolds can be very singular, but still share many good properties with ordinary manifolds, e.g.\ an oriented s-manifold $\bX$ has a fundamental class $[\bX]_\fund$ in Steenrod homology $H_n^\St(X,\Z)$, and transverse fibre products exist in the category of s-manifolds.

S-manifolds are designed for applications in Symplectic Geometry. In future work we hope to show that after suitable perturbations, the moduli spaces $\oM$ of $J$-holomorphic curves used to define Gromov--Witten invariants, Lagrangian Floer cohomology, Fukaya categories, and so on, can be made into s-manifolds or s-manifolds with corners, and their fundamental classes used to define Gromov--Witten invariants, Lagrangian Floer cohomology, \ldots.
\end{abstract}

\setcounter{tocdepth}{2}
\tableofcontents

\section{Introduction}
\label{sm1}

\baselineskip 11.98pt minus .1pt

This paper defines and studies two new classes of spaces: {\it stratified manifolds\/} (or {\it s-manifolds\/} for short), and {\it stratified manifolds with corners\/} (or {\it s-manifolds with corners\/} for short). They are generalizations of smooth manifolds (with corners), but allowing singularities in positive codimension.

An s-manifold $\bX$ of dimension $n$ is a locally compact, Hausdorff, second countable topological space $X$ equipped with a locally finite stratification $X=\coprod_{i\in I}X^i$ into locally closed strata $X^i\subset X$, which have the structure of a smooth manifold of dimension $\dim X^i\le n$, satisfying some conditions. These strata can be glued together in complicated ways, so that $X$ may be a pathological topological space. Figure \ref{sm1fig1} shows a compact, oriented s-manifold of dimension 1 given in Example \ref{sm3ex4}.

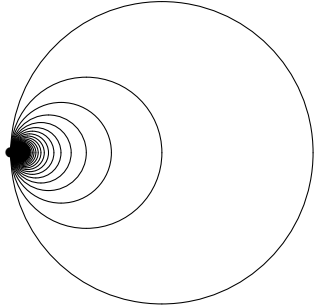
\begin{figure}[htb]
\centerline{$\splinetolerance{.8pt}
\begin{xy}
0;<1cm,0cm>:
,(4,2)*{}
,(-.5,-2)*{}
,(0,0)*{\bu}
,(1,0)*{\ellipse<2cm>{-}}
,(.5,0)*{\ellipse<1cm>{-}}
,(.333,0)*{\ellipse<.666cm>{-}}
,(.25,0)*{\ellipse<.5cm>{-}}
,(.2,0)*{\ellipse<.4cm>{-}}
,(.166,0)*{\ellipse<.333cm>{-}}
,(.1429,0)*{\ellipse<.2857cm>{-}}
,(.125,0)*{\ellipse<.25cm>{-}}
,(.111,0)*{\ellipse<.222cm>{-}}
,(.1,0)*{\ellipse<.2cm>{-}}
,(.0909,0)*{\ellipse<.1818cm>{-}}
,(.0833,0)*{\ellipse<.1666cm>{-}}
,(.0769,0)*{\ellipse<.1538cm>{-}}
,(.0714,0)*{\ellipse<.1428cm>{-}}
,(.0666,0)*{\ellipse<.1333cm>{-}}
,(.0625,0)*{\ellipse<.125cm>{-}}
,(.12,-.01)*{\hbox{\huge$\bullet$}}
\end{xy}$}
\caption{s-manifold $\bX$ of dimension 1 in Example \ref{sm3ex4}}
\label{sm1fig1}
\end{figure}

Note that although each stratum $X^i$ is a smooth manifold, there is {\it no global smooth structure\/} on $\bX$: the strata $X^i$ are glued together only as a topological space, with no extra smooth information on how they are related.

An oriented $n$-manifold $X$ has a {\it fundamental class\/} $[X]_\fund$ in the Borel--Moore homology group $H_n^\BM(X,\Z)$. If $X$ is compact, $H_n^\BM(X,\Z)=H_n(X,\Z)$ is ordinary homology. Similarly, an oriented s-manifold $\bX$ of dimension $n$ has a fundamental class $[\bX]_\fund$ in the {\it Steenrod\/} homology group $H_n^\St(X,\Z)$. 

Here {\it Steenrod homology\/} is a homology theory of locally compact Hausdorff topological spaces which has good properties under limits, and agrees with Borel--Moore homology for nice topological spaces, and with ordinary homology for nice, compact topological spaces. For $\bX$ as in Figure \ref{sm1fig1}, $[\bX]_\fund$ does not exist in ordinary homology $H_1(X,\Z)$, which is why we need Steenrod homology. 

Several areas of differentiable geometry of manifolds (with corners) extend nicely to s-manifolds (with corners), for example, orientations, partitions of unity, the Whitney Embedding Theorem, submersions,  transverse fibre products, and boundaries of (s-)manifolds with corners.

The author has designed s-manifolds (with corners) especially for applications in enumerative invariants, Floer homology theories, and so on, and intends to use them in a project \cite{CJR} with Guillem Cazassus and Alexander Ritter which gives a new construction of Fukaya categories of symplectic manifolds.

In enumerative invariant problems in symplectic geometry, gauge theory, \ldots, one forms a moduli space $\oM$ of geometric objects such as $J$-holomorphic curves, or connections satisfying an instanton-type curvature condition. Often $\oM$ must be compactified by adding singular objects, such as Deligne--Mumford stable $J$-holomorphic curves, or `bubble trees' of singular instantons. This gives $\oM$ a natural stratification by singularity type; under genericity conditions, each stratum may naturally be a smooth manifold.

To form enumerative invariants, we want $\oM$ to behave like a compact oriented manifold (with corners), in that it should have a fundamental class (or fundamental chain) in homology. Invariants (e.g.\ Gromov--Witten invariants of symplectic manifolds, or Donaldson invariants of 4-manifolds) are integrals of cohomology classes over these fundamental classes; Floer homology theories and Fukaya categories use fundamental chains in a more complicated way.

S-manifolds (with corners) $\bX$ have several features which make them very well suited to be the geometric structure on such moduli spaces $\oM$ used to define the fundamental class (or fundamental chain):
\begin{itemize}
\setlength{\itemsep}{0pt}
\setlength{\parsep}{0pt}
\item[(i)] There is {\it no global smooth structure\/} on $\bX$, just smooth manifold structures on each stratum $X^i$. The $X^i$ are only glued together topologically.
\item[(ii)] The conditions on how strata $X^i$ are glued together are {\it extremely weak}.
\item[(iii)] (Oriented) s-manifolds (with corners) have {\it transverse fibre products}.
\item[(iv)] An oriented s-manifold with corners $\bX$ has a {\it fundamental class\/} $[\bX]_\fund$ in Steenrod homology $H^\St_*(\bX,\Z)$. If $\bfX$ is an oriented s-manifold with boundary and $i:\pd\bfX\hookra\bfX$ the inclusion then $i_*([\pd\bfX]_\fund)=0$ in~$H^\St_*(\bfX,\Z)$.
\end{itemize}
Parts (i),(ii) mean that it should be {\it very much less work\/} to construct an s-manifold structure on a moduli space $\oM$ than competing structures such as Kuranishi spaces \cite{FOOO,FuOn} or polyfolds \cite{HWZ1,HWZ2}. Part (iii) is needed for relating fundamental chains of different moduli spaces in Floer theories, when boundaries of moduli spaces are disjoint unions of transverse fibre products of other moduli spaces. Part (iv) is essential for defining enumerative invariants.

We begin in \S\ref{sm2} with background material on (co)homology and Steenrod homology, and manifolds with corners. Sections \ref{sm3} and \ref{sm4} define and study s-manifolds, and s-manifolds with corners. Finally, \S\ref{sm5} discusses applications of s-manifolds in Symplectic Geometry.

The next remark compares our s-manifolds with other classes of spaces.

\begin{rem}
\label{sm1rem1}
Several areas of Geometry and Topology involve the study of a topological space $X$ (often with extra structure, e.g.\ $X$ could be a complex algebraic variety) with a locally finite stratification $X=\coprod_{i\in I}X^i$ into locally closed subspaces $X^i\subset X$, such that each $X^i$ is a topological or smooth manifold. It is common to require the closure $\bar X^i$ of $X^i$ in $X$ to be a union of strata $X^j$ (the `condition of the frontier'). We discuss some of these areas.
\smallskip

\noindent{\bf(a)} Probably the spaces in the literature most similar to s-manifolds are Kreck's {\it stratifolds\/} \cite{Krec}. A {\it stratifold\/} $X$ of {\it dimension\/} $m$ is a locally compact, second countable, Hausdorff `differential space' $X$ (so in particular, it has a {\it global smooth structure}, which s-manifolds do not) such that defining $X^j=\{x\in X:\dim T_xX=j\}$, then each $X^j$ is a smooth manifold of dimension $j$, and $X^j=\es$ for $j>m$, and smooth bump functions exist supported in arbitrarily small neighbourhoods of any $x\in X$. He also defines {\it stratifolds with boundary}.

Kreck is primarily interested in defining homology theories of topological spaces $Y$ as bordism groups of continuous maps $f:X\ra Y$ from compact $m$-dimensional stratifolds $X$. To do this he imposes the condition that $X$ is $\Z_2$-{\it oriented}, that is, $X^{m-1}=\es$, as without this condition, stratifolds would not have a well-behaved fundamental class in homology.

In stratifolds, the strata $X^j$ can be glued together in a pathological way, as for s-manifolds. A $\Z_2$-oriented $m$-dimensional stratifold $X$ (with boundary) can be made into an s-manifold (with boundary) by forgetting the global smooth structure, and remembering only the smooth structure on each~$X^j$.

Our oriented s-manifolds have a well-behaved fundamental class in (Steenrod) homology, without assuming that~$X^{m-1}=\es$.

One could define a homology theory using compact s-manifolds (with boundary), in the same way as for stratifolds in Kreck~\cite{Krec}.
\smallskip

\noindent{\bf(b)} There is also some similarity, both in structure and in the intended applications in Symplectic Geometry, between our s-manifolds and the {\it pseudocycles\/} of McDuff--Salamon \cite{McSa2} and Zinger \cite{Zing}, and the {\it pseudomanifolds\/} of Ruan--Tian \cite[Def.~5.1]{RuTi}. Here a {\it pseudocycle\/} in a manifold $Y$ is a smooth map $f:X\ra Y$ where $X$ is an oriented $n$-manifold, such that $\bigcap_{\text{$K\subseteq X$ compact}}\ov{f(X\sm K)}$ has dimension at most $n-2$ in $Y$. A {\it pseudomanifold\/} is roughly a topological space $X$ with an open nonsingular part $X^\ci\subseteq X$ which is an $n$-manifold, such that $X\sm X^\ci$ is a simplicial complex of dimension $\le n-2$. In both cases, we do not have a full stratification $X=\coprod_{i\in I}X^i$ into manifolds, but only a top-dimensional stratum $X^\ci$ which is a manifold, plus a lower-dimensional singular part.

Our s-manifolds have better categorical properties than pseudocycles and pseudomanifolds, for example, they have a good notion of transverse fibre product, which is helpful in applications. They also have a `corners' version useful for Lagrangian Floer cohomology and Fukaya categories.
\smallskip

\noindent{\bf(c)} Let $X$ be a quasiprojective variety over $\K=\R$ or $\C$, so that $X$ has an embedding $X\subset\mathbb{KP}^N$ for $N\ge 0$. Then $X$ has a natural stratification $X=\coprod_{j=0}^mX^j$, where $m=\dim_\K X$, and $X^j$ is a quasiprojective $\K$-manifold of dimension $j$, and the closure $\bar X^j$ of $X^j$ in $X$ has $\bar X^j\subseteq\coprod_{k=0}^jX^k$. We define $X^j$ by induction on $m,m-1,\ldots,0$, such that if we have defined $X^m,\ldots,X^k$ then $X\sm(\coprod_{j=k}^mX^j)$ is a quasiprojective variety of dimension $\le k-1$, and $X^{k-1}$ is the set of nonsingular points of $X\sm(\coprod_{j=k}^mX^j)$ of dimension~$k-1$.

It turns out that this stratification is often not the best for applications. A {\it Whitney stratification\/} is $X=\coprod_{i\in I}X^i$ as above such that any two strata $X^i,X^j$ satisfy the {\it Whitney conditions\/} \cite[\S 19]{Whit}, which concern the limits of tangent spaces $T_{x_a}X^i$, $T_{y_b}X^j$ of sequences $(x_a)_{a=1}^\iy$ in $X^i$ and $(y_b)_{b=1}^\iy$ in $X^j$. Whitney showed that any $\K$-algebraic or $\K$-analytic variety has a Whitney stratification. 

The Whitney conditions have nice consequences. For example, Thom \cite{Thom} and Mather \cite{Math} showed that $X$ is {\it equisingular\/} along each stratum. That is, if $x,y$ are points in the same connected component of $X^i$ then $x,y$ have homeomorphic open neighbourhoods in $X$, compatible with stratifications. 

Our notion of s-manifold is much weaker than Whitney stratification. They are not equisingular, so much more pathological behaviour is allowed. Also, the Whitney conditions require $X$ to be (locally) embedded in an ambient smooth manifold $Y$, so $X$ has a global smooth structure. In an s-manifold $X$ each stratum $X^i$ is a smooth manifold, but there is no global smooth structure.

For geometry and analysis of Whitney stratified spaces (and slightly more general spaces), see Pflaum \cite{Pfla}. Goresky and MacPherson \cite{GoMa} develop Morse Theory for Whitney stratified spaces.

\smallskip

\noindent{\bf(d)} {\it Thom--Mather stratified spaces\/} were introduced by Mather \cite[\S 8]{Math} to axiomatize some properties of Whitney stratified spaces. A Thom--Mather space $X$ has a stratification $X=\coprod_{i\in I}X^i$ as above with the $X^i$ smooth manifolds, together with `control data' specifying how the strata glue together, which roughly says that each stratum $X^i$ has a tubular neighbourhood which is locally a cone bundle over $X^i$, with link another Thom--Mather stratified space. Our notion of s-manifold is much weaker than Thom--Mather stratified spaces.

\smallskip

\noindent{\bf(e)} {\it Topological pseudomanifolds\/} are a class of topological spaces for which one can define {\it intersection\/} ({\it co\/}){\it homology}, as in Kirwan and Woolf \cite{KiWo}. We first define a {\it topologically stratified space\/} $X$ of {\it dimension\/} $m$, inductively on dimension, to be a paracompact Hausdorff topological space $X$ with a stratification $X=\coprod_{j=0}^mX^j$, such that $X^j$ is a topological manifold of dimension $j$, and near each point $x\in X^j$ there is a local stratification-preserving homeomorphism $X\cong\R^j\t C(L_{m-j-1})$, where $L_{m-j-1}$ is a topologically stratified space of dimension $m-j-1$. We call $X$ a {\it topological pseudomanifold\/} if $X^{m-1}=\es$ and $X^m$ is dense in $X$. The stratification must exist, but is not part of the data.

Complex algebraic varieties are important examples of pseudomanifolds.

Intersection (co)homology $IH_*(X),IH^*(X)$ behave a lot like (co)homology of manifolds (or of K\"ahler manifolds, in the complex case). For example, they satisfy a version of Poincar\'e duality, and for complex algebraic varieties, analogues of the Lefschetz Theorems and Hodge decomposition hold.

In an s-manifold, the strata $X^j$ are {\it smooth\/} manifolds, which is not required in pseudomanifolds. But the conditions on how strata are glued together are much weaker for s-manifolds than for pseudomanifolds.
\smallskip

\noindent{\bf(f)} Weinberger \cite{Wein} studies a class of stratified topological spaces, with the~aim~of extending classification theory for topological manifolds to the stratified context. His `PL stratified spaces' are similar to topologically stratified spaces~in~{\bf(e)}.
\end{rem}

\noindent {\it Acknowledgements.} The author learned about the special properties of Steenrod homology which make the fundamental classes in this paper possible from Dusa McDuff, whom he would like to thank, and \v Cech homology over $\Q$ (which agrees with Steenrod homology in this case) is used for a similar purpose in McDuff--Wehrheim \cite{McWe}. The author would like to thank Guillem Cazassus, Jason Lotay and Alexander Ritter for helpful conversations, and a referee. This research was supported by EPSRC grant EP/T012749/1. For the purpose of open access, the author has applied a CC BY public copyright licence to any Author Accepted Manuscript (AAM) version arising from this submission.

\section{Background material}
\label{sm2}

\subsection{Basic properties of homology and cohomology}
\label{sm21}

We begin by briefly reviewing homology and cohomology, and their variants Borel--Moore homology and compactly-supported cohomology. There are many excellent textbooks on (co)homology, e.g.\ Bott and Tu \cite{BoTu}, Bredon \cite{Bred1}, Eilenberg and Steenrod \cite{EiSt}, Hatcher \cite{Hatc}, MacLane \cite{MacL}, Massey \cite{Mass}, Maunder \cite{Maun}, Munkres \cite{Munk}, Spanier \cite{Span}, and tom Dieck \cite{toDi}. For Borel--Moore homology see \cite{BoMo,Dimc,EiSt,Hatc,HuRa,Iver,Mass,Petk,Sklj}. For compactly-supported cohomology, see \cite{BoTu,Hatc,Mass,Span}. Here is the definition of singular homology.

\begin{dfn}
\label{sm2def1}
For $k=0,1,\ldots$, the $k$-{\it simplex\/} $\De_k$ is
\e
\De_k=\bigl\{(x_0,\ldots,x_k)\in\R^{k+1}:x_i\ge 0,\;\>
x_0+\cdots+x_k=1\bigr\}.
\label{sm2eq1}
\e
It is a compact, oriented $k$-manifold with corners, as in \S\ref{sm24}.

Let $X$ be a topological space, and $R$ a commutative ring. Define
$C_k^\rsi(X,R)$ to be the $R$-module spanned by {\it singular
simplices}, which are continuous maps $\si:\De_k\ra X$. Elements of
$C_k^\rsi(X,R)$, which are called {\it singular chains}, are finite
sums $\sum_{a\in A}\rho_a\,\si_a$, where $A$ is a finite indexing
set, $\rho_a\in R$, and $\si_a:\De_k\ra X$ is continuous for~$a\in
A$.

The {\it boundary operator} $\pd:C_k^\rsi(X,R)\ra C_{k-1}^\rsi(X,R)$
is \cite[\S IV.1]{Bred1}:
\e
\pd:\ts\sum_{a\in A}\rho_a\,\si_a\longmapsto \ts\sum_{a\in
A}\sum_{j=0}^k(-1)^j\rho_a(\si_a\ci F_j^k),
\label{sm2eq2}
\e
where for $j=0,\ldots,k$ the map $F_j^k:\De_{k-1}\ra\De_k$ is given
by $F_j^k(x_0,\ldots,x_{k-1})=(x_0,\ldots,x_{j-1},0,x_j,\ldots,
x_{k-1})$. As a manifold with boundary, we have
$\pd\De_k=\coprod_{j=0}^k\De_j^k$, where $\De_j^k$ is the connected
component of $\pd\De_k$ on which $x_j\equiv 0$, so that
$F_j^k:\De_{k-1}\ra\De_j^k$ is a diffeomorphism. The orientation on
$\De_k$ induces one on $\pd\De_k$, and so induces orientations on
$\De_j^k$ for $j=0,\ldots,k$. It is easy to show that under
$F_j^k:\De_{k-1}\ra\De_j^k$, the orientations of $\De_{k-1}$ and
$\De_j^k$ differ by a factor $(-1)^j$, which is why this appears in
\eq{sm2eq2}. So $\pd$ in \eq{sm2eq2} basically restricts from
$\De_k$ to $\pd\De_k$, as oriented manifolds with boundary and
corners.

From \eq{sm2eq2} we find that $\pd\ci\pd=0$, since each codimension
2 face $\De_{k-2}$ of $\De_k$ contributes twice to $\pd\ci\pd$, once
with sign 1 and once with sign $-1$. Thus we may define the {\it
singular homology group}
\begin{equation*}
H_k^\rsi(X,R)=\frac{\Ker\bigl(\pd:C_k^\rsi(X,R)\ra
C_{k-1}^\rsi(X,R)\bigr)}{\Im\bigl(\pd:C_{k+1}^\rsi(X,R)\ra
C_k^\rsi(X,R)\bigr)}\,.
\end{equation*}

If $Y\subseteq X$ is a subspace we define {\it relative singular homology\/} $H_*^\rsi(X;Y,R)$ using the quotient chain groups $C_k^\rsi(X;Y,R)=C_k^\rsi(X,R)/C_k^\rsi(Y,R)$.

If $X$ is a smooth manifold, we can instead define {\it smooth singular homology\/}
$C_k^\ssi(X,R),\ab H_k^\ssi(X,R)$ using {\it smooth\/} maps
$\si:\De_k\ra X$, as in \cite[\S V.5]{Bred1}, which we call {\it
smooth singular simplices}, and this gives the same homology groups.
\end{dfn}

\begin{property}
\label{sm2pr1}
Let $X$ be a topological space, $Y\subseteq X$ a subspace, and $R$ a commutative ring. Then using some choice of homology or cohomology theory (e.g.\ singular homology \cite{Bred1,EiSt,Maun,Munk,Span}, singular cohomology \cite{EiSt,Munk,Span}, \v Cech cohomology \cite{EiSt,Munk}, Alexander--Spanier cohomology \cite{Span}, de Rham cohomology of manifolds \cite{BoTu,Bred1}, and sheaf cohomology \cite{Bred2,Iver}), for all $k=0,1,2,\ldots$ one can define $R$-modules as follows:
\begin{itemize}
\setlength{\itemsep}{0pt}
\setlength{\parsep}{0pt}
\item[(i)] The {\it relative homology group\/} $H_k(X;Y,R)$.
\item[(ii)] The {\it relative cohomology group\/} $H^k(X;Y,R)$.
\end{itemize}
If $Y=\es$ we write these as $H_k(X,R),H^k(X,R)$, and call them ({\it co\/}){\it homology groups}, rather than relative (co)homology groups.

If $X$ is a locally compact topological space we also define:
\begin{itemize}
\setlength{\itemsep}{0pt}
\setlength{\parsep}{0pt}
\item[(iii)] The {\it Borel--Moore homology group\/} $H_k^\BM(X,R)$. (These have several different names in the literature: they are called {\it Borel--Moore homology groups\/} in \cite{BoMo,Iver}, {\it locally finite homology\/} in \cite{HuRa}, {\it homology with closed supports\/} in \cite{Dimc}, and {\it homology of the second kind\/} in \cite{Petk,Sklj}.)
\item[(iv)] The {\it compactly-supported cohomology group\/} $H^k_\cs(X,R)$.
\end{itemize}

They have the following properties:
\begin{itemize}
\setlength{\itemsep}{0pt}
\setlength{\parsep}{0pt}
\item[(a)] For locally compact $X$ there are natural $R$-module morphisms
\e
H_k(X,R)\longra H_k^\BM(X,R),\quad  H^k_\cs(X,R)\longra H^k(X,R).
\label{sm2eq3}
\e
If $X$ is compact then \eq{sm2eq3} are isomorphisms, so that ordinary homology and Borel--Moore homology agree, and ordinary and compactly-supported cohomology agree.

If $X$ is locally compact, noncompact, and particularly nice (`forward tame' in the sense of Hughes--Ranicki \cite[\S 7]{HuRa}) then we have isomorphisms
\e
H_k^\BM(X,R)\cong H_k(\ti X;\{\iy\},R),\quad H^k_\cs(X,R)\cong H^k(\ti X;\{\iy\},R),
\label{sm2eq4}
\e
where $\ti X=X\cup\{\iy\}$ is the one-point compactification of $X$. However, \eq{sm2eq4} fails for general noncompact $X$. For example, if $X=\N$ with the discrete topology (which is not forward tame) then
\end{itemize}
\begin{equation*}
H_0^\BM(X,R)\cong H^0(\ti X;\{\iy\},R)\cong R[[x]],\;\> H^0_\cs(X,R)\cong H_0(\ti X;\{\iy\},R)\cong R[x].
\end{equation*}
\begin{itemize}
\setlength{\itemsep}{0pt}
\setlength{\parsep}{0pt}
\item[(b)] Suppose $f:X_1\ra X_2$ is continuous. Then there are covariantly functorial {\it pushforwards\/} $f_*:H_k(X_1,R)\ra H_k(X_2,R)$, and contravariantly functorial {\it pullbacks\/} $f^*:H^k(X_2,R)\ra H^k(X_1,R)$.

If $X_1,X_2$ are locally compact and $f:X_1\ra X_2$ is proper then pushforwards $f_*$ are also defined on Borel--Moore homology, and pullbacks $f^*$ on compactly-supported cohomology.

If $f:X_1\hookra X_2$ is an open inclusion then there are contravariantly functorial {\it pullbacks\/} $f^*:H_k^\BM(X_2,R)\ra H_k^\BM(X_1,R)$, and covariantly functorial {\it pushforwards\/} $f_*:H^k_\cs(X_1,R)\ra H^k_\cs(X_2,R)$.
\item[(c)] There are natural $R$-bilinear dual pairings
\begin{equation*}
H_k(X,R)\t H^k(X,R)\longra R, \quad H_k^\BM(X,R)\t H^k_\cs(X,R)\longra R.
\end{equation*}
If $R$ is a field $\bF$ these are perfect pairings, and induce isomorphisms
\begin{equation*}
H^k(X,\bF)\cong H_k(X,\bF)^*, \quad H_k^\BM(X,\bF)\cong H^k_\cs(X,\bF)^*.
\end{equation*}
\item[(d)] There are $R$-bilinear {\it cup\/} and {\it cap products\/}
\begin{align*}
\cup&:H^k(X,R)\t H^l(X,R)\longra H^{k+l}(X,R),\\
\cup&:H^k(X,R)\t H^l_\cs(X,R)\longra H^{k+l}_\cs(X,R),\\
\cup&:H^k_\cs(X,R)\t H^l_\cs(X,R)\longra H^{k+l}_\cs(X,R),\\
\cap&:H_k(X,R)\t H^l(X,R)\longra H_{k-l}(X,R), \\
\cap&:H_k^\BM(X,R)\t H^l(X,R)\longra H_{k-l}^\BM(X,R),\\
\cap&:H_k^\BM(X,R)\t H^l_\cs(X,R)\longra H_{k-l}(X,R).
\end{align*}
There is a natural {\it identity\/} $1_X\in H^0(X,R)$. Then $\cup,1_X$ make $H^*(X,R)$ into a supercommutative graded $R$-algebra, and $\cap$ makes $H_*(X,R)$ and $H_*^\BM(X,R)$ into graded modules over $H^*(X,R)$. There is an identity $1_X$ in $H^0_\cs(X,R)$ only if $X$ is compact.
\item[(e)] Suppose $X$ is an oriented manifold of dimension $n$. Then there is a canonical {\it fundamental class\/} $[X]_\fund\in H_n^\BM(X,R)$. Cap product with $[X]_\fund$ induces {\it Poincar\'e duality isomorphisms\/}
\begin{align*}{}
[X]_\fund\cap-&:H^l(X,R)\longra H_{n-l}^\BM(X,R),\\
[X]_\fund\cap-&:H^l_\cs(X,R)\longra H_{n-l}(X,R).
\end{align*}
If also $X$ is compact then $[X]_\fund$ is defined in $H_n(X,R)$, and Poincar\'e duality identifies $H^l(X,R)\cong H_{n-l}(X,R)$.
\item[(f)] Let $U,V\subseteq X$ be open with $X=U\cup V$, and write $\inc_U^X:U\hookra X$ for the inclusion, and so on. Then there is a {\it Mayer--Vietoris exact sequence}:
\e
\begin{gathered}
\!\!\!\!\!\!\!\!\!\!\!\xymatrix@C=20pt@R=11pt{
\cdots \ar[r] & H_k(U\cap V,R) \ar[rrrrr]_(0.35){(\inc_{U\cap V}^U)_*\op(\inc_{U\cap V}^V)_*} &&&&& *+[l]{H_k(U,R)\op H_k(V,R)} \ar@<-3ex>[d]_{(\inc_U^X)_*\op -(\inc_V^X)_*} \\
\cdots & H_{k-1}(U\cap V,R) \ar[l] &&&&& *+[l]{H_k(X,R).} \ar[lllll]_(0.65)\pd    }
\end{gathered}\!\!\!\!\!\!\!\!
\label{sm2eq5}
\e
\item[(g)] For $Y\subseteq X$ there are long exact sequences of $R$-modules
\end{itemize}
\e
\begin{aligned}
\xymatrix@C=10pt@R=15pt{
\cdots \ar[r] & H_k(Y,R) \ar[r]^{\inc_*} & H_k(X,R) \ar[r] & H_k(X;Y,R) \ar[r]^\pd & H_{k-1}(Y,R) \ar[r] & \cdots, \\
\cdots \ar[r] & H^{k-1}(Y,R) \ar[r]^\pd & H^k(X;Y,R) \ar[r] & H^k(X,R) \ar[r]^{\inc^*} & H^k(Y,R) \ar[r] & \cdots. }
\end{aligned}\!\!\!\!\!\!\!\!
\label{sm2eq6}
\e
\begin{itemize}
\setlength{\itemsep}{0pt}
\setlength{\parsep}{0pt}
\item[(h)] Suppose $Z\subseteq Y\subseteq X$ are such that the closure of $Z$ is contained in the interior of $Y$. Then there are canonical {\it excision isomorphisms\/}
\end{itemize}
\begin{equation*}
H_k(X;Y,R)\cong H_k(X\sm Z;Y\sm Z,R),\;\> H^k(X;Y,R)\cong H^k(X\sm Z;Y\sm Z,R).
\end{equation*}

Now let $\pi:P\ra X$ be a principal $\Z_2$-bundle. Then we can also consider homology and cohomology {\it twisted by\/} $P$. We write the corresponding twisted groups as the {\it relative twisted homology group\/} $H_k(X;Y,P,R)$, the {\it twisted homology group\/} $H_k(X,P,R)$, the {\it relative twisted cohomology group\/} $H^k(X;Y,P,R)$, the {\it twisted Borel--Moore homology group\/} $H_k^\BM(X,P,R)$, and the {\it twisted comp\-act\-ly-supported cohomology group\/} $H^k_\cs(X,P,R)$. To avoid confusion between relative homology $H_k(X;Y,R)$ and twisted homology $H_k(X,P,R)$, we write relative groups with a semicolon $X;Y$, and twisted groups with a comma~$X,P$.

Most of (a)--(h) above have easy extensions to twisted (co)homology. Also:
\begin{itemize}
\setlength{\itemsep}{0pt}
\setlength{\parsep}{0pt}
\item[(j)] Suppose $X$ is a manifold of dimension $n$, which need not be oriented or orientable, and write $\pi:\Or_X\ra X$ for the principal $\Z_2$-bundle of orientations on $X$. Then part (e) above generalizes as follows: there is a canonical {\it fundamental class\/} $[X]_\fund$ in the twisted homology group $H_n^\BM(X,\Or_X,R)$. Cap product with $[X]_\fund$ induces {\it Poincar\'e duality isomorphisms\/}
\begin{align*}{}
[X]_\fund\cap-&:H^l(X,R)\longra H_{n-l}^\BM(X,\Or_X,R),\\
[X]_\fund\cap-&:H^l(X,\Or_X,R)\longra H_{n-l}^\BM(X,R),\\
[X]_\fund\cap-&:H^l_\cs(X,R)\longra H_{n-l}(X,\Or_X,R),\\
[X]_\fund\cap-&:H^l_\cs(X,\Or_X,R)\longra H_{n-l}(X,R).
\end{align*}
If $X$ is compact then $[X]_\fund\in H_n(X,\Or_X,R)$, and Poincar\'e duality identifies $H^l(X,R)\cong H_{n-l}(X,\Or_X,R)$ and~$H^l(X,\Or_X,R)\cong H_{n-l}(X,R)$.
\end{itemize}
\end{property}

\subsection{Steenrod homology}
\label{sm22}

Homology and cohomology theories are often explained using the {\it Eilenberg--Steenrod axioms\/} \cite{EiSt}. Most (co)homology theories worthy of the name satisfy these axioms. The Eilenberg--Steenrod axioms determine (co)homology groups $H_*(X;Y,R),$ $H_*(X,R),$ $H^*(X;Y,R),$ $H^*(X,R)$ uniquely when $X,Y$ are nice topological spaces (e.g.\ finite CW complexes, or compact polyhedra, or compact manifolds). So for many applications, it does not matter which (co)homology theory you choose.

However, for more pathological topological spaces $X$, which include the topological spaces of our s-manifolds below, different (co)homology theories can yield different answers. To define the fundamental class $[\bX]_\fund$ of an s-manifold $\bX$, we will need to work in {\it Steenrod homology}, and not for example in singular homology, which does not have all the properties we need.

Steenrod homology was introduced by Steenrod \cite{Stee1,Stee2}, as homology groups $H^\St_k(X;Y,R)$ for $X$ a compact metric space, $Y\subseteq X$ a compact subspace, and $R$ a commutative ring. They depend only on the underlying topological spaces, so we could instead take $X,Y$ to be metrizable topological spaces. 

Milnor \cite{Miln} proved that for such pairs $(X,Y)$, Steenrod homology can be characterized uniquely by the seven Eilenberg--Steenrod axioms, plus two others. If $X$ is any metrizable topological space then we define
\begin{equation*}
H^\St_k(X,R)=H^\St_k(\ti X;\{\iy\},R),	
\end{equation*}
where $\ti X=X\amalg\{\iy\}$ is the one-point compactification of $X$. Then one of Milnor's extra axioms \cite[Ax.~$8^{\rm LF}$]{Miln} says that for any compact metrizable $Y\subseteq X$, we have canonical isomorphisms
\e
H^\St_k(X;Y,R)\cong H^\St_k(X\sm Y,R).	
\label{sm2eq7}
\e
This is a strengthened form of the excision axiom, Property \ref{sm2pr1}(h). It means that in Steenrod homology, there is no need to use relative homology $H^\St_k(X;Y,R)$, as we can write everything in terms of absolute homology $H_k^\St(X,R)$ --- it is what Eilenberg and Steenrod call a `single space homology theory'.

For many commutative rings $R$ including $R=\Q$, but not for $R=\Z$, {\it \v Cech homology\/} $\check H_k(X,R)$, as defined in Eilenberg--Steenrod \cite[\S IX--\S X]{EiSt}, satisfies all of Milnor's axioms, and so agrees with Steenrod homology on compact metrizable pairs $(X,Y)$. (For $R=\Z$, \v Cech homology does not satisfy Eilenberg--Steenrod's exactness axiom \cite[\S X.4]{EiSt}.) Massey \cite[Preface]{Mass} argues that the traditional definition of \v Cech homology is unsatisfactory, and should be corrected. The resulting theory, which he calls {\it \v Cech--Alexander--Spanier homology} (but we will call {\it Steenrod homology\/}), then satisfies Milnor's axioms for all~$R$.

Here are some properties of Steenrod homology:

\begin{property}
\label{sm2pr2}
Let $X$ be a locally compact Hausdorff topological space, and $R$ a commutative ring. Then for all $k=0,1,2,\ldots$ one can define $R$-modules $H_k^\St(X,R)$ called the {\it Steenrod homology groups}. To be definite, we take these to be as in Massey \cite[\S 4]{Mass}. But for our applications, as if $X$ is an s-manifold then the 1-point compactification $\ti X$ is a metrizable topological space, any homology theory satisfying Milnor's axioms \cite{Miln} will give the same groups~$H_k^\St(X,R)$.

These satisfy the following analogues of Property \ref{sm2pr1}(a)--(g) (see~\cite[\S 4.3]{Mass}):
\begin{itemize}
\setlength{\itemsep}{0pt}
\setlength{\parsep}{0pt}
\item[(a)] If $X$ is a sufficiently nice topological space, for example, a manifold or locally finite CW complex, then $H_k^\St(X,R)$ agrees with Borel--Moore homology $H_k^\BM(X,R)$ from Property \ref{sm2pr1}(iii) (e.g.\ defined using locally finite singular chains as in \cite[\S 3]{HuRa}). If $X$ is also compact then $H_k^\St(X,R)$ agrees with ordinary homology $H_k(X,R)$ (e.g.\ singular homology).
\item[(b)] Suppose $f:X_1\ra X_2$ is continuous and proper. Then there are covariantly functorial {\it pushforwards\/} $f_*:H_k^\St(X_1,R)\ra H_k^\St(X_2,R)$. If $f:X_1\hookra X_2$ is an open inclusion then there are contravariantly functorial {\it pullbacks\/} $f^*:H_k^\St(X_2,R)\ra H_k^\St(X_1,R)$.
\item[(c)] The cohomology theory naturally dual to Steenrod homology is {\it \v Cech cohomology\/} $\check H^k(X,R)$, or {\it compactly-supported \v Cech cohomology\/} $\check H^k_\cs(X,R)$. There is a natural $R$-bilinear dual pairing
\begin{equation*}
H_k^\St(X,R)\t \check H^k_\cs(X,R)\longra R.
\end{equation*}
If $R$ is a field $\bF$ this is a perfect pairing, and induce isomorphisms
\begin{equation*}
H_k^\St(X,\bF)\cong\check H^k_\cs(X,\bF)^*.
\end{equation*}
\item[(d)] There are $R$-bilinear {\it cap products\/}
\begin{align*}
\cap&:H_k^\St(X,R)\t\check H^l(X,R)\longra H_{k-l}^\St(X,R),\\
\cap&:H_k^\St(X,R)\t\check H^l_\cs(X,R)\longra H_{k-l}^\St(X,R).
\end{align*}
These make $H_*^\St(X,R)$ a graded module over $\check H^*(X,R)$ and $\check H^*_\cs(X,R)$.
\item[(e)] Let $X$ be an oriented manifold of dimension $n$. Then there is a canonical {\it fundamental class\/} $[X]_\fund\in H_n^\St(X,R)$. Cap product with $[X]_\fund$ induces {\it Poincar\'e duality isomorphisms\/}
\begin{equation*}
[X]_\fund\cap-:\check H^l(X,R)\longra H_{n-l}^\St(X,R).
\end{equation*}
\item[(f)] Let $U,V\subseteq X$ be open with $X=U\cup V$, and write $\inc_U^X:U\hookra X$ for the inclusion, and so on. Then there is a {\it Mayer--Vietoris exact sequence}:
\begin{equation*}
\xymatrix@C=20pt@R=11pt{
\cdots \ar[r] & H_{k+1}(U\cap V,R) \ar[rrrrr]_(0.37)\pd &&&&& *+[l]{H_k(X,R)} \ar@<-3ex>[d]_{(\inc_U^X)^*\op -(\inc_V^X)^*} \\
\cdots & H_k(U\cap V,R) \ar[l] &&&&& *+[l]{H_k(U,R)\op H_k(V,R).} \ar[lllll]_(0.62){(\inc_{U\cap V}^U)^*\op(\inc_{U\cap V}^V)^*}    }
\end{equation*}
Note that this goes the {\it opposite way\/} to \eq{sm2eq5}, as Steenrod homology is contravariantly functorial under open inclusions.

\item[(g)] Let $X$ be a locally compact Hausdorff space and $Y\subseteq X$ a closed subspace, and write $i:Y\hookra X$, $j:X\sm Y\hookra X$ for the inclusions. Then there is a long exact sequence of $R$-modules
\end{itemize}
\e
\!\!\!\!\xymatrix@C=10pt@R=15pt{
\cdots \ar[r] & H_k^\St(Y,R) \ar[r]^{i_*} & H_k^\St(X,R) \ar[r]^{j^*} & H_k^\St(X\sm Y,R) \ar[r]^(0.55)\pd & H_{k-1}^\St(Y,R) \ar[r] & \cdots, }\!\!\!\!
\label{sm2eq8}
\e
\begin{itemize}
\setlength{\itemsep}{0pt}
\setlength{\parsep}{0pt}
\item[] where $i_*,j^*$ are as in (b), and $\pd$ is a natural morphism.
\end{itemize}
Here \eq{sm2eq8} is analogous to the top line of \eq{sm2eq6}, but involves $H_k^\St(X\sm Y,R)$ rather than relative homology $H_k(X;Y,R)$. Because of \eq{sm2eq7} we need no analogue of excision, Property~\ref{sm2pr1}(h).

Steenrod homology also satisfies a {\it continuity property\/} under inverse limits of topological spaces. Here is one form of this, which follows from~\cite[Th.~4]{Miln}:
\begin{itemize}
\setlength{\itemsep}{0pt}
\setlength{\parsep}{0pt}
\item[(h)] Let $Y$ be a metrizable topological space.
\begin{itemize}
\setlength{\itemsep}{0pt}
\setlength{\parsep}{0pt}
\item[(i)] Suppose $Y\supseteq X_1\supseteq X_2\supseteq\cdots$ is a decreasing family of closed subsets, and set $X=\bigcap_{i=1}^\iy X_i$. Then for each $k=0,1,\ldots$ there is an exact sequence
\e
\!\!\!\!\!\!\!\!\!\!\!\!\!\!\!\!\!\!\xymatrix@C=15pt{ 0 \ar[r] & \varprojlim^1 H_{k+1}^\St(X_i,R) \ar[r] & H_k^\St(X,R) \ar[r] & 
\varprojlim H_k^\St(X_i,R) \ar[r] & 0, }
\label{sm2eq9}
\e
where $\varprojlim^1$ is the first derived functor of the inverse limit, and the inverse limits are formed using pushforwards $\inc_*:H_k^\St(X_{i+1},R)\ra H_k^\St(X_i,R)$ along the proper inclusion maps~$\inc:X_{i+1}\hookra X_i$.
\item[(ii)] Suppose $X_1\subseteq X_2\subseteq\cdots\subseteq Y$ is an increasing family of open subsets, and set $X=\bigcup_{i=1}^\iy X_i$. Then for each $k=0,1,\ldots$ there is an exact sequence \eq{sm2eq9}, where the inverse limits are formed using pullbacks $\inc^*:H_k^\St(X_{i+1},R)\ra H_k^\St(X_i,R)$ along the open inclusion maps~$\inc:X_i\hookra X_{i+1}$.
\end{itemize}
\end{itemize}

Now let $\pi:P\ra X$ be a principal $\Z_2$-bundle. Then we can also consider {\it Steenrod homology twisted by\/} $P$, written $H_k^\St(X,P,R)$. Most of (a)--(h) above have easy extensions to twisted Steenrod homology. Also:
\begin{itemize}
\setlength{\itemsep}{0pt}
\setlength{\parsep}{0pt}
\item[(j)] Suppose $X$ is a manifold of dimension $n$, which need not be oriented or orientable, and write $\pi:\Or_X\ra X$ for the principal $\Z_2$-bundle of orientations on $X$. Then by Property \ref{sm2pr1}(j) and (a) above, there is a canonical {\it fundamental class\/} $[X]_\fund$ in the twisted homology group $H_n^\St(X,\Or_X,R)$. Cap product with $[X]_\fund$ induces {\it Poincar\'e duality isomorphisms\/}
\begin{align*}{}
[X]_\fund\cap-&:\check H^l(X,R)\longra H_{n-l}^\St(X,\Or_X,R),\\
[X]_\fund\cap-&:\check H^l(X,\Or_X,R)\longra H_{n-l}^\St(X,R).
\end{align*}
\end{itemize}
\end{property}

\subsection{Transverse fibre products of manifolds}
\label{sm23}

Here is a definition from category theory.

\begin{dfn}
\label{sm2def2}
Let $\cC$ be a category, and $g:X\ra Z$, $h:Y\ra Z$ be morphisms in $\cC$. A {\it fibre product\/} of $g,h$ in $\cC$ is an object $W$ and morphisms $e:W\ra X$ and $f:W\ra Y$ in $\cC$, such that $g\ci e=h\ci f$, with the universal property that if $e':W'\ra X$ and $f':W'\ra Y$ are morphisms in $\cC$ with $g\ci e'=h\ci f'$ then there is a unique morphism $b:W'\ra W$ with $e'=e\ci b$ and $f'=f\ci b$. Then we write $W=X\t_{g,Z,h}Y$ or $W=X\t_ZY$. The diagram
\begin{equation*}
\xymatrix@R=13pt@C=50pt{ W \ar[r]_f \ar[d]^e & Y
\ar[d]_h \\ X \ar[r]^g & Z}
\end{equation*}
is called a {\it Cartesian square}. Fibre products need not exist, but if they do exist they are unique up to canonical isomorphism in~$\cC$.
\end{dfn}

The next definition and theorem are well known, see e.g.\ Lee~\cite[Cor.~2.49]{Lee1}.

\begin{dfn}
\label{sm2def3}
Let $g:X\ra Z$ and $h:Y\ra Z$ be smooth maps of manifolds without boundary. We call $g,h$ {\it transverse\/} if $T_xg\op T_yh:T_xX\op T_yY\ra T_zZ$ is surjective for all $x\in X$ and $y\in Y$ with $g(x)=h(y)=z\in Z$. We call $h:Y\ra Z$ a {\it submersion\/} if $T_yh:T_yY\ra T_zZ$ is surjective for all $y\in Y$ with $h(y)=z\in Z$. If $h$ is a submersion then $g,h$ are transverse for any~$g:X\ra Z$.
\end{dfn}

\begin{thm}
\label{sm2thm1}
Suppose $g:X\ra Z$ and\/ $h:Y\ra Z$ are transverse smooth maps of manifolds without boundary. Then a fibre product $W=X\t_{g,Z,h}Y$ exists in the category of manifolds $\Man,$ with\/ $\dim W=\dim X+\dim Y-\dim Z$. We may write
\begin{equation*}
W=\bigl\{(x,y)\in X\t Y:\text{$g(x)=h(y)$ in $Z$}\bigr\}
\end{equation*}
as an embedded submanifold of\/ $X\t Y,$ where $e:W\ra X$ and\/ $f:W\ra Y$ act by $e:(x,y)\mapsto x$ and\/~$f:(x,y)\mapsto y$.
\end{thm}

\subsection{Manifolds with corners}
\label{sm24}

We define the category $\Manc$ of {\it manifolds with corners}, spaces locally modelled on $\R^n_k=[0,\iy)^k\t\R^{n-k}$ for $0\le k\le n$. Some references are Melrose \cite{Melr2,Melr3} and the author \cite{Joyc1,Joyc3}. 

\begin{dfn}
\label{sm2def4}
Use the notation $\R^m_k=[0,\iy)^k\t\R^{m-k}$
for $0\le k\le m$, and write points of $\R^m_k$ as $u=(x_1,\ldots,x_m)$ for $x_1,\ldots,x_k\in[0,\iy)$, $x_{k+1},\ldots,x_m\in\R$. Let $U\subseteq\R^m_k$ and $V\subseteq \R^n_l$ be open, and $f=(f_1,\ldots,f_n):U\ra V$ be a continuous map, so that $f_j=f_j(x_1,\ldots,x_m)$ maps $U\ra[0,\iy)$ for $j=1,\ldots,l$ and $U\ra\R$ for $j=l+1,\ldots,n$. Then we say:
\begin{itemize}
\setlength{\itemsep}{0pt}
\setlength{\parsep}{0pt}
\item[(a)] $f$ is {\it weakly smooth\/} if all derivatives $\frac{\pd^{a_1+\cdots+a_m}}{\pd x_1^{a_1}\cdots\pd x_m^{a_m}}f_j(x_1,\ldots,x_m):U\ra\R$ exist and are continuous for all $j=1,\ldots,n$ and $a_1,\ldots,a_m\ge 0$, including one-sided derivatives where $x_i=0$ for $i=1,\ldots,k$.
\item[(b)] $f$ is {\it smooth\/} if it is weakly smooth and every $u=(x_1,\ldots,x_m)\in U$ has an open neighbourhood $\ti U$ in $U$ such that for each $j=1,\ldots,l$, either:
\begin{itemize}
\setlength{\itemsep}{0pt}
\setlength{\parsep}{0pt}
\item[(i)] we may uniquely write $f_j(\ti x_1,\ldots,\ti x_m)=F_j(\ti x_1,\ldots,\ti x_m)\cdot\ti x_1^{a_{1,j}}\cdots\ti x_k^{a_{k,j}}$ for all $(\ti x_1,\ldots,\ti x_m)\in\ti U$, where $F_j:\ti U\ra(0,\iy)$ is weakly smooth and $a_{1,j},\ldots,a_{k,j}\in\N=\{0,1,2,\ldots\}$, with $a_{i,j}=0$ if $x_i\ne 0$; or 
\item[(ii)] $f_j\vert_{\smash{\ti U}}=0$.
\end{itemize}
\item[(c)] $f$ is {\it interior\/} if it is smooth, and case (b)(ii) does not occur.
\item[(d)] $f$ is {\it b-normal\/} if it is interior, and in case (b)(i), for each $i=1,\ldots,k$ we have $a_{i,j}>0$ for at most one $j=1,\ldots,l$.
\item[(e)] $f$ is {\it strongly smooth\/} if it is smooth, and in case (b)(i), for each $j=1,\ldots,l$ we have $a_{i,j}=1$ for at most one $i=1,\ldots,k$, and $a_{i,j}=0$ otherwise. 
\item[(f)] $f$ is a {\it diffeomorphism\/} if it is a smooth bijection with smooth inverse.
\end{itemize}
All the classes (a)--(f) include identities and are closed under composition.
\end{dfn}

\begin{dfn}
\label{sm2def5}
Let $X$ be a second countable Hausdorff topological space. An {\it $m$-dimensional chart on\/} $X$ is a pair $(U,\phi)$, where
$U\subseteq\R^m_k$ is open for some $0\le k\le m$, and $\phi:U\ra X$ is a
homeomorphism with an open set~$\phi(U)\subseteq X$.

Let $(U,\phi),(V,\psi)$ be $m$-dimensional charts on $X$. We call
$(U,\phi)$ and $(V,\psi)$ {\it compatible\/} if
$\psi^{-1}\ci\phi:\phi^{-1}\bigl(\phi(U)\cap\psi(V)\bigr)\ra
\psi^{-1}\bigl(\phi(U)\cap\psi(V)\bigr)$ is a diffeomorphism between open subsets of $\R^m_k,\R^m_l$, in the sense of Definition~\ref{sm2def4}(f).

An $m$-{\it dimensional atlas\/} for $X$ is a system
$\{(U_a,\phi_a):a\in A\}$ of pairwise compatible $m$-dimensional
charts on $X$ with $X=\bigcup_{a\in A}\phi_a(U_a)$. We call such an
atlas {\it maximal\/} if it is not a proper subset of any other
atlas. Any atlas $\{(U_a,\phi_a):a\in A\}$ is contained in a unique
maximal atlas, the set of all charts $(U,\phi)$ of this type on $X$
which are compatible with $(U_a,\phi_a)$ for all~$a\in A$.

An $m$-{\it dimensional manifold with corners\/} is a second
countable Hausdorff topological space $X$ equipped with a maximal
$m$-dimensional atlas. Usually we refer to $X$ as the manifold,
leaving the atlas implicit, and by a {\it chart\/ $(U,\phi)$ on\/}
$X$, we mean an element of the maximal atlas.

Now let $X,Y$ be manifolds with corners of dimensions $m,n$, and $f:X\ra Y$ a continuous map. We call $f$ {\it weakly smooth}, or {\it smooth}, or {\it interior}, or {\it b-normal}, or {\it strongly smooth}, if whenever $(U,\phi),(V,\psi)$ are charts on $X,Y$ with $U\subseteq\R^m_k$, $V\subseteq\R^n_l$ open, then
\begin{equation*}
\psi^{-1}\ci f\ci\phi:(f\ci\phi)^{-1}(\psi(V))\longra V
\end{equation*}
is weakly smooth, or smooth, or interior, or b-normal, or strongly smooth, respectively, as maps between open subsets of $\R^m_k,\R^n_l$ in the sense of Definition~\ref{sm2def4}.

We write $\Manc$ for the category with objects manifolds with corners $X,Y,$ and morphisms smooth maps $f:X\ra Y$ in the sense above. We also write $\Mancbn\subset\Mancin\subset\Manc$ for the subcategories with the same objects and with morphisms b-normal, or interior, maps.

Write $\cManc$ for the category whose objects are disjoint unions $\coprod_{m=0}^\iy X_m$, where $X_m$ is a manifold with corners of dimension $m$, allowing $X_m=\es$, and whose morphisms are continuous maps $f:\coprod_{m=0}^\iy X_m\ra\coprod_{n=0}^\iy Y_n$, such that $f\vert_{X_m\cap f^{-1}(Y_n)}:X_m\cap f^{-1}(Y_n)\ra Y_n$ is a smooth map of manifolds with corners for all $m,n\ge 0$. Objects of $\cManc$ will be called {\it manifolds with corners of mixed dimension}. Write $\cMancbn\subset\cMancin\subset\cManc$ for the subcategories with morphisms b-normal, or interior, maps.
\end{dfn}

\begin{rem}
\label{sm2rem1}
{\bf(a)} There are several non-equivalent definitions of categories of manifolds with corners. Just as objects, without considering morphisms, most authors define manifolds with corners as in Definition \ref{sm2def5}. However, Melrose \cite{Melr1,Melr2,Melr3} imposes an extra condition: in \S\ref{sm25} we will define the boundary $\pd X$ of a manifold with corners $X$, with an immersion $i_X:\pd X\ra X$. Melrose requires that $i_X\vert_C:C\ra X$ should be injective for each connected component $C$ of $\pd X$ (such $X$ are sometimes called {\it manifolds with faces\/}).

There is no general agreement in the literature on how to define smooth maps, or morphisms, of manifolds with corners: 
\begin{itemize}
\setlength{\itemsep}{0pt}
\setlength{\parsep}{0pt}
\item[(i)] Our smooth maps are due to Melrose \cite[\S 1.12]{Melr3}, \cite[\S 1]{Melr1}, who calls them {\it b-maps}. Interior and b-normal maps are also due to Melrose.
\item[(ii)] The author \cite{Joyc1} defined and studied strongly smooth maps above (which were just called `smooth maps' in \cite{Joyc1}). 
\item[(iii)] Monthubert's {\it morphisms of manifolds with corners\/} \cite[Def.~2.8]{Mont} coincide with our strongly smooth b-normal maps. 
\item[(iv)] Most other authors, such as Cerf \cite[\S I.1.2]{Cerf}, define smooth maps of manifolds with corners to be weakly smooth maps, in our notation.
\end{itemize}

\noindent{\bf(b)} For generalizations of manifolds with corners, see the author's {\it manifolds with generalized corners\/} \cite{Joyc3} and {\it manifolds with analytic corners\/} \cite{Joyc4}, as in~\S\ref{sm27}.

\smallskip

\noindent{\bf(c)} For theorems on existence of fibre products in $\Manc$ under transversality conditions, generalizing Theorem \ref{sm2thm1}, see \cite[\S 6]{Joyc1} and~\cite[\S 4.3]{Joyc3}.
\end{rem}

\subsection{Boundaries and corners of manifolds with corners}
\label{sm25}

The material of this section broadly follows the author \cite{Joyc1,Joyc3}.

\begin{dfn}
\label{sm2def6}
Let $U\subseteq\R^m_k$ be open. For each
$u=(x_1,\ldots,x_m)$ in $U$, define the {\it depth\/} $\depth_Uu$ of
$u$ in $U$ to be the number of $x_1,\ldots,x_k$ which are zero. That
is, $\depth_Uu$ is the number of boundary faces of $U$
containing~$u$.

Let $X$ be an $m$-manifold with corners. For $x\in X$, choose a
chart $(U,\phi)$ on the manifold $X$ with $\phi(u)=x$ for $u\in U$,
and define the {\it depth\/} $\depth_Xx$ of $x$ in $X$ by
$\depth_Xx=\depth_Uu$. This is independent of the choice of
$(U,\phi)$. For each $l=0,\ldots,m$, define the {\it depth\/ $l$
stratum\/} of $X$ to be
\begin{equation*}
S^l(X)=\bigl\{x\in X:\depth_Xx=l\bigr\}.
\end{equation*}
Then $X=\coprod_{l=0}^mS^l(X)$ and $\overline{S^l(X)}=
\bigcup_{k=l}^m S^k(X)$. The {\it interior\/} of $X$ is
$X^\ci=S^0(X)$. Each $S^l(X)$ has the structure of an
$(m-l)$-manifold without boundary.
\end{dfn}

\begin{dfn}
\label{sm2def7}
Let $X$ be an $m$-manifold with corners, $x\in X$, and $k=0,1,\ldots,m$. A {\it local $k$-corner component\/ $\ga$ of\/ $X$ at\/} $x$ is a local choice of connected component of $S^k(X)$ near $x$. That is, for each small open neighbourhood $V$ of $x$ in $X$, $\ga$
gives a choice of connected component $W$ of $V\cap S^k(X)$ with
$x\in\overline W$, and any two such choices $V,W$ and $V',W'$ must
be compatible in that $x\in\overline{(W\cap W')}$. When $k=1$, we call $\ga$ a {\it local boundary component}.

As sets, define the {\it boundary\/} $\pd X$ and {\it k-corners\/} $C_k(X)$ for $k=0,1,\ldots,m$ by
\begin{align*}
\pd X&=\bigl\{(x,\be):\text{$x\in X$, $\be$ is a local boundary
component of $X$ at $x$}\bigr\},\\
C_k(X)&=\bigl\{(x,\ga):\text{$x\in X$, $\ga$ is a local $k$-corner 
component of $X$ at $x$}\bigr\}.
\end{align*}
Define $i_X:\pd X\ra X$ and $\Pi_k:C_k(X)\ra X$ by $i_X:(x,\be)\mapsto x$, $\Pi_k:(x,\ga)\mapsto x$.

If $(U,\phi)$ is a chart on $X$ with $U\subseteq\R^m_k$ open, then for each $i=1,\ldots,k$ we can define a chart $(U_i,\phi_i)$ on $\pd X$ by
\begin{align*}
&U_i=\bigl\{(x_1,\ldots,x_{m-1})\in \R^{m-1}_{k-1}:
(x_1,\ldots,x_{i-1},0,x_i,\ldots,x_{m-1})\in
U\subseteq\R^m_k\bigr\},\\
&\phi_i:(x_1,\ldots,x_{m-1})\longmapsto\bigl(\phi
(x_1,\ldots,x_{i-1}, 0,x_i,\ldots,x_{m-1}),\phi_*(\{x_i=0\})\bigr).
\end{align*}
The set of all such charts on $\pd X$ forms an atlas, making $\pd X$ into a manifold with corners of dimension $m-1$, and $i_X:\pd X\ra X$ into a smooth (but not interior) map. Similarly, we make $C_k(X)$ into an $(m-k)$-manifold with corners, and $\Pi_k:C_k(X)\ra X$ into a smooth map.

We call $X$ a {\it manifold without boundary\/} if $\pd X=\es$, and
a {\it manifold with boundary\/} if $\pd^2X=\es$. We write $\Man$ and $\Manb$ for the full subcategories of $\Manc$ with objects manifolds without boundary, and manifolds with boundary, so that $\Man\subset\Manb\subset\Manc$. This definition of $\Man$ is equivalent to the usual definition of the category of manifolds.
\end{dfn}

For $X$ a manifold with corners and $k\ge 0$, there are natural identifications 
\ea
\begin{split}
\pd^kX\cong\bigl\{(x&,\be_1,\ldots,\be_k):\text{$x\in X,$
$\be_1,\ldots,\be_k$ are distinct}\\
&\text{local boundary components for $X$ at $x$}\bigr\},
\end{split}
\label{sm2eq10}\\
\begin{split}
C_k(X)\cong\bigl\{(x&,\{\be_1,\ldots,\be_k\}):\text{$x\in X,$
$\be_1,\ldots,\be_k$ are distinct}\\
&\text{local boundary components for $X$ at $x$}\bigr\}.
\end{split}
\label{sm2eq11}
\ea
There is a natural, free, smooth action of the symmetric group $S_k$ on $\pd^kX$, by permutation of $\be_1,\ldots,\be_k$ in \eq{sm2eq10}, and \eq{sm2eq10}--\eq{sm2eq11} give a natural diffeomorphism
\e
C_k(X)\cong\pd^kX/S_k.
\label{sm2eq12}
\e

Corners commute with boundaries: there are natural isomorphisms
\e
\begin{aligned}
\pd C_k(X)\cong C_k&(\pd X)\cong \bigl\{(x,\{\be_1,\ldots,
\be_k\},\be_{k+1}):x\in X,\; \be_1,\ldots,\be_{k+1}\\
&\text{are distinct local boundary components for $X$ at
$x$}\bigr\}.
\end{aligned}
\label{sm2eq13}
\e
For products of manifolds with corners we have natural diffeomorphisms
\ea
\pd(X\t Y)&\cong (\pd X\t Y)\amalg (X\t\pd Y),
\label{sm2eq14}\\
C_k(X\t Y)&\cong \ts\coprod_{i,j\ge 0,\; i+j=k}C_i(X)\t C_j(Y).
\label{sm2eq15}
\ea

\begin{ex}
\label{sm2ex1}
The {\it teardrop\/} $T=\bigl\{(x,y)\in\R^2:x\ge 0$, $y^2\le x^2-x^4\bigr\}$, shown in Figure \ref{sm2fig1}, is a manifold with corners of dimension 2. The boundary $\pd T$ is diffeomorphic to $[0,1]$, and so is connected, but $i_T:\pd T\ra T$ is not injective. Thus $T$ is not a manifold with faces, in the sense of Remark \ref{sm2rem1}.
\begin{figure}[htb]
\begin{xy}
,(-1.5,0)*{}
,<6cm,-1.5cm>;<6.7cm,-1.5cm>:
,(3,.3)*{x}
,(-1.2,2)*{y}
,(-1.5,0)*{\bullet}
,(-1.5,0); (1.5,0) **\crv{(-.5,1)&(.1,1.4)&(1.5,1.2)}
?(.06)="a"
?(.12)="b"
?(.2)="c"
?(.29)="d"
?(.4)="e"
?(.5)="f"
?(.6)="g"
?(.7)="h"
?(.83)="i"
,(-1.5,0); (1.5,0) **\crv{(-.5,-1)&(.1,-1.4)&(1.5,-1.2)}
?(.06)="j"
?(.12)="k"
?(.2)="l"
?(.29)="m"
?(.4)="n"
?(.5)="o"
?(.6)="p"
?(.7)="q"
?(.83)="r"
,"a";"j"**@{.}
,"b";"k"**@{.}
,"c";"l"**@{.}
,"d";"m"**@{.}
,"e";"n"**@{.}
,"f";"o"**@{.}
,"g";"p"**@{.}
,"h";"q"**@{.}
,"i";"r"**@{.}
\ar (-1.5,0);(3,0)
\ar (-1.5,0);(-3,0)
\ar (-1.5,0);(-1.5,2)
\ar (-1.5,0);(-1.5,-2)
\end{xy}
\caption{The teardrop, a 2-manifold with corners}
\label{sm2fig1}
\end{figure}
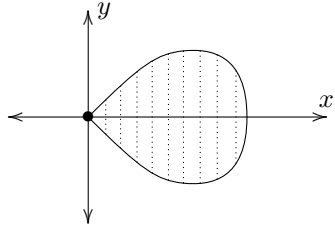
\end{ex}

The following lemma is easy to prove from Definition \ref{sm2def4}(b).

\begin{lem}
\label{sm2lem1}
Let\/ $f:X\ra Y$ be a smooth map of manifolds with corners. Then $f$ \begin{bfseries}is compatible with the depth stratifications\end{bfseries} $X=\coprod_{k\ge 0}S^k(X),$ $Y=\coprod_{l\ge 0}S^l(Y)$ in Definition\/ {\rm\ref{sm2def6},} in the sense that if\/ $\es\ne W\subseteq S^k(X)$ is a connected subset for some $k\ge 0,$ then $f(W)\subseteq S^l(Y)$ for some unique $l\ge 0$.
\end{lem}

It is {\it not\/} true that general smooth $f:X\ra Y$ induce maps $\pd f:\pd X\ra\pd Y$ or $C_k(f):C_k(Y)\ra C_k(Y)$. For example, if $f:X\ra Y$ is the inclusion $[0,\iy)\hookra\R$ then no map $\pd f:\pd X\ra\pd Y$ exists, as $\pd X\ne\es$ and $\pd Y=\es$. So boundaries and $k$-corners do not give functors on $\Manc$. However, if we work in the enlarged category $\cManc$ of Definition \ref{sm2def5} and consider the full corners $C(X)=\coprod_{k\ge 0}C_k(X)$, we can define a functor.

\begin{dfn}
\label{sm2def8}
Define the {\it corners\/} $C(X)$ of a manifold with corners $X$ by
\begin{align*}
&C(X)=\ts\coprod_{k=0}^{\dim X}C_k(X)\\
&=\bigl\{(x,\ga):\text{$x\in X$, $\ga$ is a local $k$-corner 
component of $X$ at $x$, $k\ge 0$}\bigr\},
\end{align*}
considered as an object of $\cManc$ in Definition \ref{sm2def5}, a manifold with corners of mixed dimension. Define $\Pi_X:C(X)\ra X$ by $\Pi_X:(x,\ga)\mapsto x$. This is smooth (i.e. a morphism in $\cManc$) as the maps $\Pi_k:C_k(X)\ra X$ are smooth for~$k\ge 0$.

Let $f:X\ra Y$ be a smooth map of manifolds with corners, and suppose $\ga$ is a local $k$-corner component of $X$ at $x\in X$. For each sufficiently small open neighbourhood $V$ of $x$ in $X$, $\ga$ gives a choice of connected component $W$ of $V\cap S^k(X)$ with $x\in\overline W$, so by Lemma \ref{sm2lem1} $f(W)\subseteq S^l(Y)$ for some $l\ge 0$. As $f$ is continuous, $f(W)$ is connected, and $f(x)\in\ov{f(W)}$. Thus there is a unique local $l$-corner component $f_*(\ga)$ of $Y$ at $f(x)$, such that if $\ti V$ is a sufficiently small open neighbourhood of $f(x)$ in $Y$, then the connected component $\ti W$ of $\ti V\cap S^l(Y)$ given by $f_*(\ga)$ has $f(W)\cap\ti W\ne\es$. This $f_*(\ga)$ is independent of the choice of sufficiently small $V,\ti V$, so is well-defined.

Define a map $C(f):C(X)\ra C(Y)$ by $C(f):(x,\ga)\mapsto (f(x),f_*(\ga))$. Then $C(f)$ is an interior morphism in $\cManc$. If $g:Y\ra Z$ is another smooth map of manifolds with corners then $C(g\ci f)=C(g)\ci C(f):C(X)\ra C(Z)$, so $C:\Manc\ra\cManc$ is a functor, which we call the {\it corner functor}.
\end{dfn}

Equations \eq{sm2eq13} and \eq{sm2eq15} imply that if $X,Y$ are manifolds with corners, we have natural isomorphisms
\ea
\pd C(X)&\cong C(\pd X),
\label{sm2eq16}\\
C(X\t Y)&\cong C(X)\t C(Y).
\label{sm2eq17}
\ea
The corner functor $C$ {\it preserves products and direct products}. That is, if $f:W\ra Y,$ $g:X\ra Y,$ $h:X\ra Z$ are smooth then the following commute
\begin{equation*}
\xymatrix@C=60pt@R=20pt{ *+[r]{C(W\t X)} \ar[d]^\cong \ar[r]_{C(f\t
h)} & *+[l]{C(Y\t Z)} \ar[d]_\cong \\ *+[r]{C(W)\!\t\! C(X)}
\ar[r]^{\raisebox{8pt}{$\st C(f) \t C(h)$}} &
*+[l]{C(Y)\!\t\! C(Z),} }\;
\xymatrix@C=65pt@R=3pt{ & C(Y\t Z) \ar[dd]^\cong \\
C(X) \ar[ur]^(0.4){C((g,h))} \ar[dr]_(0.4){(C(g),C(h))} \\
& C(Y)\!\t\! C(Z), }
\end{equation*}
where the columns are the isomorphisms~\eq{sm2eq17}.

\begin{ex}
\label{sm2ex2}
{\bf(a)} Let $X=[0,\iy)$, $Y=[0,\iy)^2$, and define $f:X\ra Y$
by $f(x)=(x,x)$. We have
\begin{align*}
C_0(X)&\cong[0,\iy), \qquad\quad C_1(X)\cong\{0\}, & C_0(Y)&\cong[0,\iy)^2,\\
C_1(Y)&\cong\bigl(\{0\}\t [0,\iy)\bigr)\amalg \bigl([0,\iy)\t\{0\}\bigr),&
C_2(Y)&\cong\{(0,0)\}.
\end{align*}
Then $C(f)$ maps $C_0(X)\ra C_0(Y)$, $x\mapsto (x,x)$, and
$C_1(X)\ra C_2(Y)$, $0\mapsto(0,0)$.

\smallskip

\noindent{\bf(b)} Let $X=*$, $Y=[0,\iy)$ and define $f:X\ra Y$ by
$f(*)=0$. Then $C_0(X)\cong *$, $C_0(Y)\cong [0,\iy)$,
$C_1(Y)\cong\{0\}$, and $C(f)$ maps $C_0(X)\ra C_1(Y)$, $*\mapsto
0$.
\end{ex}

Note that $C(f)$ need not map $C_k(X)\ra C_k(Y)$.

\subsection{\texorpdfstring{`Fundamental classes' of manifolds with corners}{‘Fundamental classes’ of manifolds with corners}}
\label{sm26}

If $X$ is a compact oriented $n$-manifold with corners, and $\pd X\ne\es$, then $X$ does {\it not\/} have a fundamental class $[X]_\fund$ in $H_n(X,\Z)$, in the usual sense. It does have a fundamental class $[X]_\fund$ in {\it relative\/} homology $H_n(X;i_X(\pd X),\Z)$, but that is not quite what we want to discuss.

Authors who deal with manifolds with corners and homology in applications, such as Fukaya--Oh--Ohta--Ono \cite{FOOO}, sometimes take the view that for manifolds with corners one should work {\it at the chain level}, and choose a {\it fundamental chain\/} $[X]_\fund\in C^\rsi_n(X,\Z)$ for $X$, say in singular chains. As $\pd[X]_\fund\ne 0$, this does not represent a homology class. If $X$ is a manifold with boundary then $\pd [X]_\fund=[\pd X]_\fund$ is a fundamental chain for the boundary~$\pd X$.

I wish to advocate a slightly different point of view. The main ideas are these:
\begin{itemize}
\setlength{\itemsep}{0pt}
\setlength{\parsep}{0pt}
\item[(i)] Let $X$ be a manifold with corners. Then (after choosing an orientation convention) an orientation on $X$ determines orientations on $\pd X,\pd^2X,\ldots.$
\item[(ii)] Let $X$ be a compact, oriented $n$-manifold with corners. Then the interior $X^\ci$ of $X$ is an oriented $n$-manifold, which is generally {\it noncompact}, but by Properties \ref{sm2pr1}(e), \ref{sm2pr2}(e) it still has a fundamental class $[X^\ci]_\fund$ in {\it Borel--Moore\/} homology $H_n^\BM(X^\ci,\Z)$, or equivalently in {\it Steenrod homology\/} $H_n^\St(X^\ci,\Z)$. Similarly, $(\pd^kX)^\ci$ has a fundamental class $[(\pd^kX)^\ci]_\fund$ in~$H_{n-k}^\BM((\pd^kX)^\ci,\Z)\cong H_{n-k}^\St((\pd^kX)^\ci,\Z)$.
\item[(iii)] Write $X_{\le 1}=S^0(X)\amalg S^1(X)\subseteq X$, so that $S^0(X)=X^\ci$ is open in $X_{\le 1}$ with $X_{\le 1}\sm X^\ci=S^1(X)\cong(\pd X)^\ci$. Thus \eq{sm2eq8} gives a long exact sequence
\e
\!\!\!\!\!\!\!\!\!\!\!\!\!\!\!\!\!\!\xymatrix@C=14pt{
\cdots \ar[r] & H_n^\St(X^\ci,R) \ar[r]^(0.45)\pd & H_{n-1}^\St((\pd X)^\ci,R) \ar[r]^(0.52){\inc_*} & H_{n-1}^\St(X_{\le 1},R) \ar[r] &\cdots. }
\label{sm2eq18}
\e
(This also works in Borel--Moore homology, for manifolds with corners.) Under this map $\pd$ we have
\e
\pd\bigl([X^\ci]_\fund\bigr)=[(\pd X)^\ci]_\fund.
\label{sm2eq19}
\e
Similarly, working in $(\pd^k X)_{\le 1}$ we have $\pd\bigl([(\pd^kX)^\ci]_\fund\bigr)=[(\pd^{k+1}X)^\ci]_\fund$.
\item[(iv)] Equations \eq{sm2eq18}--\eq{sm2eq19} imply that
\e
\inc_*\bigl([(\pd X)^\ci]_\fund\bigr)=0\qquad\text{in $H_{n-1}^\St(X_{\le 1},R)$.}
\label{sm2eq20}
\e
This can be seen as a consequence of Stokes' Theorem.
\item[(v)] If $X$ is a compact manifold with boundary then $X=X_{\le 1}$ and $(\pd X)^\ci=\pd X$ are compact, so Steenrod and ordinary homology coincide. Then
\e
\!\!\!\!\!\!\![\pd X]_\fund\in H_{n-1}(\pd X,\Z)\;\text{with}\; \inc_*\bigl([\pd X]_\fund\bigr)=0\;\text{in}\; H_{n-1}(X,\Z).
\label{sm2eq21}
\e
\item[(vi)] For example, if $X$ is a compact oriented 1-manifold with boundary, then pushing forward \eq{sm2eq21} to the point $*$ yields the obvious
\e
\#\bigl([\pd X]_\fund\bigr)=0\quad\text{in $\Z$,}
\label{sm2eq22}
\e
where an orientation on a 0-manifold is a sign $\pm 1$ at each point, and $\#(\cdots)$ counts points with signs.
\item[(vii)] In our intended applications, we have no need to work with fundamental chains. Instead, we regard the important objects to be fundamental classes $[X^\ci]_\fund$ and $[(\pd^kX)^\ci]_\fund$ in Steenrod homology, and relations \eq{sm2eq19}--\eq{sm2eq20}. In fact, after taking transverse fibre products with insertion chains, we can arrange that we only really care about counts $\#\bigl([Y]_\fund\bigr)$ when $\dim Y=0$, and relations \eq{sm2eq22} if we can write $Y=\pd X$.
\end{itemize}
All of (i)--(vii), except the parts about Borel--Moore homology and ordinary homology, extend to s-manifolds with corners in \S\ref{sm44}. In particular, the definition of s-manifold with corners is designed to ensure that \eq{sm2eq19}--\eq{sm2eq20} hold.

\subsection{Manifolds with g-corners and a-corners}
\label{sm27}

Finally we discuss two variants on manifolds with corners due to the author \cite{Joyc3,Joyc4}, {\it manifolds with g-corners\/} and {\it manifolds with a-corners}.

\subsubsection{Manifolds with generalized corners}
\label{sm271}

We briefly summarize some ideas about {\it manifolds with generalized corners\/} ({\it manifolds with g-corners\/}) from the author \cite{Joyc3}. These were inspired by the `interior binomial varieties' of Kottke and Melrose~\cite[\S 9]{KoMe}.
\smallskip

\noindent{\bf(a)} Let $(P,0,+)$ be a commutative monoid that is {\it weakly toric\/} (i.e. finitely-generated, integral, saturated, and torsion-free). Define $X_P$ to be the set of monoid morphisms $x:P\ra[0,\iy)$, where $\bigl([0,\iy),1,\cdot\bigr)$ is the monoid $[0,\iy)$ with operation multiplication and identity 1. We make $X_P$ into a topological space, with a `smooth structure' (i.e. there are natural notions of smooth maps from open subsets of $X_P$ to $\R$, and between open subsets of $X_P$ and $X_Q$). The {\it dimension\/} of $X_P$ is $\dim X_P=\rank P$.
\smallskip

\noindent{\bf(b)} When $P=\Z$ we can identify $X_P\cong\R$ such that $t\in\R$ corresponds to the morphism $\Z\ra[0,\iy)$ mapping $n\mapsto e^{nt}$.

When $P=\N$ we can identify $X_P\cong[0,\iy)$ such that $t\in[0,\iy)$ corresponds to the morphism $\N\ra[0,\iy)$ mapping $n\mapsto t^n$.  

When $P=\N^k\t\Z^{m-k}$ for $0\le k\le m$ we can identify $X_P\cong\R^m_k=[0,\iy)^k\t\R^{m-k}$, where $\R^m_k$ has the `smooth structure' in Definition \ref{sm2def4}. Thus, the model spaces $\R^m_k$ used to define manifolds with corners $\Manc$ in \S\ref{sm24} are special examples of the spaces $X_P$ in~{\bf(a)}.
\smallskip

\noindent{\bf(c)} The category $\Mangc$ of {\it manifolds with g-corners\/} is defined in a very similar way to Definition \ref{sm2def5}, but using the spaces $X_P$ as local models rather than $\R^m_k$. As $\R^m_k$ is an example of a space $X_P$, there is a full inclusion $\Manc\subset\Mangc$. That is, manifolds with corners are examples of manifolds with g-corners.
\smallskip

\noindent{\bf(d)} A manifold with g-corners $X$ of dimension $m$ has a {\it depth stratification\/} $X=\coprod_{k=0}^mS^k(X)$, where $S^k(X)$ is an ordinary manifold of dimension $m-k$, and is a locally closed subset in $X$ with closure $\ov{S^k(X)}=\coprod_{l=k}^mS^l(X)$. The {\it interior\/} of $X$ is $X^\ci=S^0(X)$. 
\smallskip

\noindent{\bf(e)} A smooth map $f:X\ra Y$ in $\Mangc$ is called {\it interior\/} if $f(X^\ci)\subseteq Y^\ci$. We write $\Mangcin\subset\Mangc$ for the subcategory with morphisms interior maps.

As for $\cMancin\subset\cManc$ in \S\ref{sm24}, we also write $\cMangcin\subset\cMangc$ for the categories with objects {\it manifolds with g-corners with mixed dimension}, and morphisms interior or smooth maps.
\smallskip

\noindent{\bf(f)} The simplest example of a manifold with g-corners which is not an ordinary manifold with corners is the space $X_P$ from the rank 3 weakly toric monoid
\begin{equation*}
P=\bigl\{(a,b,c)\in \Z^3:a\ge 0,\; b\ge 0,\; a+b\ge c\ge 0\bigr\}.
\end{equation*}
Write
\begin{equation*}
p_1=(1,0,0),\;\> p_2=(0,1,1),\;\> p_3=(0,1,0),\;\> p_4=(1,0,1).
\end{equation*}
Then $p_1,p_2,p_3,p_4$ are generators for $P$, subject to the single relation
\begin{equation*}
p_1+p_2=p_3+p_4.
\end{equation*}
Because of this we may identify
\e
X_P\cong \bigl\{(x_1,x_2,x_3,x_4)\in[0,\iy)^4:x_1x_2=x_3x_4\bigr\}.
\label{sm2eq23}
\e

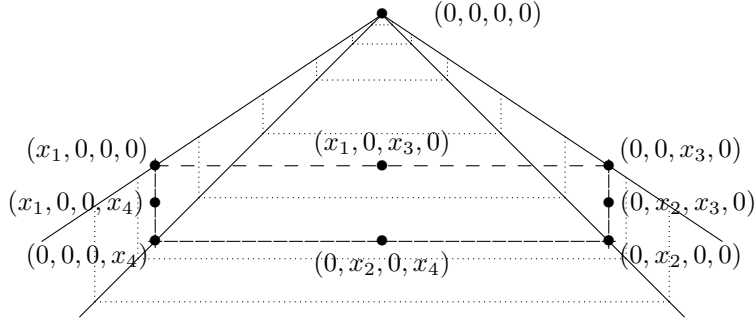
\begin{figure}[htb]
\centerline{$\splinetolerance{.8pt}
\begin{xy}
0;<1mm,0mm>:
,(0,0)*{\bu}
,(14,0)*{(0,0,0,0) }
,(-30,-20)*{\bu}
,(-39,-18)*{(x_1,0,0,0)}
,(30,-20)*{\bu}
,(39.5,-18)*{(0,0,x_3,0)}
,(-30,-30)*{\bu}
,(-39,-32)*{(0,0,0,x_4)}
,(30,-30)*{\bu}
,(39.5,-32)*{(0,x_2,0,0)}
,(0,-30)*{\bu}
,(0,-33)*{(0,x_2,0,x_4)}
,(0,-20)*{\bu}
,(0,-17)*{(x_1,0,x_3,0)}
,(-30,-25)*{\bu}
,(-40.5,-25)*{(x_1,0,0,x_4)}
,(30,-25)*{\bu}
,(40.5,-25)*{(0,x_2,x_3,0)}
,(0,0);(-45,-30)**\crv{}
?(.8444)="aaa"
?(.85)="bbb"
?(.75)="ccc"
?(.65)="ddd"
?(.55)="eee"
?(.45)="fff"
?(.35)="ggg"
?(.25)="hhh"
?(.15)="iii"
?(.05)="jjj"
,(0,0);(45,-30)**\crv{}
?(.8444)="aaaa"
?(.85)="bbbb"
?(.75)="cccc"
?(.65)="dddd"
?(.55)="eeee"
?(.45)="ffff"
?(.35)="gggg"
?(.25)="hhhh"
?(.15)="iiii"
?(.05)="jjjj"
,(0,0);(-40,-40)**\crv{}
?(.95)="a"
?(.85)="b"
?(.75)="c"
?(.65)="d"
?(.55)="e"
?(.45)="f"
?(.35)="g"
?(.25)="h"
?(.15)="i"
?(.05)="j"
,(0,0);(40,-40)**\crv{}
?(.95)="aa"
?(.85)="bb"
?(.75)="cc"
?(.65)="dd"
?(.55)="ee"
?(.45)="ff"
?(.35)="gg"
?(.25)="hh"
?(.15)="ii"
?(.05)="jj"
,"a";"aa"**@{.}
,"b";"bb"**@{.}
,"c";"cc"**@{.}
,"d";"dd"**@{.}
,"e";"ee"**@{.}
,"f";"ff"**@{.}
,"g";"gg"**@{.}
,"h";"hh"**@{.}
,"i";"ii"**@{.}
,"j";"jj"**@{.}
,"a";"aaa"**@{.}
,"b";"bbb"**@{.}
,"c";"ccc"**@{.}
,"d";"ddd"**@{.}
,"e";"eee"**@{.}
,"f";"fff"**@{.}
,"g";"ggg"**@{.}
,"h";"hhh"**@{.}
,"i";"iii"**@{.}
,"j";"jjj"**@{.}
,"aa";"aaaa"**@{.}
,"bb";"bbbb"**@{.}
,"cc";"cccc"**@{.}
,"dd";"dddd"**@{.}
,"ee";"eeee"**@{.}
,"ff";"ffff"**@{.}
,"gg";"gggg"**@{.}
,"hh";"hhhh"**@{.}
,"ii";"iiii"**@{.}
,"jj";"jjjj"**@{.}
,(-30,-20);(30,-20)**@{--}
,(-30,-20);(-30,-30)**\crv{}
,(-30,-30);(30,-30)**\crv{}
,(30,-30);(30,-20)**\crv{}
\end{xy}$}
\caption{3-manifold with g-corners $X_P$ in \eq{sm2eq23}}
\label{sm2fig2}
\end{figure}

We sketch $X_P$ in Figure \ref{sm2fig2}. We picture it as a 3-dimensional infinite pyramid on a square base. It has one vertex $(0,0,0,0)$, four 1-dimensional edges of points $(x_1,0,0,0),(0,x_2,0,0), (0,0,x_3,0),(0,0,0,x_4)$, four 2-dimensional faces of points $(x_1,0,x_3,0)$, $(x_1,0,0,x_4)$, $(0,x_2,x_3,0)$, $(0,x_2,0,x_4)$, and an interior $X_P^\ci\ab\cong\R^3$ of points $(x_1,x_2,x_3,x_4)$. Then $X_P$ is not a manifold with corners near $(0,0,0,0)$, as we can see from the non-simplicial face structure.
\smallskip

\noindent{\bf(g)} Let $X$ be a manifold with g-corners of dimension $m$. Then as in \S\ref{sm25} we can define the {\it boundary\/} $\pd X$, a manifold with g-corners of dimension $m-1$, and the $k$-{\it corners\/} $C_k(X)$, a manifold with g-corners of dimension $m-k$ for $k=0,\ldots,m$, with $\pd X=C_1(X)$. The {\it corners\/} is $C(X)=\coprod_{k=0}^mC_k(X)$, an object of $\cMangcin\subset\cMangc$. There is a {\it corner functor\/} $C:\Mangc\ra\cMangcin$ acting by $X\mapsto C(X)$ on objects.
\smallskip

\noindent{\bf(h)} As in \cite[\S 3.4]{Joyc4}, some properties of boundaries and corners of manifolds with corners are {\it false\/} for manifolds with g-corners. In particular, equations \eq{sm2eq10}--\eq{sm2eq13} are all false in $\Mangc$. For example, \eq{sm2eq10}--\eq{sm2eq11} with $k=2$ are false for $X_P$ in \eq{sm2eq23}, as the right hand sides contain additional points over $(0,0,0,0)$ from boundary components $\be_1,\be_2$ on opposite sides of the pyramid.
\smallskip

\noindent{\bf(i)} One reason for studying manifolds with g-corners is that {\it transverse fibre products\/} exist in $\Mangc$ under weaker conditions than in $\Manc$, as in \cite[\S 4.3]{Joyc3}. For example, $X_P$ in \eq{sm2eq23} may be written as a transverse fibre product $[0,\iy)^2\t_{f,[0,\iy),f}[0,\iy)^2$, where $f:[0,\iy)^2\ra[0,\iy)$ maps $(x,y)\mapsto xy$. The corresponding fibre product does not exist in~$\Manc$.
\smallskip

\noindent{\bf(j)} Another reason is that there are moduli spaces occurring in nature which have the structure of manifolds with g-corners, but not of manifolds with corners. For example, Ma'u and Woodward \cite{MaWo} define moduli spaces $\oM_{n,1}$ of `stable $n$-marked quilted discs'. As in \cite[\S 6]{MaWo}, for $n\ge 4$ these are not manifolds with corners, but have an exotic corner structure; in the language of \cite{Joyc3}, the $\oM_{n,1}$ are manifolds with g-corners.

\subsubsection{Manifolds with analytic corners}
\label{sm272}

We summarize some ideas about {\it manifolds with analytic corners\/} ({\it manifolds with a-corners\/}) from the author \cite{Joyc4}.
\smallskip

\noindent{\bf(a)} As for manifolds with corners in \S\ref{sm24}, manifolds with a-corners are spaces $X$ locally modelled on $\R^m_k=[0,\iy)^k\t\R^{m-k}$ for $0\le k\le m$. However, the notions of {\it smooth map\/} from open subsets of $\R^m_k$ to $\R$, and between open subsets of $\R^m_k,\R^n_l$, are {\it not\/} those in Definition \ref{sm2def7}, but are modified.

We will not give the full definition, but to give the general idea, a continuous map $f:[0,\iy)\ra\R$ is {\it a-smooth\/} if  $\frac{\d^kf}{\d x^k}$ exists and is continuous on $(0,\iy)$, and $\frac{\d^kf}{\d x^k}=O(x^{\al-k})$ as $x\ra 0$ for some $\al>0$, for each $k=1,2,\ldots.$

In this way we define categories $\Manacin\subset\Manac$ of {\it manifolds with a-corners}, with interior or a-smooth maps.
\smallskip

\noindent{\bf(b)} Essentially all of the properties of manifolds with corners in \S\ref{sm24}--\S\ref{sm26} extend to manifolds with a-corners without change, including the depth stratification $X=\coprod_{k=0}^mS^k(X)$, boundaries $\pd X$ and $k$-corners $C_k(X)$, and the corner functor $C:\Manac\ra\cManacin\subset\cManac$.
\smallskip

\noindent{\bf(c)} The author introduced manifolds with a-corners for two reasons. Firstly, as in the work  of Melrose and others \cite{Melr1,Melr2,Melr3}, many interesting analytic problems involving elliptic or parabolic equations on noncompact manifolds with asymptotic ends can be rephrased in terms of functions on the noncompact interior $X^\ci$ of a compact manifold with boundary or corners $X$. But to extend these equations from $X^\ci$ to $X$, the appropriate smooth structure on $X$ is that of a manifold with a-corners, not an ordinary manifold with corners. So, several areas of analysis of p.d.e.s with asymptotic conditions can be rewritten in the language of manifolds with a-corners, and a general theory developed.

Secondly, for some classes of moduli spaces of solutions of nonlinear elliptic p.d.e.s in which the moduli space is compactified by adding a `boundary' of singular solutions, such as moduli spaces of $J$-holomorphic discs in Symplectic Geometry, the correct smooth structure on the moduli space appears to be that of manifold with a-corners, not a manifold with corners.

\section{\texorpdfstring{Stratified manifolds (s-manifolds)}{Stratified manifolds (s-manifolds)}}
\label{sm3}

We define s-manifolds in \S\ref{sm31}, and give examples in \S\ref{sm32}. Sections \ref{sm33}--\ref{sm38} show that much of the standard differential geometry of smooth manifolds extends to s-manifolds in a nice way, in particular orientations, fundamental classes, partitions of unity, the Whitney Embedding Theorem, (oriented) transverse fibre products, Sard's Theorem, and vector bundles and sections.

\subsection{The definition of s-manifolds}
\label{sm31}

\begin{dfn}
\label{sm3def1}
An {\it s-manifold} (short for {\it stratified manifold\/}) $\bX$ of {\it dimension\/} $\dim\bX=n\ge 0$, is data $\bX=\bigl(X,\{X^i:i\in I\}\bigr)$ satisfying (a)--(e), where:
\begin{itemize}
\setlength{\itemsep}{0pt}
\setlength{\parsep}{0pt}
\item[(a)] $X$ is a locally compact, Hausdorff, second countable topological space.
\item[(b)] $\{X^i:i\in I\}$ is a stratification of $X$ into nonempty, locally closed subsets. That is, $\es\ne X^i\subseteq X$ with $X^i$ locally closed for each $i\in I$, and $X^i\cap X^j=\es$ for $i\ne j$ in $I$, and $X=\coprod_{i\in I}X^i$. We call the $X^i$ the {\it strata\/} of $\bX$. Here $I$ is an indexing set which we use for notational convenience, but is not part of the data of $\bX$, which is just the set of subsets $X^i$ of $X$. We require that $I$ is finite or countable.
\item[(c)] Each stratum $X^i$ has the structure of a smooth manifold of dimension $\dim X^i\le n$. These smooth manifold structures are part of the data of $\bX$.

If $x\in X$ (which we also write as $x\in\bX$) then $x$ lies in a unique stratum $X^i\subseteq X$, and we write $\dim_x\bX=\dim X^i$. We also define the {\it tangent space\/} $T_x\bX$ to be $T_xX^i$. Note that $T_x\bX$ can have dimension $0,1,\ldots,\dim\bX$.
\item[(d)] Each $x\in X$ should admit an open neighbourhood $U$ in $X$ and a continuous map $f:U\ra\R^N$ for $N\gg 0$ (in fact $N=2n+1$ will do, see Theorem \ref{sm3thm3}) such that $f:U\ra f(U)$ is a homeomorphism, and $f\vert_{U\cap X^i}:U\cap X^i\ra\R^N$ is a smooth embedding of manifolds for all~$i\in I$.
\item[(e)] Each $x\in X$ should admit an open neighbourhood $V$ in $X$ such that $V\cap X^i\ne\es$ for only finitely many $i\in I$. (Then we say that $\{X^i:i\in I\}$ is a  {\it locally finite stratification\/} of $X$.) Furthermore, there should exist a partial order $\pr$ on $I_V:=\{i\in I:V\cap X^i\ne\es\}$ such that:
\begin{itemize}
\setlength{\itemsep}{0pt}
\setlength{\parsep}{0pt}
\item[(i)] For each $i\in I_V$ we have $\ov{(V\cap X^i)}\subseteq\coprod_{j\in I:i\pr j}(V\cap X^j)$, where
$\ov{(V\cap X^i)}$ is the closure of $V\cap X^i$ in $V$. 
\item[(ii)] Let $i,j\in I_V$ with $i\ne j$ and $i\pr j$. Then either:
\begin{itemize}
\setlength{\itemsep}{0pt}
\setlength{\parsep}{0pt}
\item[(A)] $\dim X^i=n-1$ or $n$ and $\dim X^j<\dim X^i$; or
\item[(B)] $\dim X^i\le n-2$ and $\dim X^j\le n-2$.
\end{itemize}
\end{itemize}
\end{itemize}

For an s-manifold $\bX$, define subsets $X_0,X_1,X_{\le 1},X_{\ge 2}$ of $X$ by
\e
\begin{aligned}
X_0&=\coprod_{i\in I:\dim X^i=n}X^i, & X_1&=\coprod_{i\in I:\dim X^i=n-1}X^i, \\
X_{\le 1}&=\coprod_{i\in I:\dim X^i\ge n-1}X^i,& X_{\ge 2}&=\coprod_{i\in I:\dim X^i\le n-2}X^i.
\end{aligned}
\label{sm3eq1}
\e
Here the subscript refers to the codimension, so $X_{\le 1}$ is the union of strata of codimension $\le 1$, and so on. Then (e) implies that:
\begin{itemize}
\setlength{\itemsep}{0pt}
\setlength{\parsep}{0pt}
\item[(i)] $X_0$ is a smooth manifold of dimension $n$. In particular, (e)(i),(ii) imply that if $X^i,X^j$ are strata of $\bX$ of dimension $n$ for $i\ne j$ then $\ov{X^i}\cap X^j=\es$, so as a topological space $X_0$ is the disjoint union of the $X^i$ for $i\in I$ with $\dim X^i=n$. There are no issues with strata $X^i$ being topologically glued together in a pathological way.
\item[(ii)] $X_1$ is a smooth manifold of dimension $n-1$.
\item[(iii)] $X_{\le 1}=X_0\amalg X_1$ and $X=X_0\amalg X_1\amalg X_{\ge 2}=X_{\le 1}\amalg X_{\ge 2}$.
\item[(iv)] $X_0$ and $X_{\le 1}$ are open in $X$, and $X_{\ge 2}$ is closed in $X$, and $X_0$ is open in $X_{\le 1}$, and $X_1$ is closed in $X_{\le 1}$.
\end{itemize}

Let $\bX$ be an s-manifold, and $U\subseteq X$ be open. It is easy to see that $\bU=\bigl(U,\{U\cap X^i:i\in I$, $U\cap X^i\ne\es\}\bigr)$ is an s-manifold with $\dim\bU=\dim\bX$. We call $\bU\subseteq\bX$ an {\it open s-submanifold}.

Now let $\bX$ be an s-manifold as above, and $\bY=\bigl(Y,\{Y^j:j\in J\}\bigr)$ be another s-manifold. A {\it morphism}, or {\it smooth map}, $f:\bX\ra\bY$ is a continuous map $f:X\ra Y$, such that for each $i\in I$ there exists a unique $j\in J$ such that $f(X^i)\subseteq Y^j$, and $f\vert_{X^i}:X^i\ra Y^j$ is a smooth map of manifolds. We call $f$ a {\it diffeomorphism\/} if it is smooth and invertible with smooth inverse. Compositions of smooth maps are smooth, and identities are smooth. Write $\SMan$ for the category with objects s-manifolds and morphisms smooth maps.

If $f:\bX\ra\bY$ is smooth and $x\in\bX$ with $f(x)=y\in\bY$, then $x\in X^i$, $y\in Y^j$ for strata $X^i,Y^j$ of $\bX,\bY$, with $f\vert_{X^i}:X^i\ra Y^j$ smooth. We define the {\it tangent map\/} $T_xf:T_x\bX\ra T_y\bY$ to be $T_x(f\vert_{X^i}):T_xX^i\ra T_yY^j$, for $T_x\bX,T_y\bY$ as in~(c).

Any smooth $n$-manifold $X$ can be made into an s-manifold $\bX$ with just one stratum $X^i=X$, so that $X_0=X$ and $X_1=X_{\ge 2}=\es$. This induces a full and faithful inclusion $\inc:\Man\hookra\SMan$ of the category $\Man$ of manifolds into $\SMan$. So we can regard ordinary manifolds as examples of s-manifolds.

As for $\cManc$ in Definition \ref{sm2def2}, we write $\cSMan$ for the category whose objects are disjoint unions $\coprod_{m=0}^\iy\bX_m$, where $\bX_m$ is an s-manifold of dimension $m$, allowing $\bX_m=\es$, and whose morphisms are continuous maps $f:\coprod_{m=0}^\iy X_m\ra\coprod_{n=0}^\iy Y_n$, such that $f\vert_{X_m\cap f^{-1}(Y_n)}:\bX_m\cap f^{-1}(\bY_n)\ra\bY_n$ is a smooth map of s-manifolds for all $m,n\ge 0$. Objects of $\cSMan$ will be called {\it s-manifolds of mixed dimension}.
\end{dfn}

\begin{rem}
\label{sm3rem1}
{\bf(a)} We can think of an s-manifold $\bX$ as a kind of {\it singular manifold}, where $X_0\subseteq X$ is the nonsingular part, which is a smooth manifold of dimension $\dim\bX$, and $X_{\ge 1}$ is the singular part, which has codimension~$\ge 1$.

In order to make {\it fundamental classes\/} $[\bX]_\fund$ work in \S\ref{sm33}, it is necessary to impose some conditions on the codimension 1 part $X_1$, such as Definition \ref{sm3def1}(e)(ii)(A). But strata $X^i$ with $\dim X^i\le n-2$ are more-or-less unconstrained, and can be glued together in very strange ways.
\smallskip

\noindent{\bf(b)} We need $X$ to be {\it locally compact\/} and {\it Hausdorff\/} so that Steenrod homology is defined, as we will need this for fundamental classes in \S\ref{sm33}. We will show in Corollary \ref{sm3cor2} that $X$ is automatically {\it metrizable}, so Property \ref{sm2pr2}(j) holds.
\end{rem}

\subsection{Examples of s-manifolds}
\label{sm32}

\begin{ex}
\label{sm3ex1}
{\bf(Manifolds with corners, g-corners, and a-corners.)} In \S\ref{sm24}--\S\ref{sm27} we discussed the categories $\Manc,\Mangc,\Manac$ of manifolds with corners, g-corners, and a-corners. If $X$ lies in any of these categories with dimension $n$ it has a {\it depth stratification\/} $X=\coprod_{k=0}^nS^k(X)$, where $S^k(X)$ is an ordinary manifold of dimension $n-k$, and is a locally closed subset in $X$ with closure $\ov{S^k(X)}=\coprod_{l=k}^nS^l(X)$. Then $\bX=\bigl(X,\{S^k(X)':S^k(X)'$ is a connected component of $S^k(X),$ $k=0,\ldots,n\}\bigr)$ is an s-manifold of dimension $n$. This induces inclusion functors $\Manc\hookra\SMan$, $\Mangc\hookra\SMan$, $\Manac\hookra\SMan$, which are faithful, but not full.
\end{ex}

\begin{rem}
\label{sm3rem2}
Example \ref{sm3ex1} does {\it not\/} yet give examples of s-manifolds with corners in the sense of \S\ref{sm4}. An s-manifold with corners $\bfX$ is an s-manifold $\bX$ together with extra data needed to form the boundary $\pd\bfX$ and $k$-corners~$C_k(\bfX)$.
\end{rem}

\begin{ex}
\label{sm3ex2}
{\bf(Simplicial complexes.)}
A {\it simplicial complex\/} is a topological space $X$ with a stratification into $k$-{\it simplices\/} $\si:\De_k\ra X$, for $\si$ a homeomorphism with its image in $X$, where $\De_k$ is the standard $k$-simplex from \eq{sm2eq1}, and all boundary faces of $\si$ are also simplices of $X$. It is {\it locally finite\/} if $X$ can be covered by open sets intersecting only finitely many simplices. It is a {\it simplicial\/ $n$-complex\/} if the maximum dimension of simplices is $n$.

Let $X$ be a locally compact, Hausdorff, locally finite $n$-complex. Then we can make $X$ into an s-manifold $\bX$ of dimension $n$ by taking the strata to be $X^\si=\si(\De_k^\ci)$ for all simplices $\si:\De_k\ra X$ of $X$, where $\De_k^\ci=\De_k\cap(0,\iy)^{k+1}$ is the interior of~$\De_k$.
\end{ex}

\begin{ex}
\label{sm3ex3}
{\bf(Orbifolds.)} An ({\it effective\/}) {\it orbifold\/} $X$ is a generalization of a manifold which is locally modelled on $\R^n/G$ for $G\subset\GL(n,\R)$ a finite group. Much of the differential geometry of manifolds extends nicely to orbifolds (though with some complexities, depending on whether orbifolds form a category or a 2-category). See for example Satake \cite{Sata}, Moerdijk \cite{Moer}, and Adem--Leida--Ruan \cite{ALR}. 

Let $X$ be an orbifold of dimension $n$. Then $X$ has a natural locally finite, locally closed stratification $X=\coprod_{[G]}X^{[G]}$ indexed by conjugacy classes $[G]$ of finite subgroups $G\subset\SL(n,\R)$, where $x\in X^{[G]}$ if $X$ near $x$ is locally modelled on $\R^n/G$ near 0. Each $X^{[G]}$ is naturally a smooth manifold of dimension $\dim(\R^n)^G$, where $(\R^n)^G$ is the $G$-invariant subspace of $\R^n$. We call the $X^{[G]}$ the {\it orbifold strata\/} of $X$. It is easy to show that
\begin{align*}
\bX=\bigl(X,\bigl\{X^{\prime[G]}:\,&\text{$[G]$ is a conjugacy class of finite subgroups $G\subset\GL(n,\R)$,}\\
&\text{and $X^{\prime[G]}$ is a connected component of $X^{[G]}$}\bigr\}\bigr)
\end{align*} 
is an s-manifold $\bX$ of dimension $n$.

Morphisms of orbifolds induce morphisms of the corresponding s-manifolds. This yields an inclusion functor $\Orb\hookra\SMan$ (treating orbifolds as a category, rather than a 2-category), which is faithful but not full.
\end{ex}

\begin{ex}
\label{sm3ex4}
{\bf(An example with pathological topological space.)} Let $n\ge 1$, and define a subset $X\subset\R^{n+1}$ by
\begin{align*}
X=\bigl\{(x_0,\ldots,x_n)\in\R^{n+1}:\,&\ts(x_0-\frac{1}{k})^2+x_1^2+\cdots+x_n^2=\frac{1}{k^2}\\
&\text{for some $k=1,2,\ldots$}\bigr\},
\end{align*}
so $X$ is the union of $n$-spheres $\cS^n_{1/k}$ of radius $\frac{1}{k}$ and centre $(\frac{1}{k},0,\ldots,0)$, which all intersect at $(0,\ldots,0)$. Divide $X$ into two strata $X^1=X\sm\{(0,\ldots,0)\}$ and $X^2=\{(0,\ldots,0)\}$, so that $X=\coprod_{i\in I}X^i$ for $I=\{1,2\}$. Then $X^1$ is an $n$-manifold, and $X^2$ a 0-manifold, as submanifolds of $\R^{n+1}$. Definition \ref{sm3def1}(a)--(e) for $\bX$ are easy, where in (e) we define $\pr$ on $I=\{1,2\}$ by $\pr=\le$. Thus $\bX$ is a compact s-manifold, of dimension $n$. Figure \ref{sm3fig1} illustrates the case $n=1$.
\begin{figure}[htb]
\centerline{$\splinetolerance{.8pt}
\begin{xy}
0;<1cm,0cm>:
,(4,2)*{}
,(-.5,-2)*{}
,(0,0)*{\bu}
,(-.5,0)*+{(0,0)}
,(1,0)*{\ellipse<2cm>{-}}
,(.5,0)*{\ellipse<1cm>{-}}
,(.333,0)*{\ellipse<.666cm>{-}}
,(.25,0)*{\ellipse<.5cm>{-}}
,(.2,0)*{\ellipse<.4cm>{-}}
,(.166,0)*{\ellipse<.333cm>{-}}
,(.1429,0)*{\ellipse<.2857cm>{-}}
,(.125,0)*{\ellipse<.25cm>{-}}
,(.111,0)*{\ellipse<.222cm>{-}}
,(.1,0)*{\ellipse<.2cm>{-}}
,(.0909,0)*{\ellipse<.1818cm>{-}}
,(.0833,0)*{\ellipse<.1666cm>{-}}
,(.0769,0)*{\ellipse<.1538cm>{-}}
,(.0714,0)*{\ellipse<.1428cm>{-}}
,(.0666,0)*{\ellipse<.1333cm>{-}}
,(.0625,0)*{\ellipse<.125cm>{-}}
,(.12,-.01)*{\hbox{\huge$\bullet$}}
\end{xy}$}
\caption{s-manifold $\bX$  in Example \ref{sm3ex4} when $n=1$}
\label{sm3fig1}
\end{figure}
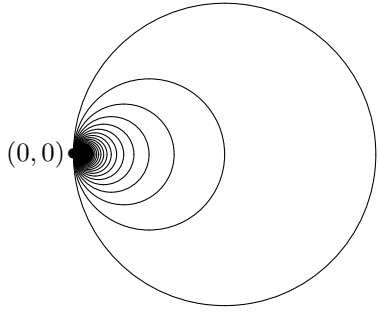
\end{ex}

\begin{ex}
\label{sm3ex5}
{\bf(A variant of Example \ref{sm3ex4} with $n=1$.)} Define $Y\subset\R^2$ by
\begin{align*}
Y=\bigl\{(y_0,y_1)\in\R^2:\,&\ts(y_0-\frac{1}{k})^2+y_1^2=\frac{1}{k^2}\text{ for some $k=1,2,\ldots$}\bigr\}\\
&\cup\bigl\{(y,0):y\in[-1,0]\bigr\},
\end{align*}
so $Y$ is the union of circles $\cS^1_{1/k}$ of radius $\frac{1}{k}$ and centre $(\frac{1}{k},0)$ for $k\ge 1$, and the interval $[-1,0]\t\{0\}$, all intersecting at $(0,0)$. This is illustrated in Figure~\ref{sm3fig2}.
\begin{figure}[htb]
\centerline{$\splinetolerance{.8pt}
\begin{xy}
0;<1cm,0cm>:
,(0,0)*+{\bu}
,(-4,0)*+{\bu}
,(-4,-.3)*+{(-1,0)}
,(-.4,-.3)*+{(0,0)}
,(-3.9,0);(-.1,0)**\crv{}
,(-3.8,0);(-.2,0)**\crv{}
,(4,2)*{}
,(-4,-2)*{}
,(1,0)*{\ellipse<2cm>{-}}
,(.5,0)*{\ellipse<1cm>{-}}
,(.333,0)*{\ellipse<.666cm>{-}}
,(.25,0)*{\ellipse<.5cm>{-}}
,(.2,0)*{\ellipse<.4cm>{-}}
,(.166,0)*{\ellipse<.333cm>{-}}
,(.1429,0)*{\ellipse<.2857cm>{-}}
,(.125,0)*{\ellipse<.25cm>{-}}
,(.111,0)*{\ellipse<.222cm>{-}}
,(.1,0)*{\ellipse<.2cm>{-}}
,(.0909,0)*{\ellipse<.1818cm>{-}}
,(.0833,0)*{\ellipse<.1666cm>{-}}
,(.0769,0)*{\ellipse<.1538cm>{-}}
,(.0714,0)*{\ellipse<.1428cm>{-}}
,(.0666,0)*{\ellipse<.1333cm>{-}}
,(.0625,0)*{\ellipse<.125cm>{-}}
,(.12,-.01)*{\hbox{\huge$\bullet$}}
\end{xy}$}
\caption{s-manifold $\bY$ in Example \ref{sm3ex5}}
\label{sm3fig2}
\end{figure}

Divide $Y$ into strata $Y^1=Y\sm\{(-1,0),(0,0)\}$ and $Y^2=\{(-1,0),(0,0)\}$, so that $Y=\coprod_{i\in I}Y^i$ for $I=\{1,2\}$. Then $Y^1$ is a 1-manifold, and $Y^2$ a 0-manifold, as submanifolds of $\R^2$, and $\bY=(Y,\{Y^1,Y^2\})$ is a compact s-manifold, of dimension~1. 
\end{ex}

We will see in \S\ref{sm33} that $\bX$ in Example \ref{sm3ex4} admits an orientation and a fundamental class, but $\bY$ in Example \ref{sm3ex5} does not. However, we will show in \S\ref{sm42}--\S\ref{sm43} that $\bY$ in Example \ref{sm3ex5} can be enhanced to an s-manifold with boundary $\bfY$ which does admit an orientation and a fundamental class.

We will return to these examples in Example \ref{sm3ex6}. We can construct many more examples of s-manifolds, including those with pathological topological spaces, by taking transverse fibre products as in Theorem \ref{sm3thm4}, or zeroes of transverse sections of vector bundles as in Theorem~\ref{sm3thm6}.

\subsection{Orientations and fundamental classes of s-manifolds}
\label{sm33}

As in Properties \ref{sm2pr1}(e) and \ref{sm2pr2}(e), if $X$ is an oriented $n$-manifold it has a fundamental class $[X]_\fund$ in Borel--Moore homology $H_n^\BM(X,\Z)$, which agrees with Steenrod homology $H_n^\St(X,\Z)$ in this case. Our goal is to show that, for a suitable notion of {\it orientation\/} on $\bX$, an oriented s-manifold $\bX$ of dimension $n$ has a fundamental class $[\bX]_\fund$ in Steenrod homology~$H_n^\St(X,\Z)$.

As in Example \ref{sm3ex1}, s-manifolds $\bX$ include manifolds with corners $X$, which do {\it not\/} have fundamental classes $[X]_\fund$ if $\pd X\ne\es$ (except in relative homology $H_n(X;i_X(\pd X),\Z)$, which is not relevant at this point). So, our notion of orientation on an s-manifold $\bX$ should be restrictive: orientations will not exist for s-manifolds $\bX$ from manifolds with corners $X$ with $\pd X\ne\es$. We introduce an extra structure of an {\it orientation bundle\/} $(\Or_\bX,\om_\bX)$ on an s-manifold $\bX$, which makes orientations and fundamental classes behave like those on manifolds. 

\begin{dfn}
\label{sm3def2}
Let $\bX$ be an s-manifold of dimension $n$, and define $X_0,X_1,X_{\le 1}$ as in Definition \ref{sm3def1}. Then $X_0$ is an $n$-manifold and $X_1$ an $(n-1)$-manifold, and $X_1\subset X_{\le 1}$ is closed with $X_0=X_{\le 1}\sm X_1$, with $X_{\le 1}$ locally compact and Hausdorff. Thus by Property \ref{sm2pr2}(g) for Steenrod homology, we have a long exact sequence
\e
\xymatrix@C=20pt{ \cdots \ar[r] & H_n^\St(X_{\le 1},\Z) \ar[r]^{\vert_{X_0}} & H_n^\St(X_0,\Z) \ar[r]^(0.49)\pd & H_{n-1}^\St(X_1,\Z) \ar[r] & \cdots.  }
\label{sm3eq2}
\e
An {\it orientation\/} on $\bX$ is an orientation on $X_0$ such that
\e
\pd\bigl([X_0]_\fund\bigr)=0\qquad\text{in $H_{n-1}^\St(X_1,\Z)$,}
\label{sm3eq3}
\e
with $[X_0]_\fund\in H_n^\St(X_0,\Z)$ the fundamental class from Property~\ref{sm2pr2}(e).

If $X_0$ is an $n$-manifold then $U_0\mapsto H_n^\St(X_0,\Z)$ for open subsets $U_0\subseteq X_0$ is a sheaf of abelian groups on $X_0$. So the middle and right terms in \eq{sm3eq2} are sheaves over $X_0,X_1$, and $\pd$ is also a morphism of sheaves. Thus, the condition $\pd\bigl([X_0]_\fund\bigr)=0$ is local (it is enough to verify it on the sets of any open cover of $X_{\le 1}$), and orientations can be glued on the sets of an open cover.

We call $\bX$ {\it oriented\/} if it has a particular choice of orientation, and {\it orientable\/} if it admits an orientation, and {\it locally orientable\/} if $\bX$ can be covered with open s-submanifolds $\bU\subseteq\bX$ with $\bU$ orientable. If $\bX$ is an oriented s-manifold, we write $-\bX$ for $\bX$ with the opposite orientation.
\end{dfn}

\begin{dfn}
\label{sm3def3}
Let $\bX$ be an s-manifold of dimension $n$, and use the notation of Definition \ref{sm3def2}. An {\it orientation bundle\/} $(\Or_\bX,\om_\bX)$ on $\bX$ consists of:
\begin{itemize}
\setlength{\itemsep}{0pt}
\setlength{\parsep}{0pt}
\item[(a)] a topological principal $\Z_2$-bundle $\pi:\Or_\bX\ra X$ over $X$, and
\item[(b)] an isomorphism $\om_\bX:\Or_{X_0}\ra\Or_\bX\vert_{X_0}$ of principal $\Z_2$-bundles over $X_0\subseteq X$, where $\pi:\Or_{X_0}\ra X_0$ is the principal $\Z_2$-bundle of orientations of the smooth $n$-manifold $X_0$, whose fibre over $x_0\in X_0$ is the $\Z_2$-torsor of orientations on $T_{x_0}X_0$, such that
\item[(c)] By Property \ref{sm2pr2}(g) for Steenrod homology, as $X_1$ is closed in $X_{\le 1}$ with $X_0=X_{\le 1}\sm X_1$, we have a long exact sequence
\e
\begin{gathered}
\xymatrix@C=25pt@R=15pt{ \cdots \ar[r] & H_n^\St(X_1,\Or_\bX\vert_{X_1},\Z) \ar[r]_(0.47){\inc_*} & H_n^\St(X_{\le 1},\Or_\bX\vert_{X_{\le 1}},\Z) \ar[d]^{\vert_{X_0}} \\
\cdots  & H_{n-1}^\St(X_1,\Or_\bX\vert_{X_1},\Z) \ar[l] & H_n^\St(X_0,\Or_\bX\vert_{X_0},\Z). \ar[l]_(0.49)\pd }
\end{gathered}
\label{sm3eq4}
\e
The $n$-manifold $X_0$ has a fundamental class $[X_0]_\fund\in H_n^\St(X_0,\Or_{X_0},\Z)$ by Property \ref{sm2pr2}(e). Applying $\om_\bX$ in (h) gives an isomorphism
\e
(\om_\bX)_*:H_n^\St(X_0,\Or_{X_0},\Z)\longra H_n^\St(X_0,\Or_\bX\vert_{X_0},\Z).
\label{sm3eq5}
\e
We require that
\e
\pd\ci (\om_\bX)_*\bigl([X_0]_\fund\bigr)=0\qquad\text{in $H_{n-1}^\St(X_1,\Or_\bX\vert_{X_1},\Z)$,}
\label{sm3eq6}
\e
where $\pd,(\om_\bX)_*$ are as in \eq{sm3eq4}--\eq{sm3eq5}.
\end{itemize}

Thus, $\Or_\bX$ is an extension of the principal $\Z_2$-bundle $\Or_{X_0}$ from $X_0$ to $X$, satisfying a condition over $X_{\le 1}$. Note that (c) is a local condition on $\bX$, that is, it is enough for it to hold for any family of arbitrarily small open sets $U\subseteq X$ covering $X_{\le 1}$. It is easy to see that:
\begin{itemize}
\setlength{\itemsep}{0pt}
\setlength{\parsep}{0pt}
\item[(i)] If an s-manifold $\bX$ has an orientation bundle $(\Or_\bX,\om_\bX)$ then it is locally orientable. 
\item[(ii)] If $\bX$ is oriented then it has an orientation bundle $(\Or_\bX,\om_\bX)$ with $\Or_\bX=X\t\Z_2$, and $\om_\bX:\Or_{X_0}\ra\Or_\bX\vert_{X_0}=X_0\t\Z_2$ the trivialization of $\Or_{X_0}$ induced by the orientation on $X_0$, as \eq{sm3eq3} implies \eq{sm3eq6}.
\item[(iii)] If $\bX$ has an orientation bundle $(\Or_\bX,\om_\bX)$, any trivialization $\Or_\bX\cong X\t\Z_2$ induces an orientation on $\bX$, as \eq{sm3eq6} implies~\eq{sm3eq3}.
\end{itemize}

However, {\it beware\/} that in pathological examples such as Example \ref{sm3ex4}, given an orientation bundle $(\Or_\bX,\om_\bX)$, there is {\it not\/} a 1-1 correspondence between trivializations $\Or_\bX\cong X\t\Z_2$ and orientations on $\bX$: an orientation on $\bX$ gives a trivialization of $\Or_\bX$ on $X_0\subseteq X$, which need not extend to  a trivialization on $X$, and if it does extend, it might not extend uniquely. For the same reason, in pathological examples, orientation bundles $(\Or_\bX,\om_\bX)$ on $\bX$ may not be unique up to isomorphism. See Example \ref{sm3ex6}(d) below for more on this.

We say that an orientation on $\bX$ is {\it compatible with\/} $(\Or_\bX,\om_\bX)$ if it corresponds to a trivialization~$\Or_\bX\cong X\t\Z_2$. 
\end{dfn}

Theorem \ref{sm3thm1} will define fundamental classes $[\bX]_\fund$ for s-manifolds $\bX$ with an orientation or an orientation bundle. First we prove a vanishing result for Steenrod homology.

\begin{prop}
\label{sm3prop1}
Let\/ $\bX$ be an s-manifold of dimension $n,$ and use the notation of Definition\/ {\rm\ref{sm3def1}}. Then for each commutative ring $R$ we have
\e
H_k^\St(X,R)=0\quad\text{if\/ $k>n,$}\quad
H_k^\St(X_{\ge 2},R)=0\quad\text{if\/ $k>n-2$.}
\label{sm3eq7}
\e	
If\/ $\bX$ has an orientation bundle $(\Or_\bX,\om_\bX)$ then
\e
H_k^\St(X,\Or_\bX,R)=0\;\text{if\/ $k>n,$}\;
H_k^\St(X_{\ge 2},\Or_\bX\vert_{X_{\ge 2}},R)=0\;\text{if\/ $k>n-2$.}
\label{sm3eq8}
\e	
\end{prop}

\begin{proof}
Each $x\in X$ has an open neighbourhood $V$ satisfying Definition \ref{sm3def1}(e). We first show that  $H_k^\St(V,R)=0$ if $k>n$. Choose a total order $\le$ on $I_V$ compatible with the partial order $\pr$, and write $I_V=\{i_1,i_2,\ldots,i_M\}$ in the order $\le$, so that $i_j\pr i_k$ implies that $j\le k$. Hence $\ov{V\cap X^{i_j}}\subseteq \coprod_{k=j}^MV\cap X^{i_k}$ by Definition \ref{sm3def1}(e)(i), so $V\cap(X^{i_1}\amalg X^{i_2}\amalg\cdots\amalg X^{i_m})$ is open in $V$ for all $m=1,2,\ldots,M$. Thus $V\cap X^{i_m}$ is closed in $V\cap(X^{i_1}\amalg X^{i_2}\amalg\cdots\amalg X^{i_m})$ with $(V\cap (X^{i_1}\amalg\cdots\amalg X^{i_m}))\sm (V\cap X^{i_m})=V\cap (X^{i_1}\amalg\cdots\amalg X^{i_{m-1}})$. Hence for each commutative ring $R$, Property \ref{sm2pr2}(g) gives a long exact sequence
\e
\begin{gathered}
\xymatrix@C=25pt@R=15pt{ \cdots \ar[r] & H_k^\St\bigl(V\cap X^{i_m},R\bigr) \ar[r]_(0.45){\inc_*} & H_k^\St\bigl(V\cap(\bigcup_{j=1}^mX^{i_j}),R\bigr) \ar[d]^{\vert_{V\cap(\bigcup_{j=1}^{m-1}X^{i_j})}} \\
\cdots  & H_{k-1}^\St\bigl(V\cap X^{i_m},R\bigr) \ar[l] & H_k^\St\bigl(V\cap(\bigcup_{j=1}^{m-1}X^{i_j}),R\bigr). \ar[l]_(0.55)\pd }
\end{gathered}
\label{sm3eq9}
\e

Now $H_*^\St(V\cap X^{i_m},R)\cong H_*^\BM(V\cap X^{i_m},R)$ is the usual Borel--Moore homology of the manifold $V\cap X^{i_m}$ by Property \ref{sm2pr2}(a), where $\dim V\cap X^{i_m}\le n$ by Definition \ref{sm3def1}(c), so $H_*^\St(V\cap X^{i_m},R)=0$ for $k>n$. Thus we prove by induction on $m=1,\ldots,M$ that $H_k^\St\bigl(V\cap (\bigcup_{j=1}^mX^{i_j}),R\bigr)=0$ for $k>n$, where the first step holds as $V\cap X^{i_1}$ is a manifold of dimension $\le n$, and the inductive step holds by \eq{sm3eq9} and $H_*^\St(V\cap X^{i_m},R)=0$ for $k>n$. When $m=M$ this shows that $H_k^\St(V,R)=0$ when $k>n$. 

As $X$ is second countable, we can choose a countable cover of $X$ by open sets $V$ satisfying Definition \ref{sm3def1}(e). Write this as $\{V_1,V_2,\ldots\}$. For each $m=2,3,\ldots,$ Property \ref{sm2pr2}(f) gives a long exact sequence
\e
\begin{gathered}
\!\!\!\!\!\!\!\!\!\!\!\xymatrix@C=20pt@R=11pt{
\cdots \ar[r] & H_{k+1}\bigl((\bigcup_{j=1}^{m-1}V_j)\cap V_m,R\bigr) \ar[rrr] &&& *+[l]{H_k(\bigcup_{j=1}^mV_j,R)} \ar[d] \\
\cdots & H_k\bigl((\bigcup_{j=1}^{m-1}V_j)\cap V_m,R\bigr) \ar[l] &&& *+[l]{H_k(\bigcup_{j=1}^{m-1}V_j,R)\op H_k(V_m,R).} \ar[lll]    }
\end{gathered}\!\!\!\!\!\!\!\!
\label{sm3eq10}
\e
Here $H_k(V_m,R)=0$ for $k>n$ by the first part. Also $H_{k+1}\bigl((\bigcup_{j=1}^{m-1}V_j)\cap V_m,R\bigr)=H_k\bigl((\bigcup_{j=1}^{m-1}V_j)\cap V_m,R\bigr)=0$ by the first part with $(\bigcup_{j=1}^{m-1}V_j)\cap V_m$ in place of $V_m$. Hence \eq{sm3eq10} shows that if $H_k(\bigcup_{j=1}^{m-1}V_j,R)=0$ for $k>n$ then $H_k(\bigcup_{j=1}^mV_j,R)=0$ for $k>n$. Thus by induction on $m=1,2,\ldots$ we see that $H_k^\St\bigl(\bigcup_{j=1}^mV_j,R\bigr)=0$ for $k>n$ and all $m\ge 1$.

Applying Property \ref{sm2pr2}(h)(ii) to the increasing sequence $\bigl(\bigcup_{j=1}^mV_j\bigr)_{m\ge 1}$ of open subsets of $X$ with $\bigcup_{m\ge 1}\bigl(\bigcup_{j=1}^mV_j\bigr)=X$ gives an exact sequence
\begin{equation*}
\xymatrix@C=10pt{ 0 \ar[r] & \varprojlim_{m\ra\iy}^1 H_{k+1}^\St\bigl(\bigcup\limits_{j=1}^mV_j,R\bigr) \ar[r] & H_k^\St\bigl(X,R\bigr) \ar[r] & 
\varprojlim_{m\ra\iy} H_k^\St\bigl(\bigcup\limits_{j=1}^mV_j,R\bigr) \ar[r] & 0 }
\end{equation*}
for each $k$. When $k>n$ the terms in the direct limits are zero, so $H_k^\St\bigl(X,R\bigr)=0$. This proves the first equation of \eq{sm3eq7}. For the second, $X_{\ge 2}$ is the union of all strata $X^{i_j}$ with $\dim X^{i_j}\le n-2$, so we can use the same arguments with $k>n-2$ rather than $k>n$. For \eq{sm3eq8} we use the same arguments in Steenrod homology twisted by~$\Or_\bX$.
\end{proof}

The next theorem is now immediate from Proposition \ref{sm3prop1}.

\begin{thm}
\label{sm3thm1}
{\bf(a)} Let\/ $\bX$ be an oriented s-manifold of dimension $n,$ and use the notation of Definition\/ {\rm\ref{sm3def1}}. As $X_{\ge 2}\subseteq X$ is closed with\/ $X_{\le 1}=X\sm X_{\ge 2},$ and\/ $X_1\subseteq X_{\le 1}$ is closed with\/ $X_0=X_{\le 1}\sm X_1,$ by Property\/ {\rm\ref{sm2pr2}(g)} for Steenrod homology, {\rm\eq{sm3eq7},} and\/ $H_n^\St(X_1,\Z)=0$ we have exact sequences
\e
\begin{gathered}
\xymatrix@C=25pt{  0 \ar[r] & H_n^\St(X,\Z) \ar[rr]^(0.45){\vert_{X_{\le 1}}} && H_n^\St(X_{\le 1},\Z) \ar[r] & 0, }\\
\xymatrix@C=20pt{  0 \ar[r] & H_n^\St(X_{\le 1},\Z) \ar[r]^(0.52){\vert_{X_0}} & 
H_n^\St(X_0,\Z) \ar[r]^(0.48)\pd & H_{n-1}^\St(X_1,\Z). }
\end{gathered}
\label{sm3eq11}
\e
Combining these gives an exact sequence
\e
\xymatrix@C=20pt{  0 \ar[r] & H_n^\St(X,\Z) \ar[r]^(0.45){\vert_{X_0}} & 
H_n^\St(X_0,\Z) \ar[r]^(0.47)\pd & H_{n-1}^\St(X_1,\Z). }
\label{sm3eq12}
\e
As $X_0$ is oriented we have a fundamental class $[X_0]_\fund\in H_n^\St(X_0,\Z)$ with $\pd\bigl([X_0]_\fund\bigr)=0$ by \eq{sm3eq3}. Hence there is a unique class $[\bX]_\fund$ in the Steenrod homology group\/ $H_n^\St(X,\Z)$ such that  
\e
[\bX]_\fund\vert_{X_0}=[X_0]_\fund.
\label{sm3eq13}
\e
We call\/ $[\bX]_\fund$ the \begin{bfseries}fundamental class\end{bfseries} of\/ $\bX$. The existence of some $[\bX]_\fund$ satisfying \eq{sm3eq13} is equivalent to \eq{sm3eq3}.
\smallskip

\noindent{\bf(b)} Let\/ $\bX$ be a dimension $n$ s-manifold with an orientation bundle $(\Or_\bX,\om_\bX)$. As for \eq{sm3eq12} but twisted by $\Or_\bX,$ using {\rm\eq{sm3eq8},} we have an exact sequence
\e
\text{\begin{small}$
\xymatrix@C=13pt{  0 \ar[r] & H_n^\St(X,\Or_\bX,\Z) \ar[r]^(0.45){\vert_{X_0}} & 
H_n^\St(X_0,\Or_\bX\vert_{X_0},\Z) \ar[r]^(0.47)\pd & H_{n-1}^\St(X_1,\Or_\bX\vert_{X_1},\Z). }$\end{small}}
\label{sm3eq14}
\e
Equation \eq{sm3eq6} now implies that there is a unique class $[\bX]_\fund$ in the twisted Steenrod homology group\/ $H_n^\St(X,\Or_\bX,\Z)$ such that  
\e
[\bX]_\fund\vert_{X_0}=(\om_\bX)_*\bigl([X_0]_\fund\bigr).
\label{sm3eq15}
\e
We call\/ $[\bX]_\fund$ the \begin{bfseries}fundamental class\end{bfseries} of\/ $\bX$. The existence of some $[\bX]_\fund$ satisfying \eq{sm3eq15} is equivalent to~\eq{sm3eq6}.

A trivialization $\Or_\bX\cong X\t\Z_2$ induces an orientation on $\bX,$ and then the fundamental classes $[\bX]_\fund$ from {\bf(a)\rm,\bf(b)} agree in $H_n^\St(X,\Or_\bX,\Z)\!\cong\! H_n^\St(X,\Z)$.
\end{thm}

\begin{ex}
\label{sm3ex6}
{\bf(a) (Ordinary manifolds.)} Let $X$ be an $n$-manifold, made into an s-manifold $\bX$ as in Definition \ref{sm3def1}. Then $X_0=X$ and $X_1=\es$, so equations \eq{sm3eq3} and \eq{sm3eq6} are trivial. Orientations on $X$ and on $\bX$ are equivalent, and then $[\bX]_\fund$ in Theorem \ref{sm3thm1}(a) is the usual oriented fundamental class $[X]_\fund$ of $X$ from Property \ref{sm2pr2}(e). Also $\bX$ has a canonical orientation bundle $(\Or_{X_0},\id_{\Or_{X_0}}),$ and $[\bX]_\fund$ in Theorem \ref{sm3thm1}(a) is the usual twisted fundamental class $[X]_\fund$ of $X$ from Property~\ref{sm2pr2}(j).
\smallskip

\noindent{\bf(b) (Manifolds with corners.)} Let $X$ be an $n$-manifold with corners, or g-corners, or a-corners, with $\pd X\ne\es$, and let $\bX$ be the corresponding s-manifold from Example \ref{sm3ex1}. Then $\bX$ is not orientable, or locally orientable, and no orientation bundle $(\Or_\bX,\om_\bX)$ exists on $\bX$, so Theorem \ref{sm3thm1} does not apply. This is because with appropriate choices of orientations on $X_0,X_1$ we have $\pd\bigl([X_0]_\fund\bigr)=[X_1]_\fund\ne 0$, so \eq{sm3eq3} and \eq{sm3eq6} cannot hold.
\smallskip

\noindent{\bf(c) (Orbifolds.)} Let $X$ be an $n$-orbifold, and $\bX$ be the corresponding s-manifold from Example \ref{sm3ex3}. We call $X$ {\it locally orientable\/} if is locally modelled on $\R^n/G$ for $G\subset\SL(n,\R)$, that is, $G$ should preserve the orientation on~$\R^n$.

Then $\bX$ is locally orientable as an s-manifold if and only if $X$ is locally orientable as an orbifold, and if this happens there is a canonical orientation bundle $(\Or_\bX,\om_\bX)$ on $\bX$ with $\Or_\bX$ the principal $\Z_2$-bundle of local orientations on $\bX$. The locally orientable condition on $X$ also implies that there are no strata $X^{[G]}$ with $\dim X^{[G]}=n-1$, so $X_1=\es$ and \eq{sm3eq3} and \eq{sm3eq6} are trivial.
\smallskip

\noindent{\bf(d) (Example \ref{sm3ex4}.)} Let $\bX$ be the s-manifold of dimension $n\ge 1$ defined in Example \ref{sm3ex4}. We can show that the singular homology $H_*^\rsi(X,\Z)$ and Steenrod homology $H_*^\St(X,\Z)$ are different: we have
\e
H_n^\rsi(X,\Z)\cong z\Z[z],\qquad H_n^\St(X,\Z)\cong z\Z[[z]],
\label{sm3eq16}
\e
where $z^k$ for $k\ge 1$ corresponds to the fundamental class $[\cS^n_{1/k}]_\fund$ of $\cS^n_{1/k}\subset X$, and the natural map $H_n^\rsi(X,\Z)\ra H_n^\St(X,\Z)$ to the inclusion $z\Z[z]\hookra z\Z[[z]]$. 

As $X_0=\coprod_{k\ge 1}\cS^n_{1/k}\sm \{(0,\ldots,0)\}$ is the disjoint union of infinitely many copies of $\R^n$, there are {\it infinitely many\/} orientations on $X_0$, corresponding to an independent choice of orientation on $\cS^n_{1/k}\sm \{(0,\ldots,0)\}$ for each~$k=1,2,\ldots.$ 

If $n>1$ then $X_1=\es$ so \eq{sm3eq3} is trivial, and all of these orientations on $X_0$ are orientations on $\bX$. If $n=1$ we can give a nontrivial proof that \eq{sm3eq3} holds for every orientation on $X_0$ by taking a limit in Steenrod homology. The basic idea is that for each $m=1,2,\ldots$ the open subset $\coprod_{k=1}^m\cS^1_{1/k}\sm\{(0,0)\}$ in $X_0$ contributes zero to $\pd\bigl([X_0]_\fund\bigr)$ in \eq{sm3eq3}, as the two ends of $\cS^1_{1/k}\sm\{(0,0)\}$ at $(0,0)$ contribute $+1,-1$ or $-1,+1$, depending the orientation on $\cS^1_{1/k}\sm\{(0,0)\}$.

The fundamental class $[\bX]_\fund$ of $\bX$ is then identified with $\sum_{k\ge 1}\pm z^k$ under \eq{sm3eq16}, where the sign of $z^k$ is determined by the choice of orientation on $\cS^n_{1/k}\sm \{(0,\ldots,0)\}$. Observe that $[\bX]_\fund$ lies in $H_n^\St(X,\Z)$ but not in $H_n^\rsi(X,\Z)$. This illustrates why we must define fundamental classes in Steenrod homology, not ordinary (singular) homology.

Each orientation $o$ on $X_0$ determines an orientation bundle $(\Or^o_\bX,\om^o_\bX)$ on $\bX$ as in Definition \ref{sm3def3}. Two orientations $o,o'$ yield isomorphic orientation bundles, $(\Or^o_\bX,\om^o_\bX)\cong(\Or^{o'}_\bX,\om^{o'}_\bX)$, if and only if $o=\pm o'$, that is, if the fundamental classes satisfy $[\bX]^o_\fund=\pm[\bX]^{o'}_\fund$. Hence there are infinitely many non-isomorphic orientation bundles on $\bX$. If $n>1$ then all orientation bundles on $\bX$ are trivializable, and come from orientations. If $n=1$ there also exist non-trivializable orientation bundles $(\Or_\bX,\om_\bX)$ on $X$, in which the monodromy of $\Or_\bX$ around an even finite number of circles $\cS^1_{1/k}$ is $-1$.

One moral is that for s-manifolds $\bX$ with topological spaces of infinite topological type, orientations and orientation bundles can be complicated.
\smallskip

\noindent{\bf(d) (Example \ref{sm3ex5}.)} Let $\bY$ be the s-manifold of dimension 1 in Example \ref{sm3ex5}. Then $\bY$ does not admit an orientation, or an orientation bundle, even locally, as we cannot satisfy \eq{sm3eq3} or \eq{sm3eq6} near $(-1,0)$ or $(0,0)$.

At $(-1,0)$ this is obvious: the curve $(-1,0)\t\{0\}\subset Y_0$ contributes $\pm 1$ to $\pd\bigl([Y_0]_\fund\bigr)$ at $(1,0)$, so $\pd\bigl([Y_0]_\fund\bigr)\ne 0$ at~$(-1,0)\in Y_1$.

At $(0,0)$ it is less obvious: both $X$ in Example \ref{sm3ex4} with $n=1$ and $Y$ in Example \ref{sm3ex5} have infinitely many curves ending at $(0,0)$, so why should Example \ref{sm3ex4} be orientable but Example \ref{sm3ex5} non-orientable near~$(0,0)$?

The explanation is that any small open neighbourhood $U$ of $(0,0)$ in $X$ or $Y$ will contain the whole of  
$\cS^1_{1/k}$ for $k>N$, for some $N\gg 0$. Then each $\cS^1_{1/k}$ for $k>N$ does not contribute to the orientability of $X$ or $Y$ at $(0,0)$, as the two ends of $\cS^1_{1/k}$ at $(0,0)$ cancel out. Thus we can determine orientability of $\bX$ or $\bY$ by considering $\cS^1_{1/k}$ for $k=1,\ldots,N$, and (for $Y$) $[-1,0]\t\{0\}$. These give $2N$ curves ending at $(0,0)$ for $X$, and $2N+1$ curves ending at $(0,0)$ for $Y$. As $2N$ is even we can choose orientations on these $2N$ curves so they contribute 0 to \eq{sm3eq3}, making $\bX$ locally orientable, but as $2N+1$ is odd the contribution to \eq{sm3eq3} at $(0,0)$ for $\bY$ must be odd, and $\bY$ is not locally orientable at $(0,0)$. In Steenrod homology one should take a limit as $U$ shrinks to $(0,0)$ and~$N\ra\iy$.
\end{ex}

\subsection{\texorpdfstring{$\Z_2$-fundamental classes}{ℤ₂-fundamental classes}}
\label{sm34}

If $X$ is an $n$-manifold, not necessarily orientable, it has a fundamental class $[X]_\fund$ in Borel--Moore  or Steenrod homology $H_n^\BM(X,\Z_2)\cong H_n^\St(X,\Z_2)$ over $\Z_2$. For the same thing to work for s-manifolds $\bX$ we need to impose a condition analogous to~\eq{sm3eq3}.

\begin{dfn}
\label{sm3def4}
Let $\bX$ be an s-manifold of dimension $n$, and define $X_0,X_1,\ab X_{\le 1}$ as in Definition \ref{sm3def1}. Then as for \eq{sm3eq2} we have a long exact sequence
\begin{equation*}
\xymatrix@C=18pt{ \cdots \ar[r] & H_n^\St(X_{\le 1},\Z_2) \ar[r]^{\vert_{X_0}} & H_n^\St(X_0,\Z_2) \ar[r]^(0.49)\pd & H_{n-1}^\St(X_1,\Z_2) \ar[r] & \cdots.  }
\end{equation*}
We say that $\bX$ {\it admits a $\Z_2$-fundamental class\/} if
\begin{equation*}
\pd\bigl([X_0]_\fund\bigr)=0\qquad\text{in $H_{n-1}^\St(X_1,\Z_2)$,}
\end{equation*}
with $[X_0]_\fund\in H_n^\St(X_0,\Z_2)$ the usual $\Z_2$-fundamental class. In Theorem \ref{sm3thm1}(a) the analogues of \eq{sm3eq11}--\eq{sm3eq12} work over $\Z_2$, giving an exact sequence
\begin{equation*}
\xymatrix@C=20pt{  0 \ar[r] & H_n^\St(X,\Z_2) \ar[r]^(0.45){\vert_{X_0}} & 
H_n^\St(X_0,\Z_2) \ar[r]^(0.47)\pd & H_{n-1}^\St(X_1,\Z_2). }
\end{equation*}
Hence if $\bX$ admits a $\Z_2$-fundamental class then there is a unique class $[\bX]_\fund$ in the Steenrod homology group $H_n^\St(X,\Z_2)$ such that  
\begin{equation*}
[\bX]_\fund\vert_{X_0}=[X_0]_\fund.
\end{equation*}
We call $[\bX]_\fund$ the $\Z_2$-{\it fundamental class\/} of~$\bX$.

If $X$ is an $n$-manifold then $U\mapsto H_n^\St(U,\Z_2)$ for open $U\subseteq X$ is a sheaf on $X$. Because of this, the condition that $\bX$ admits a $\Z_2$-fundamental class is a local condition on $X_{\le 1}$. If $\bX$ is locally orientable in the sense of Definition \ref{sm3def2} then it admits a $\Z_2$-fundamental class.
\end{dfn}

\subsection{Partitions of unity, a Whitney Embedding Theorem}
\label{sm35}

\begin{lem}
\label{sm3lem1}
Let $\bX$ be an s-manifold, $x\in X,$ and\/ $V$ be an open neighbourhood of $x$ in $X$. Then there exists a smooth map $f:\bX\ra\R$ in the sense of s-manifolds such that\/  $\ov{\{y\in X:f(y)\ne 0\}}\subseteq V\subseteq X,$ and\/ $f(x)>0,$ and\/ $f(X)\subseteq[0,\iy)$. We call\/ $f$ a \begin{bfseries}bump function\end{bfseries} on $\bX$.
\end{lem}

\begin{proof}
By Definition \ref{sm3def1}(d) we can choose an open neighbourhood $U$ of $x$ in $X$ and a smooth map $g:U\ra\R^N$ in the sense of s-manifolds for $N\gg 0$ with $g:U\ra g(U)$ a homeomorphism. Replacing $g$ by $g-g(x)$ we can suppose that $g(x)=0$. As $X$ and hence $U\cap V,g(U\cap V)$ are locally compact, we may choose small $\ep>0$ such that $g(U\cap V)\cap \,\ov{\!B}_{2\ep}(0)$ is compact, where $\,\ov{\!B}_{2\ep}(0)$ is the closed ball of radius $2\ep$ in $\R^N$. Now define $f:X\ra\R$ by
\begin{equation*}
f(y)=\begin{cases} e^{-1/(\ep^2-z_1^2-z_2^2-\cdots-z_N^2)}, & \begin{aligned}&\text{$y\in U\cap V$ and $g(y)=(z_1,\ldots,z_N)$} \\ &\text{with $z_1^2+z_2^2+\cdots+z_N^2<\ep^2,$} \end{aligned}\\
0, & \text{otherwise.}
\end{cases}
\end{equation*}
It is easy to check $f$ satisfies the conditions of the lemma.	
\end{proof}

As s-manifolds are locally compact, Hausdorff, and second countable, they are paracompact. Thus the standard proof of the existence of smooth partitions of unity on manifolds using bump functions, as in Bredon \cite[Th.~II.10.1]{Bred1} or Lee \cite[Th.~1.73]{Lee1}, \cite[Th.~2.25]{Lee2}, works for s-manifolds with essentially no changes, giving:

\begin{thm}
\label{sm3thm2}
Let\/ $\bX$ be an s-manifold, and\/ $\{U_a:a\in A\}$ an open cover of\/ $X$. Then there exist smooth functions $\eta_a:\bX\ra\R$ in the sense of s-manifolds for $a\in A$ satisfying:
\begin{itemize}
\setlength{\itemsep}{0pt}
\setlength{\parsep}{0pt}
\item[{\bf(i)}] $\eta_a(X)\subseteq [0,1]\subset\R$ for all\/ $a\in A$.
\item[{\bf(ii)}] $\supp\eta_a:=\ov{\{x\in X:\eta_a(x)\ne 0\}}\subseteq U_a\subseteq X$ for all\/ $a\in A$.
\item[{\bf(iii)}] Each $x\in X$ has an open neighbourhood\/ $U$ such that there are only finitely many $a\in A$ with\/~$\supp\eta_a\cap U\ne\es$.
\item[{\bf(iv)}] $\sum_{a\in A}\eta_a\equiv 1$, where the sum makes sense locally on $X$ by {\bf(iii)}.
\end{itemize}
We call\/ $\{\eta_a:a\in A\}$ a \begin{bfseries}partition of unity subordinate to\end{bfseries}~$\{U_a:a\in A\}$.
\end{thm}

\begin{cor}
\label{sm3cor1}
Let\/ $\bX$ be an s-manifold with\/ $\dim\bX>0$ and\/ $X_0\ne\es$. Then the $\R$-algebra $C^\iy(\bX)$ of smooth functions $f:\bX\ra\R$ is infinite-dimensional.
\end{cor}

\begin{proof}
Near each $x\in X_0$, the s-manifold $\bX$ looks locally like $\R^n$ for $n=\dim\bX$, and smooth functions $f:\bX\ra\R$ look locally like smooth functions $\R^n\ra\R$, which are an infinite-dimensional space. We can combine local choices of smooth functions $f:\bX\ra\R$ using partitions of unity in the usual way, to get a huge space $C^\iy(\bX)$ of global smooth functions~$f:\bX\ra\R$.
\end{proof}

Next we define embeddings. We only consider embeddings into $\R^N$, as it is not clear whether there is a useful notion of embedding between s-manifolds.

\begin{dfn}
\label{sm3def5}
Let $\bX$ be an s-manifold. A smooth map $f:\bX\ra\R^N$ is called an {\it embedding\/} if $f\vert_{X^i}:X^i\ra\R^N$ is an embedding of manifolds for all strata $X^i$ of $\bX$, and $f:X\ra f(X)$ is a homeomorphism.
\end{dfn}

Here is an s-manifold version of the Whitney Embedding Theorem \cite[\S II.10]{Bred1}, \cite[Th.~10.11]{Lee2}. Whitney also gives a more difficult proof that $n$-manifolds can be embedded into $\R^{2n}$, but we will not attempt to generalize this.

\begin{thm}
\label{sm3thm3}
Let\/ $\bX$ be an s-manifold of dimension $n$. Then there exists a proper embedding $f:\bX\ra\R^{2n+1}$.	
\end{thm}

\begin{proof}
This can be proved by following the proof for manifolds in Bredon \cite[\S II.10]{Bred1} closely. First suppose $\bX$ is compact. Definition \ref{sm3def1}(d) says every $x\in\bX$ has an open neighbourhood $\bU_x\subseteq\bX$ and an embedding $f_x:\bU_x\hookra\R^{N_x}$ for $N_x\gg 0$. We choose a finite cover $\bU_{x_1},\ldots\bU_{x_k}$ of $\bX$ by such neighbourhoods. Theorem \ref{sm3thm2} gives a subordinate partition of unity $\eta_{x_1},\ldots,\eta_{x_k}$. Then
\e
(\eta_{x_1},\ldots,\eta_{x_k},\eta_{x_1}f_{x_1},\ldots,\eta_{x_k}f_{x_k}):\bX\longra\R^N
\label{sm3eq17}
\e
is an embedding for $N=k+N_{x_1}+\cdots+N_{x_k}$. By an argument using Sard's Theorem as in \cite[Th.~II.10.7]{Bred1}, we find that if $\pi:\R^N\ra\R^{2n+1}$ is a generic linear map then the composition of \eq{sm3eq17} with $\pi$ is an embedding $\bX\hookra\R^{2n+1}$. This proves the theorem when $\bX$ is compact. We can extend to noncompact $\bX$ as in \cite[Th.~II.10.7]{Bred2}, which uses $X$ second countable.	
\end{proof}

\begin{cor}
\label{sm3cor2} 
The topological space $X$ of an s-manifold\/ $\bX$ is metrizable.	
\end{cor}

\begin{rem}
\label{sm3rem3}
As in Steenrod \cite{Stee1} and Milnor \cite{Miln}, if $X\subset\cS^{N+1}$ is compact then there are canonical isomorphisms
\begin{equation*}
H_k^\St(X,R)\cong \check H^{N-k}(\cS^{N+1}\sm X,R)\qquad \text{for $0<k<N$,}
\end{equation*}
where $\check H^*(-,R)$ is \v Cech cohomology. If $\bX$ is a compact s-manifold of dimension $n$ then Theorem \ref{sm3thm3} shows that $X$ is homeomorphic to a compact subset of $\cS^{N+1}$ for any $N\ge 2n$, so this gives an alternative interpretation of~$H_*^\St(X,R)$.	
\end{rem}

\subsection{Transverse fibre products of s-manifolds}
\label{sm36}

We generalize Definition \ref{sm2def3} and Theorem \ref{sm2thm1} to s-manifolds.

\begin{dfn}
\label{sm3def6}
Let $\bX,\bY,\bZ$ be s-manifolds with $\dim\bX=l$, $\dim\bY=m$, $\dim\bZ=n$, and $g:\bX\ra\bZ$, $h:\bY\ra\bZ$ be smooth maps. We say that $g,h$ are {\it transverse\/} if whenever $x\in\bX$ and $y\in\bY$ with $g(x)=h(y)=z\in\bZ$, using the notation of Definition \ref{sm3def1} we have:
\begin{itemize}
\setlength{\itemsep}{0pt}
\setlength{\parsep}{0pt}
\item[(a)] $T_xg\op T_yh:T_x\bX\op T_y\bY\ra T_z\bZ$ is surjective; and
\item[(b)] Either:
\begin{itemize}
\setlength{\itemsep}{0pt}
\setlength{\parsep}{0pt}
\item[(i)] $\dim_x\bX=l$, $\dim_y\bY=m$, and $\dim_z\bZ=n$;
\item[(ii)] $\dim_x\bX=l-1$, $\dim_y\bY=m$, and $\dim_z\bZ=n$;
\item[(iii)] $\dim_x\bX=l$, $\dim_y\bY=m-1$, and $\dim_z\bZ=n$;
\item[(iv)] $\dim_x\bX=l-1$, $\dim_y\bY=m-1$, and $\dim_z\bZ=n-1$; or
\item[(v)] $\dim_x\bX+\dim_y\bY-\dim_z\bZ\le l+m-n-2$.
\end{itemize}
\end{itemize}

Note that as (a) forces $\dim_x\bX+\dim_y\bY-\dim_z\bZ\ge 0$, (b) shows that $l+m-n\ge 0$ if any such $x,y,z$ exist. We call $h:\bY\ra\bZ$ a {\it submersion\/} if whenever $y\in\bY$ with $h(y)=z\in\bZ$ we have:
\begin{itemize}
\setlength{\itemsep}{0pt}
\setlength{\parsep}{0pt}
\item[(a$)'$] $T_y\bY\ra T_z\bZ$ is surjective; and
\item[(b$)'$] Either:
\begin{itemize}
\setlength{\itemsep}{0pt}
\setlength{\parsep}{0pt}
\item[(i)] $\dim_z\bZ=n$; or
\item[(ii)] $\dim_y\bY-\dim_z\bZ\le m-n-2$.
\end{itemize}
\end{itemize}
Then (a$)'$,(b$)'$ imply (a),(b), so if $h$ is a submersion then $g,h$ are transverse for any $g:\bX\ra\bZ$.
\end{dfn}

\begin{rem}
\label{sm3rem4}
A useful special case of Definition \ref{sm3def6} is when $\bZ$ is an ordinary $n$-manifold. Then (b),(b$)'$ are automatic, so $g:\bX\ra\bZ$, $h:\bY\ra\bZ$ are {\it transverse\/} if $T_xg\op T_yh:T_x\bX\op T_y\bY\ra T_z\bZ$ is surjective whenever $x\in\bX$ and $y\in\bY$ with $g(x)=h(y)=z\in\bZ$, exactly as in Definition \ref{sm2def3}, and $h:\bY\ra\bZ$ is a {\it submersion\/} if $T_yh:T_y\bY\ra T_z\bZ$ is surjective whenever $y\in\bY$ with~$h(y)=z\in\bZ$.
\end{rem}

We will show in the next definition and theorem that if $g,h$ are transverse then a fibre product $\bW=\bX\t_{g,\bZ,h}\bY$ exists in~$\SMan$.

\begin{dfn}
\label{sm3def7}
Work in the situation of Definition \ref{sm3def6}, with $g,h$ transverse, and use the notation
\begin{equation*}
\bX=\bigl(X,\{X^i:i\in I\}\bigr),\quad
\bY=\bigl(Y,\{Y^j:j\in J\}\bigr),\quad
\bZ=\bigl(Z,\{Z^k:k\in K\}\bigr).
\end{equation*}

Define a topological space $W\subseteq X\t Y$, with the subspace topology, by
\e
W=\bigl\{(x,y)\in X\t Y:g(x)=h(y)\bigr\}.
\label{sm3eq18}
\e
Define $e:W\ra X$ and $f:W\ra Y$ to map $e:(x,y)\mapsto x$ and $f:(x,y)\mapsto y$. For all $(i,j,k)\in I\t J\t K$ with $g(X^i),h(Y^j)\subseteq Z^k$, define
\begin{equation*}
W^{ijk}=\bigl\{(x,y)\in W:x\in X^i,\; y\in Y^j,\; g(x)=h(y)\in Z^k\bigr\}.
\end{equation*}
Observe that
\e
W^{ijk}=X^i\t_{g\vert_{X^i},Z^k,h\vert_{Y^j}}Y^j,
\label{sm3eq19}
\e
which is a transverse fibre product, so $W^{ijk}$ is a smooth manifold of dimension $\dim W^i+\dim Y^j-\dim Z^k$ by Theorem \ref{sm2thm1}. Also $W^{ijk}$ is locally closed in $W$, as $X^i,Y^j,Z^k$ are in $X,Y,Z$.

Define an indexing set
\begin{equation*}
H=\bigl\{(i,j,k)\in I\t J\t K:g(X^i),h(Y^j)\subseteq Z^k,\; W^{ijk}\ne\es\bigr\}.	
\end{equation*}
Then we have a stratification of $W$ into nonempty locally closed subsets
\begin{equation*}
W=\coprod_{(i,j,k)\in H}W^{ijk}.
\end{equation*}
Define $\bW=\bigl(W,\{W^{ijk}:(i,j,k)\in H\}\bigr)$.

Let $W^{ijk}$ be one such stratum, pick $(x,y)\in W^{ijk}$, and set $z=g(x)=h(y)$. Set $d=l+m-n$, where we will have $\dim\bW=d$. Then $T_x\bX=T_xX^i$, $T_y\bY=T_yY^j$, and $T_z\bZ=T_zZ^k$. Thus Definition \ref{sm3def6}(b) becomes either:
\begin{itemize}
\setlength{\itemsep}{0pt}
\setlength{\parsep}{0pt}
\item[(i)] $\dim X^i=l$, $\dim Y^j=m$, and $\dim Z^k=n$, so $\dim W^{ijk}=d$;
\item[(ii)] $\dim X^i=l-1$, $\dim Y^j=m$, and $\dim Z^k=n$, so $\dim W^{ijk}=d-1$;
\item[(iii)] $\dim X^i=l$, $\dim Y^j=m-1$, and $\dim Z^k=n$, so $\dim W^{ijk}=d-1$;
\item[(iv)] $\dim X^i\!=\!l\!-\!1$, $\dim Y^j\!=\!m\!-\!1$, and $\dim Z^k\!=\!n\!-\!1$, so $\dim W^{ijk}\!=\!d\!-\!1$; or
\item[(v)] $\dim X^i+\dim Y^j-\dim Z^k\le l+m-n-2$, so $\dim W^{ijk}\le d-2$.
\end{itemize}
Hence $\dim W^{ijk}\le d$ for all strata $W^{ijk}$, as required in Definition \ref{sm3def1}(c). 
\end{dfn}

\begin{thm}
\label{sm3thm4}
In Definition\/ {\rm\ref{sm3def7},} $\bW$ is an s-manifold with $\dim\bW=d:=l+m-n,$ where $d\ge 0$ if\/ $\bW\ne\es,$ and\/ $e:\bW\ra\bX,$ $f:\bW\ra\bY$ are smooth maps with\/ $g\ci e=h\ci f$. The following is a Cartesian square in {\rm$\SMan$:}
\e
\begin{gathered}
\xymatrix@R=13pt@C=90pt{ *+[r]{\bW} \ar[r]_f \ar[d]^e & *+[l]{\bY}
\ar[d]_h \\ *+[r]{\bX} \ar[r]^g & *+[l]{\bZ.\!} }
\end{gathered}
\label{sm3eq20}
\e
Thus $\bW$ is a fibre product\/ $\bX\t_{g,\bZ,h}\bY$ in $\SMan$.
\end{thm}

\begin{proof}
First we verify $\bW$ satisfies Definition \ref{sm3def1}(a)--(e). For (a), $W$ is locally compact, Hausdorff and second countable as $X,Y$ are and $W$ is a closed subspace of $X\t Y$. We have proved (b),(c) in Definition \ref{sm3def7}; that $H$ is finite or countable in (b) follows from the fact that $I,J,K$ are. For (d), if $(x,y)\in W$ then Definition \ref{sm3def1}(d) for $\bX,\bY$ gives open neighbourhoods $U_x,U_y$ of $x,y$ in $X,Y$ and continuous maps $f_x:U_x\hookra\R^N$, $f_y:U_y\hookra\R^{N'}$ that are topological embeddings and restrict to smooth embeddings on strata. Define $U_{x,y}=(U_x\t U_y)\cap W$ and $f_{x,y}:U_{x,y}\ra\R^{N+N'}=\R^N\op\R^{N'}$ by $f_{x,y}(x',y')=f_x(x')\op f_y(y')$. Then $U_{x,y},f_{x,y}$ satisfy the conditions of~(d).

For the first part of (e), if $(x,y)\in W$, we can choose open neighbourhoods $V_x,V_y$ of $x,y$ in $X,Y$ which intersect only finitely many strata of $\bX,\bY$ by Definition \ref{sm3def1}(e), and then $V_{x,y}=W\cap(V_x\t V_y)$ is an open neighbourhood of $(x,y)$ in $W$ which intersects only finitely many strata $W^{ijk}$ of~$\bW$.

For the second part of (e), let $\pr,\dot\pr,\ddot\pr$ be the partial orders on $I_{V_x},J_{V_y},K_{V_z}$ given by Definition \ref{sm3def1}(e) for $\bX,\bY,\bZ$. Define a partial order $\dddot\pr$ on $H_{V_{x,y}}$ by $(i,j,k)\dddot\pr(i',j',k')$ if $i\pr i'$, $j\dot\pr j'$, and $k\ddot\pr k'$. It is easy to see that $\dddot\pr$ is a partial order. To prove (e)(i) for $\bW$, suppose that $(x_a,y_a)_{a=1}^\iy$ is a sequence in $W^{ijk}$ with $(x_a,y_a)\ra(x,y)\in W^{i'j'k'}$ as $a\ra\iy$. Then $x_a\in X^i$, $y_a\in Y^j$ with $g(x_a)=h(y_a)=z_a\in Z^k$, and $x\in X^{i'}$, $y\in Y^{j'}$ with $g(x)=h(y)=z\in Z^{k'}$, and $x_a\ra x$, $y_a\ra y$, $z_a\ra z$ as $a\ra\iy$. By Definition \ref{sm3def1}(e)(i) for $\bX,\bY,\bZ$ we have $i\pr i'$, $j\dot\pr j'$ and $k\ddot\pr k'$. Thus $(i,j,k)\dddot\pr(i',j',k')$, proving (e)(i) for~$\bW$.

Part (e)(ii) for $\bW$ follows from the conditions Definition \ref{sm3def6}(b)(i)--(v), since if $\dim W^{ijk}=d$ then Definition \ref{sm3def7}(i) holds, and if $\dim W^{ijk}=d-1$ then Definition \ref{sm3def7}(ii),(iii) or (iv) holds, and otherwise $\dim W^{ijk}\le d-2$ as in Definition \ref{sm3def7}(v). We can then deduce (e)(ii) from Definition \ref{sm3def7}(i)--(v) and (e)(ii) for $\bX,\bY,\bZ$. Note in particular that if $(i,j,k)$ is as in Definition \ref{sm3def7}(ii) or (iii), and $(i',j',k')$ as in Definition \ref{sm3def7}(iv), so that $\dim W^{ijk}=\dim W^{i'j'k'}=d-1$, then $(i,j,k)\dddot{\not\pr}(i',j',k')$, as if $i\pr i'$, $j\dot\pr j'$, $k\ddot\pr k'$ we must have $i=i'$ or $j=j'$ but $k\ne k'$, a contradiction. Hence $\bW$ is an s-manifold.

As in Definition \ref{sm3def7} $\dim W^{ijk}\le d$ for all strata $W^{ijk}$, so $d\ge 0$ if $\bW\ne\es$. It is immediate from Definition \ref{sm3def7} that $e,f$ are smooth maps of s-manifolds with $g\ci e=h\ci f$. Suppose $\bW'$ is an s-manifold and $e':\bW'\ra\bX$ and $f':\bW'\ra\bY$ are smooth with $g\ci e'=h\ci f'$. Define $b:W'\ra W$ by $b(w')=(e'(w'),f'(w'))$. This is well defined by \eq{sm3eq18} and $g\ci e'=h\ci f'$, and clearly continuous. Suppose $W^{\prime h}$ is a stratum of $\bW'$. Then there are unique $i\in I$, $j\in J$ and $k\in K$ such that $e'(W^{\prime h})\subseteq X^i$, $f'(W^{\prime h})\subseteq Y^j$, and $g(X^i),h(Y^j)\subseteq Z^k$, and $e'\vert_{W^{\prime h}}:W^{\prime h}\ra X^i$, $f'\vert_{W^{\prime h}}:W^{\prime h}\ra Y^j$ are smooth maps of manifolds. Thus as $W^{ijk}$ is the fibre product \eq{sm3eq19}, the map $b\vert_{W^{\prime h}}:W^{\prime h}\ra W^{ijk}$ is a smooth map of manifolds. Hence $b:\bW'\ra\bW$ is a smooth map of s-manifolds, with $e'=e\ci b$ and $f'=f\ci b$. It is unique with this property, as $b(w')=(e'(w'),f'(w'))$ is determined pointwise by $e',f'$. Therefore $\bW$ is a fibre product $\bX\t_{g,\bZ,h}\bY$ in $\SMan$, and \eq{sm3eq20} is a Cartesian square in~$\SMan$.
\end{proof}

\begin{rem}
\label{sm3rem5}
We can generalize the material above to fibre products in $\cSMan$ as follows: let $g:\bX\ra\bZ$, $h:\bY\ra\bZ$ be morphisms in $\cSMan$, and write $\bX=\coprod_{l\ge 0}\bX_l$ for $\bX_l\in\SMan$ with $\dim\bX=l$, and similarly for $\bY,\bZ$. Then we say that $g,h$ are {\it transverse\/} if $g\vert_{\cdots}:\bX_l\cap g^{-1}(\bZ_n)\ra\bZ_n$ and $h\vert_{\cdots}:\bY_m\cap h^{-1}(\bZ_n)\ra\bZ_n$ are transverse in $\SMan$ for all $l,m,n\ge 0$. Then the analogue of Theorem \ref{sm3thm4} holds in $\cSMan$, with fibre product
\begin{equation*}
\bW=\coprod_{l,m,n\ge 0}\bigl(\bX_l\cap g^{-1}(\bZ_n)\bigr)\t_{g\vert_{\cdots},\bZ_n,h\vert_{\cdots}}\bigl(\bY_m\cap h^{-1}(\bZ_n)\bigr).
\end{equation*}
\end{rem}

Note that when $\bZ=*$ is a point, Theorem \ref{sm3thm4} shows that {\it products\/} $\bX\t\bY$ of s-manifolds exist in $\SMan$, with $\dim(\bX\t\bY)=\dim\bX+\dim\bY$. Using Sard's Theorem we deduce:

\begin{cor}
\label{sm3cor3}
Let\/ $\bY$ be an s-manifold of dimension $m,$ and\/ $h:\bY\ra\R^n$ be a smooth map. Then for generic $\bs z\in\R^n$ (i.e.\ for $\bs z$ outside a null set) the level set\/ $h^{-1}(\bs z)$ has the structure of an s-manifold of dimension $m-n$.
\end{cor}

\begin{proof}
Apply Theorem \ref{sm3thm4} with $\bX=*$ the point, $\bZ=\R^n$, and $g:*\mapsto\bs z$. To show that $g,h$ are transverse, note that Definition \ref{sm3def6}(b) is automatic, and using Sard's Theorem for the smooth maps $g\vert_{Y^j}:Y^j\ra\R^n$ for the countable set of strata $Y^j$ of $\bY$ we see that Definition \ref{sm3def6}(a) holds for generic~$\bs z\in\R^n$.	
\end{proof}

\subsection{\texorpdfstring{Oriented transverse fibre products in $\SMan$}{Oriented transverse fibre products in SMan}}
\label{sm37}

Let us now consider the question: suppose $\bW=\bX\t_{g,\bZ,h}\bY$ is a transverse fibre product in $\SMan$. If $\bX,\bY,\bZ$ are oriented (or have orientation bundles), is $\bW$ canonically oriented (or has a canonical orientation bundle)?

In general the answer is no, because of difficulties in satisfying \eq{sm3eq3} and \eq{sm3eq6} for $\bW$. The next example illustrates this.

\begin{ex}
\label{sm3ex7}
Let $\bZ$ be the oriented s-manifold of dimension 1 drawn in Figure \ref{sm3fig3}, with four 1-dimensional strata diffeomorphic to $\R$, labelled $a,b,c,d$, with orientations given by the arrows, meeting at a 0-dimensional stratum, the central vertex $v$. Equation \eq{sm3eq3} for $\bZ$ becomes $-1+1-1+1=0$, listing the contributions from $a,b,c,d$.

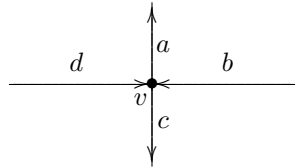
\begin{figure}[htb]
\centerline{$\splinetolerance{.8pt}
\begin{xy}
0;<.5cm,0cm>:
,(0,0)*+{\bu}
,(-3.6,0);(-.1,0)**\crv{} ?>*\dir{>}
,(-3.5,0);(-.2,0)**\crv{}
,(3.6,0);(.1,0)**\crv{} ?>*\dir{>}
,(3.5,0);(.2,0)**\crv{}
,(0,0);(0,2)**\crv{} ?>*\dir{>}
,(0,0.1);(0,1.9)**\crv{}
,(0,0);(0,-2)**\crv{} ?>*\dir{>}
,(0,-0.1);(0,-1.9)**\crv{}
,(.3,1)*+{a}
,(2,.6)*+{b}
,(.3,-1)*+{c}
,(-2,.6)*+{d}
,(-.3,-.4)*+{v}
\end{xy}$}
\caption{An oriented s-manifold $\bZ$ of dimension 1.}
\label{sm3fig3}
\end{figure}

Define $\bX\subset\bZ$ to be the s-submanifold consisting of strata $a,b,v$, and $\bY\subset\bZ$ to be the s-submanifold consisting of strata $b,c,v$. They are oriented, as \eq{sm3eq3} for $\bX$ and $\bY$ become $-1+1=0$ and $1-1=0$.

Write $g:\bX\ra\bZ$ and $h:\bY\ra\bZ$ for the inclusions. They are transverse, and the fibre product $\bW=\bX\t_{g,\bZ,h}\bY$ is the intersection $\bX\cap\bY$, with strata $b,v$. However, $\bW$ is not orientable, as \eq{sm3eq3} for $\bW$ at $v$ would be $\pm 1=0$.
\end{ex}

We can eliminate this problem by strengthening our notion of transversality:

\begin{dfn}
\label{sm3def8}
In the situation of Definition \ref{sm3def6}, we say that $g:\bX\ra\bZ$, $h:\bY\ra\bZ$ are {\it strongly transverse\/} if they are transverse, and the case Definition \ref{sm3def6}(b)(iv) does not occur. As in Remark \ref{sm3rem4}, an important special case is when $\bZ$ is an ordinary $n$-manifold; then transverse implies strongly transverse.

We also generalize strong transversality to $\cSMan$ as in Remark~\ref{sm3rem5}.
\end{dfn}

The fibre product in Example \ref{sm3ex7} is not strongly transverse.

\begin{thm}
\label{sm3thm5}
{\bf(a)} In Theorem\/ {\rm\ref{sm3thm4},} suppose that\/ $g,h$ are strongly transverse and\/ $\bX,\bY,\bZ$ are oriented. Then the transverse fibre product\/ $\bW=\bX\t_{g,\bZ,h}\bY$ has a natural orientation. This depends on an orientation convention, which we take to be that in Akaho--Joyce\/ {\rm\cite[\S 2.4]{AkJo}} and Fukaya--Oh--Ohta--Ono\/ {\rm\cite[\S 8.2]{FOOO}}. It depends on the order of\/ $\bX,\bY,$ and in oriented s-manifolds we have
\e
\bX\t_{g,\bZ,h}\bY\cong (-1)^{(l-n)(m-n)}\bY\t_{h,\bZ,g}\bX.
\label{sm3eq21}
\e

\noindent{\bf(b)} In Theorem\/ {\rm\ref{sm3thm4},} suppose that\/ $g,h$ are strongly transverse and\/ $\bX,\bY,\bZ$ have orientation bundles $(\Or_\bX,\om_\bX),\ab(\Or_\bY,\om_\bY),\ab(\Or_\bZ,\om_\bZ)$. Then we may define an orientation bundle $(\Or_\bW,\om_\bW)$ on\/ $\bW$ as follows. Define a principal\/ $\Z_2$-bundle $\pi:\Or_\bW\ra W$ by
\e
\Or_\bW=e^*(\Or_\bX)\ot_{\Z_2}f^*(\Or_\bY)\ot_{\Z_2}(g\ci e)^*(\Or_\bZ).
\label{sm3eq22}
\e
Definition {\rm\ref{sm3def7}(i)} implies that
\e
W_0=\bigl(X_0\cap g^{-1}(Z_0)\bigr)\vert_{g\vert_{\cdots},Z_0,h\vert_{\cdots}}\bigl(Y_0\cap h^{-1}(Z_0)\bigr),
\label{sm3eq23}
\e
where $X_0\cap g^{-1}(Z_0)$ and\/ $Y_0\cap h^{-1}(Z_0)$ are open and closed in $X_0,Y_0$ and so are manifolds, and\/ \eq{sm3eq23} is a transverse fibre product of manifolds. Properties of transverse fibre products imply that (after choosing a suitable orientation convention) the orientation bundle $\pi:\Or_{W_0}\ra W_0$ has a canonical isomorphism
\e
\Or_{W_0}\cong \pi_{X_0}^*(\Or_{X_0})\ot_{\Z_2}\pi_{Y_0}^*(\Or_{Y_0})\ot_{\Z_2}\pi_{Z_0}^*(\Or_{Z_0}).
\label{sm3eq24}
\e
Define an isomorphism $\om_\bW:\Or_{W_0}\ra\Or_\bW\vert_{W_0}$ of principal\/ $\Z_2$-bundles on $W_0$ to be the composition of
\ea
&\pi_{X_0}^*(\om_\bX)\!\ot\!\pi_{Y_0}^*(\om_\bY)\!\ot\!\pi_{Z_0}^*(\om_\bZ):
\pi_{X_0}^*(\Or_{X_0})\!\ot_{\Z_2}\!\pi_{Y_0}^*(\Or_{Y_0})\!\ot_{\Z_2}\!\pi_{Z_0}^*(\Or_{Z_0})
\nonumber\\
&\qquad \longra\pi_{X_0}^*(\Or_\bX)\!\ot_{\Z_2}\!\pi_{Y_0}^*(\Or_\bY)\!\ot_{\Z_2}\!\pi_{Z_0}^*(\Or_\bZ)
\label{sm3eq25}
\ea
with \eq{sm3eq22} and \eq{sm3eq24}.
\end{thm}

\begin{proof}
As orientations on s-manifolds $\bX$ correspond to orientation bundles $(\Or_\bX,\om_\bX)$ with $\Or_\bX$ trivial, the first part of (a) follows from (b) in the case when $\Or_\bX,\ab\Or_\bY,\ab\Or_\bZ$ are trivial principal $\Z_2$-bundles, so that $\Or_\bW$ in \eq{sm3eq22} is also trivial. Having proved the first part of (a), equation \eq{sm3eq21} will follow from the corresponding equation \cite[Prop.~2.10]{AkJo} for fibre products of the oriented manifolds $X_0,Y_0,Z_0$. So it is enough to prove~(b).

Work in the situation of (b). As for \eq{sm3eq23}, since Definition \ref{sm3def7}(iv) does not occur by strong transversality, Definition \ref{sm3def7}(ii),(iii) imply that 
\begin{align*}
W_1=\bigl(&X_1\cap g^{-1}(Z_0)\bigr)\vert_{g\vert_{\cdots},Z_0,h\vert_{\cdots}}\bigl(Y_0\cap h^{-1}(Z_0)\bigr)
\amalg \\
\bigl(&X_0\cap g^{-1}(Z_0)\bigr)\vert_{g\vert_{\cdots},Z_0,h\vert_{\cdots}}\bigl(Y_1\cap h^{-1}(Z_0)\bigr),
\end{align*}
which is the disjoint union of two transverse fibre products of manifolds. Hence
\e
W_{\le 1}\!=\!\bigl\{(x,y)\!\in\! \bigl(X_{\le 1}\!\cap\! g^{-1}(Z_0)\bigr)\!\t\!\bigl(Y_{\le 1}\!\cap\! h^{-1}(Z_0)\bigr):
\text{$g(x)\!=\!h(y)$ in $Z_0$}\bigr\}.
\label{sm3eq26}
\e

As $Z_0$ is an $n$-manifold, we can choose a tubular neighbourhood for the diagonal embedding $\De_{Z_0}:Z_0\hookra Z_0\t Z_0$. Using this we choose a sequence of decreasing open neighbourhoods $Z_0\t Z_0\supset U_1\supset U_2\supset U_3\supset\cdots\supset \De(Z_0)$ of $\De(Z_0)$ in $Z_0\t Z_0$, such that in the tubular neighbourhood which identifies $Z_0\t Z_0$ near $\De(Z_0)$ with the total space of $TZ_0\ra Z_0$, then $U_i$ is identified with the open balls of radius $\frac{1}{2^i}$ about the zero section using some metric on $Z_0$. Writing $\ov U_i$ for the closure of $U_i$ in $Z_0\t Z_0$ we have $Z_0\t Z_0\supset\ov U_1\supset U_1\supset\ov U_2\supset U_2\supset\cdots\supset \De(Z_0)$, where~$\bigcap_{i\ge 1}\ov U_i=\De_{Z_0}(Z_0)$.

The fundamental class $[Z_0]_\fund\in H_n^\St(Z_0,\Or_{Z_0},\Z)$ pushes forward along the proper map $\De_{Z_0}$ to $(\De_{Z_0})_*([Z_0]_\fund)$ in $H_n^\St(Z_0\t Z_0,\Pi_1^*(\Or_{Z_0}),\Z)$. As $Z_0\t Z_0$ is a manifold, this has a Poincar\'e dual cohomology class $\Pd\bigl((\De_{Z_0})_*([Z_0]_\fund)\bigr)$ in \v Cech cohomology $\check H^n(Z_0\t Z_0,\Pi_2^*(\Or_{Z_0}),\Z)$, where we change from $\Pi_1^*(\Or_{Z_0})$ to $\Pi_2^*(\Or_{Z_0})$ as the orientation bundle of $Z_0\t Z_0$ is $\Pi_1^*(\Or_{Z_0})\ot_{\Z_2}\Pi_2^*(\Or_{Z_0})$.

The same works in the \v Cech cohomology of each open neighbourhood $U_i$, and as the class is supported near $\De_{Z_0}(Z_0)$ and away from $\ov U_i\sm U_i$, we can extend to the relative cohomology of $(\ov U_i;\ov U_i\sm U_i)$, giving classes for $i=1,2,\ldots$
\begin{equation*}
\Pd\bigl((\De_{Z_0})_*([Z_0]_\fund)\bigr){}_i\in \check H^n(\ov U_i;\ov U_i\sm U_i,\Pi_2^*(\Or_{Z_0}),\Z).
\end{equation*}
To relate these for different $i$, we form similar classes
\begin{equation*}
\Pd\bigl((\De_{Z_0})_*([Z_0]_\fund)\bigr){}_{i,i+1}\in \check H^n(\ov U_i;\ov U_i\sm U_{i+1},\Pi_2^*(\Or_{Z_0}),\Z)
\end{equation*}
and then under the inclusions of pairs $\inc_i:(\ov U_i;\ov U_i\sm U_i)\hookra(\ov U_i;\ov U_i\sm U_{i+1})$ and $\inc_{i+1}':(\ov U_{i+1};\ov U_{i+1}\sm U_{i+1})\hookra(\ov U_i;\ov U_i\sm U_{i+1})$ we have
\e
\begin{split}
\inc_i^*\bigl(\Pd\bigl((\De_{Z_0})_*([Z_0]_\fund)\bigr){}_{i,i+1}\bigr)&=\Pd\bigl((\De_{Z_0})_*([Z_0]_\fund)\bigr){}_i,\\
\inc_{i+1}^{\prime *}\bigl(\Pd\bigl((\De_{Z_0})_*([Z_0]_\fund)\bigr){}_{i,i+1}\bigr)&=\Pd\bigl((\De_{Z_0})_*([Z_0]_\fund)\bigr){}_{i+1}.
\end{split}
\label{sm3eq27}
\e

The fundamental classes $[\bX]_\fund,[\bY]_\fund$ given by Theorem \ref{sm3thm1} give a class
\e
[\bX]_\fund\bt[\bY]_\fund\in H_{l+m}^\St(X\t Y,\Or_\bX\bt\Or_\bY,\Z).
\label{sm3eq28}
\e
Pulling back to the open subset $(X_{\le 1}\t Y_{\le 1})\cap (g\t h)^{-1}(U_i)$ gives a class
\begin{equation*}
\inc_i^{\prime\prime*}\bigl([\bX]_\fund\bt[\bY]_\fund\bigr)\in H_{l+m}^\St\bigl((X_{\le 1}\t Y_{\le 1})\cap (g\t h)^{-1}(U_i),\Or_\bX\bt\Or_\bY\vert_{\cdots},\Z\bigr).
\end{equation*}
By properties of relative (co)homology, we have a cap product
\begin{align*}
\cap:\,&H_{l+m}^\St\bigl((X_{\le 1}\t Y_{\le 1})\cap (g\t h)^{-1}(U_i),\Or_\bX\bt\Or_\bY\vert_{\cdots},\Z\bigr)
\t\\
&\check H^n\bigl((X_{\le 1}\t Y_{\le 1})\cap (g\t h)^{-1}(\ov U_i);\\
&\quad (X_{\le 1}\t Y_{\le 1})\cap (g\t h)^{-1}(\ov U_i\sm U_i),\Pi_2^*(\Or_{Z_0}),\Z\bigr)\longra \\
& H_d^\St\bigl((X_{\le 1}\t Y_{\le 1})\cap (g\t h)^{-1}(\ov U_i),(\Or_\bX\bt\Or_\bY\vert_{\cdots})\ot\Pi_2^*(\Or_{Z_0}),\Z\bigr),
\end{align*}
where $d=l+m-n$. Thus we may form the homology class
\ea
&\inc_i^{\prime\prime*}\bigl([\bX]_\fund\bt[\bY]_\fund\bigr)\cap (g\t h)\vert_{\cdots}^*\bigl(\Pd\bigl((\om_\bZ)_*\ci(\De_{Z_0})_*([Z_0]_\fund)\bigr){}_i\bigr)
\label{sm3eq29}\\
&\in H_d^\St\bigl((X_{\le 1}\t Y_{\le 1})\cap (g\t h)^{-1}(\ov U_i),(\Or_\bX\bt\Or_\bY\vert_{\cdots})\ot\Pi_2^*(\Or_\bZ),\Z\bigr),
\nonumber
\ea
and we use $\om_\bZ$ to convert $\Or_{Z_0}$ to $\Or_\bZ$. Using \eq{sm3eq27} we can show that under the proper inclusion $\inc_i^{i+1}:(X_{\le 1}\t Y_{\le 1})\cap (g\t h)^{-1}(\ov U_{i+1})\hookra (X_{\le 1}\t Y_{\le 1})\cap (g\t h)^{-1}(\ov U_i)$ we have
\ea
&\inc_{i*}^{i+1}\bigl(\inc_{i+1}^{\prime\prime*}\bigl([\bX]_\fund\!\bt\![\bY]_\fund\bigr)\!\cap\! (g\!\t\! h)\vert_{\cdots}^*\bigl(\Pd\bigl((\om_\bZ)_*\!\ci\!(\De_{Z_0})_*([Z_0]_\fund)\bigr){}_{i+1}\bigr)\bigr)
\nonumber\\
&=\inc_i^{\prime\prime*}\bigl([\bX]_\fund\!\bt\![\bY]_\fund\bigr)\!\cap\! (g\!\t\! h)\vert_{\cdots}^*\bigl(\Pd\bigl((\om_\bZ)_*\!\ci\!(\De_{Z_0})_*([Z_0]_\fund)\bigr){}_i\bigr).
\label{sm3eq30}
\ea

As $\bigcap_{i\ge 1}\ov U_i=\De_{Z_0}(Z_0)$ we see from \eq{sm3eq26} that
\begin{equation*}
W_{\le 1}=\bigcap_{i=1}^\iy\bigl((X_{\le 1}\t Y_{\le 1})\cap (g\t h)^{-1}(\ov U_i)\bigr).
\end{equation*}
Hence the twisted version of Property \ref{sm2pr2}(h)(i) gives an exact sequence
\e
\begin{gathered}
\xymatrix@C=12pt@R=5pt{ 0 \ar[r] & \varprojlim^1 {\begin{subarray}{l}\ts H_{d+1}^\St\bigl((X_{\le 1}\t Y_{\le 1})\cap (g\t h)^{-1}(\ov U_i), \\ \ts \quad (\Or_\bX\bt\Or_\bY\vert_{\cdots})\ot\Pi_2^*(\Or_\bZ),\Z\bigr)\end{subarray}} \ar[dr] \\
&& *+[l]{\!\!\!\!\!\!\!\!\!\!\!\! H_d^\St(W_{\le 1},\Or_\bW\vert_{W_{\le 1}},R)} \ar[dl] \\ 
0 & \varprojlim {\begin{subarray}{l}\ts H_d^\St\bigl((X_{\le 1}\t Y_{\le 1})\cap (g\t h)^{-1}(\ov U_i), \\ \ts \quad (\Or_\bX\bt\Or_\bY\vert_{\cdots})\ot\Pi_2^*(\Or_\bZ),\Z\bigr),\end{subarray}} \ar[l] }
\end{gathered}
\label{sm3eq31}
\e
since $(\Or_\bX\bt\Or_\bY\vert_{\cdots})\ot\Pi_2^*(\Or_\bZ)$ agrees with $\Or_\bW$ on $W_{\le 1}$ by \eq{sm3eq22}.

Equation \eq{sm3eq30} implies that the classes \eq{sm3eq29} define an element in the direct limit on the last line of \eq{sm3eq31}. Hence there exists a class $[W_{\le 1}]'_\fund\in H_d^\St(W_{\le 1},\Or_\bW\vert_{W_{\le 1}},R)$ such that under the proper inclusion $\ti\inc_i:W_{\le 1}\hookra (X_{\le 1}\t Y_{\le 1})\cap (g\t h)^{-1}(\ov U_i)$, we have
\ea
&(\ti\inc_i)_*\bigl([W_{\le 1}]'_\fund\bigr)
\label{sm3eq32}\\
&=\inc_i^{\prime\prime*}\bigl([\bX]_\fund\!\bt\![\bY]_\fund\bigr)\!\cap\! (g\!\t\! h)\vert_{\cdots}^*\bigl(\Pd\bigl((\om_\bZ)_*\!\ci\!(\De_{Z_0})_*([Z_0]_\fund)\bigr){}_i\bigr).
\nonumber
\ea

Now all the proof from \eq{sm3eq28} to \eq{sm3eq32} works if we substitute $W_0,X_0,Y_0$ for $W_{\le 1},X_{\le 1},Y_{\le 1}$. Also, each step commutes with the pullback maps on Steenrod homology for the open inclusions $W_0\hookra W_{\le 1},\ldots,Y_0\hookra Y_{\le 1}$. Thus we see that
\ea
&(\ti\inc_i)_*\bigl([W_{\le 1}]'_\fund\vert_{W_0}\bigr)
\label{sm3eq33}\\
&=\inc_i^{\prime\prime*}\bigl([\bX]_\fund\vert_{X_0}\bt[\bY]_\fund\vert_{Y_0}\bigr)\cap (g\t h)\vert_{\cdots}^*\bigl(\Pd\bigl((\om_\bZ)_*\!\ci\!(\De_{Z_0})_*([Z_0]_\fund)\bigr){}_i\bigr).
\nonumber\\
&=\inc_i^{\prime\prime*}\bigl((\om_\bX)_*\bigl([X_0]_\fund\bigr)\bt(\om_\bY)_*\bigl([Y_0]_\fund\bigr)\bigr)
\nonumber\\
&\qquad \cap (g\t h)\vert_{\cdots}^*\bigl((\om_\bZ)_*\ci\Pd\bigl((\De_{Z_0})_*([Z_0]_\fund)\bigr){}_i\bigr)
\nonumber\\
&=(\om_\bX\bt\om_\bY\bt\om_\bZ)_*\bigl(\inc_i^{\prime\prime*}\bigl([X_0\t Y_0]_\fund\cap (g\t h)\vert_{\cdots}^*\bigl(\Pd\bigl((\De_{Z_0})_*([Z_0]_\fund)\bigr){}_i\bigr)\bigr)
\nonumber\\
&=(\om_\bW)_*\bigl(\inc_i^{\prime\prime*}\bigl([W_0]_\fund\bigr)\bigr),
\nonumber
\ea
where in the first step we use \eq{sm3eq32}, in the second \eq{sm3eq13} for $\bX,\bY$, we rearrange in the third, and in the fourth we use that as $W_0=(X_0\cap g^{-1}(Z_0))\t_{g\vert_{\cdots},Z_0,h\vert_{\cdots}}(Y_0\cap h^{-1}(Z_0))$ is a transverse fibre product, it has fundamental class
\begin{equation*}
[W_0]_\fund=	[X_0\t Y_0]_\fund\cap (g\t h)\vert_{\cdots}^*\bigl(\Pd\bigl((\De_{Z_0})_*([Z_0]_\fund)\bigr),
\end{equation*}
which is valid on any neighbourhood of $W_0$ in $(X_0\cap g^{-1}(Z_0))\t(Y_0\cap h^{-1}(Z_0))$,
and the definition of $\om_\bW$ in the theorem.

Now in the analogue of \eq{sm3eq31} with $W_0,X_0,Y_0$ in place of $W_{\le 1},X_{\le 1},Y_{\le 1}$, the unique element of $H_d^\St(W_0,\Or_\bW\vert_{W_{\le 1}},R)$ mapping to $(\om_\bW)_*(\inc_i^{\prime\prime*}([W_0]_\fund))$ in the direct limit for each $i$ is $(\om_\bW)_*([W_0]_\fund)$. Thus as the two versions of \eq{sm3eq31} commute with restriction maps to $W_0,X_0,Y_0$ we see from \eq{sm3eq33} that
\begin{equation*}
[W_{\le 1}]'_\fund\vert_{W_0}=(\om_\bW)_*([W_0]_\fund).
\end{equation*}
Hence, the analogue of \eq{sm3eq14} for $\bW$ shows that $\pd\ci(\om_\bW)_*([W_0]_\fund)=0$ in $H_{d-1}^\St(W_1,\Or_\bW\vert_{W_1},\Z)$. This proves equation \eq{sm3eq6}, so $(\Or_\bW,\om_\bW)$ is an orientation bundle on $\bW$, as we have to show.
\end{proof}

\begin{rem}
\label{sm3rem6}
The proof of Theorem \ref{sm3thm5} does not need the orientation or orientation bundle on $\bZ$ to satisfy \eq{sm3eq3} or \eq{sm3eq6}. This is because our notion of strong transversality in Definition \ref{sm3def8} ensures that the dimension $d,d-1$ strata of $\bW$ do not involve the dimension $n-1$ strata of~$\bZ$.
\end{rem}

\subsection{Vector bundles and sections}
\label{sm38}

The usual definition of vector bundle on manifolds \cite[\S 5]{Lee2} generalizes immediately to s-manifolds.

\begin{dfn}
\label{sm3def9}
Let $\bX$ be an s-manifold of dimension $n$. A {\it vector bundle\/} $\pi:\bE\ra\bX$ over $\bX$, of {\it rank\/} $r$, is an s-manifold $\bE$ of dimension $n+r$ with a smooth map $\pi:\bE\ra\bX$, satisfying:
\begin{itemize}
\setlength{\itemsep}{0pt}
\setlength{\parsep}{0pt}
\item[(i)] Each fibre $E_x=\pi^{-1}(x)$ for $x\in X$ has the structure of an $\R$-vector space of dimension $r$, which is part of the data of $\bs E$.
\item[(ii)] We can cover $\bX$ by open s-submanifolds $\bU\subseteq\bX$ such that the open submanifold $\pi^{-1}(\bU)\subseteq\bE$ admits a diffeomorphism of s-manifolds $\pi^{-1}(\bU)\cong\bU\t\R^r$ which identifies $\pi\vert_{\pi^{-1}(\bU)}:\pi^{-1}(\bU)\ra\bU$ with the projection $\pi_\bU:\bU\t\R^r\ra\bU$, and identifies the vector space structure on $E_u$ for $u\in U$ with that on $\{u\}\t\R^r\cong\R^r$.
\end{itemize}

A {\it smooth section\/} $s$ of $\bE$ is a smooth map $s:\bX\ra\bE$ with $\pi\ci s=\id_\bX$. The {\it zero section\/} $0:\bX\ra\bE$ maps $x\mapsto 0\in E_x$ for $x\in X$.

On each stratum $X^i\subset X$, the restriction $E\vert_{X^i}\ra X^i$ is a vector bundle in the sense of manifolds, and $s\vert_{X^i}$ is a smooth section of $E\vert_{X^i}\ra X^i$.

We write $\Ga^\iy(\bE)$ for the set of smooth sections of $\bE$. It is naturally an $\R$-vector space, and a representation of the $\R$-algebra $C^\iy(\bX)$ of smooth functions $f:\bX\ra\R$ from \S\ref{sm35}. We make $\Ga^\iy(\bE)$ into a topological vector space by combining the $C^0_{\rm loc}$ topology on continuous sections of $E\ra X$, and the $C^\iy_{\rm loc}$ topology on smooth sections of $E\vert_{X^i}\ra X^i$ for each stratum $X^i$ of $X$. 

That is, a sequence $(s_j)_{j=1}^\iy$ in $\Ga^\iy(\bE)$ converges to $s\in\Ga^\iy(\bE)$ if on every compact subset $K\subseteq X$, $s_j\vert_K\ra s\vert_K$ as $j\ra\iy$ as continuous sections of $E$, and on every compact subset $K^i\subseteq X^i$ for all $i\in I$, $\nabla^ks_j\vert_{K^i}\ra \nabla^ks\vert_{K^i}$ as $j\ra\iy$ for all $k\ge 0$, where $\nabla$ is any connection on~$E\vert_{X^i}\ra X^i$.
\end{dfn}

\begin{rem}
\label{sm3rem7}
Unlike manifolds, s-manifolds $\bX$ in general do {\it not\/} have good notions of tangent and cotangent vector bundle $T\bX,T^*\bX$. Although we defined tangent spaces $T_x\bX$ in Definition \ref{sm3def1}(c), $\dim T_x\bX$ can vary discontinuously with $x$ in $\bX$, and so $T_x\bX$ are not the fibres of a vector bundle $T\bX$ on~$\bX$.	
\end{rem}

As for Corollary \ref{sm3cor1}, we can use Theorem \ref{sm3thm2} to deduce:

\begin{cor}
\label{sm3cor4}
Let\/ $\bX$ be an s-manifold with\/ $\dim\bX>0$ and\/ $X_0\ne\es,$ and $\bE\ra\bX$ be a vector bundle of rank\/ $r>0$. Then the $\R$-vector space $\Ga^\iy(\bE)$ of smooth sections of $\bE$ is infinite-dimensional.
\end{cor}

\begin{dfn}
\label{sm3def10}
Let $\bX$ be an s-manifold of dimension $n,$ and $\bE\ra\bX$ a vector bundle of rank $r$. If $x\in\bX$, so that $0(x)\in\bE$, then there is a canonical isomorphism
\e
T_{0(x)}\bE\cong T_x\bX\op E_x,
\label{sm3eq34}
\e
as this holds for vector bundles on smooth manifolds, and we can deduce it from the vector bundle $E\vert_{X^i}\ra X^i$ on the stratum $X^i$ of $\bX$ containing~$x$.

Let $s$ be a smooth section of $\bE$. Write $s^{-1}(0)\subset X$ for the closed subset of points $x\in X$ with $s(x)=0(x)$. We say that $s$ is {\it transverse\/} if for all $x\in s^{-1}(0)$, the following composition is surjective:
\begin{equation*}
\xymatrix@C=40pt{ T_x\bX \ar[r]^(0.45){T_xs} & T_{0(x)}\bE \ar[r]^(0.45){\eq{sm3eq34}}_(0.45)\cong & T_x\bX\op E_x \ar[r]^(0.6){\pi_{E_x}} & E_x. }
\end{equation*}
\end{dfn}

\begin{thm}
\label{sm3thm6}
Let\/ $\bX$ be an s-manifold of dimension $n,$ and\/ $\bE\ra\bX$ a vector bundle of rank\/ $r$. Then:
\begin{itemize}
\setlength{\itemsep}{0pt}
\setlength{\parsep}{0pt}
\item[{\bf(a)}] Generic sections of\/ $\bE\ra\bX$ are transverse. That is, the subset of transverse sections in $\Ga^\iy(\bE)$ is dense, in the topology from Definition\/~{\rm\ref{sm3def9}}.
\item[{\bf(b)}] If\/ $s$ is a transverse smooth section of\/ $\bE$ then the zero set $s^{-1}(0)\subseteq X$ naturally has the structure of an s-manifold of dimension $n-r,$ which we write as $\bs{s^{-1}(0)}$.

If\/ $\bX$ has an orientation bundle $(\Or_\bX,\om_\bX)$ then $\bs{s^{-1}(0)}$ has an orientation bundle $(\Or_{\smash{\bs{s^{-1}(0)}}},\om_{\smash{\bs{s^{-1}(0)}}})$ with\/ $\Or_{\smash{\bs{s^{-1}(0)}}}=\bigl(\Or_\bX\ot_{\Z_2}\Or_E\bigr)\big\vert{}_{s^{-1}(0)},$ where $\pi:\Or_E\ra X$ is the principal\/ $\Z_2$-bundle of orientations on the fibres of the topological vector bundle $\pi:E\ra X$.
\end{itemize}
\end{thm}

\begin{proof}
For (a), by an argument similar to parts of the proof of Theorem \ref{sm3thm3}, we can show that we can choose sections $s_1,\ldots,s_{r+n}\in\Ga^\iy(\bE)$ such that $s_1(x),\ldots,s_{r+n}(x)$ span $E\vert_x$ for each $x\in X$. Given $s\in\Ga^\iy(\bE)$, consider the perturbed section $\ti s=s+\ep_1s_1+\cdots\ep_ns_{r+n}$ for $(\ep_1,\ldots,\ep_{r+n})\in\R^{r+n}$ small. Using Sard's Theorem in a similar way to Corollary \ref{sm3cor3}, we can show that $\ti s$ is transverse for $(\ep_1,\ldots,\ep_{r+n})$ outside a null set in $\R^{r+n}$. By taking $(\ep_1,\ldots,\ep_{r+n})$ very small we can make $\ti s$ arbitrarily close to $s$ in the topology on $\Ga^\iy(\bE)$. Thus the set of transverse sections is dense in~$\Ga^\iy(\bE)$.

For (b), let $\bU\subseteq\bX$ and $\pi^{-1}(\bU)\cong\bU\t\R^r$ be as in Definition \ref{sm3def9}(ii). Then we may write $s\vert_U=(\id_U,f)$ for smooth $f:\bU\ra\R^r$, where $U\cap s^{-1}(0)=f^{-1}(0)$. As $s$ is transverse, $T_xf:T_x\bU\ra T_0\R^r=\R^r$ is surjective for each $x\in U$ with $f(x)=0$. Therefore $f:\bU\ra\R^r$ and $0:*\ra\R^r$ are transverse in the sense of Definition \ref{sm3def6}, so Definition \ref{sm3def7} and Theorem \ref{sm3thm4} imply that $U\cap s^{-1}(0)=U\t_{f,\R^r,0}*$ has the structure of an s-manifold of dimension~$n-r$. 

Examining the proofs we can show that these s-manifold structures are independent of the local trivializations $\pi^{-1}(\bU)\cong\bU\t\R^r$, and glue to give a global s-manifold structure on $s^{-1}(0)$. Rather than thinking of it as a local fibre product, we can also build the s-manifold structure on $s^{-1}(0)$ directly in terms of the manifold structures on $(s\vert_{X^i})^{-1}(0)\subseteq X^i$ for each $i\in I$, where $s\vert_{X^i}$ is a transverse section of $E\vert_{X^i}\ra X^i$ in the sense of smooth manifolds. 

For the second part of (b), note that transverse fibre products over the ordinary manifold $\R^r$ are automatically strongly transverse as in Definition \ref{sm3def8}, so we can prove it starting from Theorem \ref{sm3thm5} in a similar way.
\end{proof}

\begin{rem}
\label{sm3rem8}
The category of s-manifolds $\SMan$ has all the properties needed to define a well-defined 2-category of {\it Kuranishi spaces\/} in the sense of the author \cite{Joyc2,Joyc5,Joyc6} --- see in particular \cite{Joyc5} in which Kuranishi spaces are defined starting from a general category of `manifolds' $\bs{\bf\dot Man}$ satisfying some axioms, which hold for $\SMan$ (we have just proved a lot of the axioms). Kuranishi spaces are (in the author's view of the subject) {\it derived smooth manifolds\/} and {\it derived smooth orbifolds}, where `derived' is in the sense of Derived Algebraic Geometry. 

Kuranishi spaces were introduced by Fukaya--Oh--Ohta--Ono \cite{FOOO,FuOn}, some years before the invention of Derived Algebraic Geometry, as the geometric structure on moduli spaces of $J$-holomorphic curves in Symplectic Geometry, and were primarily used as a tool to define virtual classes for such moduli spaces. The goal of \cite{Joyc2,Joyc5,Joyc6} (again, in the author's view) was to correct the definition of Kuranishi space in \cite{FOOO,FuOn} using insights from Derived Algebraic Geometry.

A Kuranishi space $\bX$ has an atlas of local charts $(V,E,\Ga,s,\psi)$, where $V$ is a `manifold' (e.g. a manifold, or manifold with corners, or s-manifold), $E\ra X$ is a vector bundle, $\Ga$ is a finite group acting on $V,E$, $s\in\Ga^\iy(E)$ is a $\Ga$-equivariant smooth section, and $\psi:s^{-1}(0)/\Ga\ra U$ is a homeomorphism with an open subset $U\subseteq X$. There is also complicated data on transitions between charts.

Theorem \ref{sm3thm6} implies that a Kuranishi space $\bX$ in s-manifolds with trivial isotropy groups can be perturbed to an s-manifold $\bs{\ti X}$. If $\bX$ has an orientation, or an orientation bundle, then so does $\bs{\ti X}$, so $\bs{\ti X}$ has a fundamental class $[\bs{\ti X}]_\fund$ in Steenrod homology by Theorem \ref{sm3thm1}. This is a sign that Kuranishi spaces $\bX$ in s-manifolds have {\it virtual classes\/} $[\bX]_\virt$ in Steenrod homology, where roughly $[\bX]_\virt=[\bs{\ti X}]_\fund$ for any perturbation of $\bX$ to an s-manifold $\bs{\ti X}$.
\end{rem}

\subsection{\texorpdfstring{$R$-weighted s-manifolds}{R-weighted s-manifolds}}
\label{sm39}

We can include {\it weights\/} in a commutative ring $R$ on the top-dimensional strata of s-manifolds.

\begin{dfn}
\label{sm3def11}
Let $\bX=(X,\{X^i:i\in I\})$ be an s-manifold of dimension $n$, and $R$ a commutative ring, e.g.\ $R=\Q$ or $\Z$. An $R$-{\it weighting\/} $w$ for $\bX$ is a map
\begin{equation*}
w:\bigl\{X^i:i\in I,\; \dim X^i=n\bigr\}\longra R.
\end{equation*}
Then we call $(\bX,w)$ an $R$-{\it weighted s-manifold}. An unweighted s-manifold $\bX$ may be regarded as a weighted s-manifold $(\bX,1)$ with~$w\equiv 1$.

We can now modify the material of \S\ref{sm33} for $R$-weighted s-manifolds. For $(\bX,w)$ as above, suppose we are given an orientation on $X_0$. Noting that $X_0=\coprod_{i\in I:\dim X^i=n}X^i$, we define the {\it weighted fundamental class\/} of $X_0$ to be
\e
[X_0]^w_\fund=\sum_{i\in I:\dim X^i=n}w(X^i)[X^i]_\fund\qquad\text{in $H_n^\St(X_0,R)$.}
\label{sm3eq35}
\e
Then we define an {\it orientation\/} on $(\bX,w)$ to be an orientation on $X_0$ such that
\e
\pd\bigl([X_0]^w_\fund\bigr)=0\qquad\text{in $H_{n-1}^\St(X_1,R)$,}
\label{sm3eq36}
\e
generalizing \eq{sm3eq3}, where $\pd$ is as in \eq{sm3eq2} with $R$ in place of $\Z$. The analogue of Theorem \ref{sm3thm1}(a) with $R$ in place of $\Z$ then implies that there is a unique class $[\bX]^w_\fund$ in the Steenrod homology group $H_n^\St(X,R)$ such that  
\begin{equation*}
[\bX]^w_\fund\vert_{X_0}=[X_0]^w_\fund.
\end{equation*}
We call $[\bX]^w_\fund$ the {\it weighted fundamental class\/} of $(\bX,w)$. 

Similarly, we define an {\it orientation bundle\/} $(\Or_\bX,\om_\bX)$ on $(\bX,w)$ to be as in Definition \ref{sm3def3}, but replacing \eq{sm3eq6} by
\begin{equation*}
\pd\ci (\om_\bX)_*\bigl([X_0]^w_\fund\bigr)=0\qquad\text{in $H_{n-1}^\St(X_1,\Or_\bX\vert_{X_1},R)$,}
\end{equation*}
where now $[X_0]^w_\fund\in H_n^\St(X_0,\Or_{X_0},R)$. The analogue of Theorem \ref{sm3thm1}(b) with $R$ in place of $\Z$ then implies that there is a unique class $[\bX]^w_\fund$ in the Steenrod homology group $H_n^\St(X,\Or_\bX,R)$ such that, generalizing \eq{sm3eq15}
\begin{equation*}
[\bX]^w_\fund\vert_{X_0}=(\om_\bX)_*\bigl([X_0]^w_\fund\bigr).
\end{equation*}
We call $[\bX]^w_\fund$ the {\it weighted fundamental class\/} of $(\bX,w)$. 
\end{dfn}

\begin{rem}
\label{sm3rem9}
{\bf(a)} This is designed for applications in Symplectic Geometry, as in \S\ref{sm5}. To define Gromov--Witten invariants of general symplectic manifolds $(X,\om)$, moduli spaces $\oM_{g,k}(J,\be)$ of $J$-holomorphic curves must be given virtual classes $[\oM_{g,k}(J,\be)]_\virt$ in homology over $\Q$, not $\Z$ (e.g.\ see Fukaya--Ono \cite{FuOn}). Our approach will be to define perturbed versions of these moduli spaces $\bs\oM_{g,k}(J,\be)$ which are oriented $\Q$-weighted s-manifolds, and then the desired virtual class is the $\Q$-weighted fundamental class $[\bs\oM_{g,k}(J,\be)]^w_\fund$.
\smallskip

\noindent{\bf(b)} Our $\Q$-weighted s-manifolds are similar to perturbation of Kuranishi spaces using multisections in Fukaya--Ono \cite[\S 3, \S 6]{FuOn}: they perturb a Kuranishi space $\bs K$ such as $\oM_{g,k}(J,\be)$ to a multi-sheeted `manifold' $\ti K$ with rational weights, and then take the weighted fundamental class $[\ti K]_\fund^w$ of $[\ti K]$. In contrast, we will build the perturbations into the definition of the moduli spaces.  
\end{rem}

\begin{ex}
\label{sm3ex8}
Figure \ref{sm3fig4} gives a picture of a simple oriented $\Q$-weighted s-manifold of dimension $1$, where arrows give orientations and the $\Q$-weights are shown. Equation \eq{sm3eq36} at the vertex corresponds to the equation $\frac{5}{6}-\frac{1}{2}-\frac{1}{3}=0$.
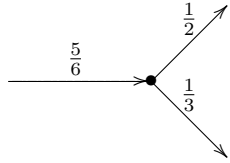
\begin{figure}[htb]
\centerline{$\splinetolerance{.8pt}
\begin{xy}
0;<.5cm,0cm>:
,(0,0)*+{\bu}
,(-3.6,0);(-.1,0)**\crv{} ?>*\dir{>}
,(-3.5,0);(-.2,0)**\crv{}
,(0,0);(2,2)**\crv{} ?>*\dir{>}
,(0,0);(2,-2)**\crv{} ?>*\dir{>}
,(-2,.6)*+{\frac{5}{6}}
,(1,1.65)*+{\frac{1}{2}}
,(1,-.35)*+{\frac{1}{3}}
\end{xy}$}
\caption{An oriented $\Q$-weighted s-manifold of dimension 1.}
\label{sm3fig4}
\end{figure}
\end{ex}

The material on (oriented) transverse fibre products of s-manifolds in \S\ref{sm36}--\S\ref{sm37} extends to $R$-weighted s-manifolds:

\begin{dfn}
\label{sm3def12}
In Definition \ref{sm3def7} and Theorem \ref{sm3thm4}, suppose $\bX,\bY$ have $R$-weightings $w_\bX,w_\bY$, but we take $\bZ$ to be unweighted. Define an $R$-weighting $w_\bW$ on $\bW$ by
\e
w_\bW(W^{ijk})=w_\bX(X^i)w_\bY(Y^j)
\label{sm3eq37}
\e
whenever $X^i,Y^j,Z^k$ are strata of $\bX,\bY,\bZ$ with $\dim X^i=l$, $\dim Y^j=m$, $\dim Z^k=n$. Definition \ref{sm3def7}(i) shows that all top-dimensional strata of $\bW$ are of this form. Then $(\bW,w_\bW)$ is a $R$-weighted s-manifold.

Theorem \ref{sm3thm5} and its proof then extend easily to the $R$-weighted case, replacing fundamental classes by weighted fundamental classes for $\bX,\bY$. The second part of Theorem \ref{sm3thm6}(b) also extends to the $R$-weighted case in the obvious way.
\end{dfn}

\begin{rem}
\label{sm3rem10}
We took $\bZ$ to be unweighted as it is unclear what is the most natural way to include $w_\bZ(Z^k)$ in \eq{sm3eq37}. In the applications the author has in mind, $\bZ$ is a manifold or manifold with corners, and is unweighted (i.e.~$w_\bZ\equiv 1$).
\end{rem}

Definition \ref{sm2def1} defined the {\it smooth singular homology groups\/} $H_n^\ssi(Y,R)$ of a manifold $Y$. We can interpret $H_n^\ssi(Y,R)$ in terms of $R$-weighted s-manifolds. 

\begin{ex}
\label{sm3ex9}
Let $Y$ be a smooth manifold, $R$ a commutative ring, and $\al\in H_n^\ssi(Y,R)$ for $n\ge 0$. Write $\al=[C]$ for $C=\sum_{a\in A}\rho_a\,\si_a$ in $C_n^\ssi(Y,R)$ with $\pd C=0$, so that $A$ is a finite indexing set, $\rho_a\in R$, and $\si_a:\De_n\ra Y$ is smooth. We take the $\si_a$ for $a\in A$ to be distinct. Define a topological space $X=(A\t\De_n)/\sim$ with the quotient topology, where $A$ has the discrete topology, and $\sim$ is the equivalence relation on $A\t\De_n$ generated by the rule that if $0\le k<n$, and $a_1,a_2\in A$, and $\io_1,\io_2:\De_k\hookra\De_n$ are inclusions of $k$-dimensional faces preserving order of coordinates (such $\io$ are of the form $F_{i_n}^n\ci F_{i_{n-1}}^{n-1}\ci\cdots\ci F_{i_{k+1}}^{k+1}$ for $F_i^j:\De_{j-1}\hookra\De_j$ as in Definition \ref{sm2def1}), such that $\si_{a_1}\ci\io_1=\si_{a_2}\ci\io_2:\De_k\ra Y$, then $(a_1,\io_1(\bs z))\sim (a_2,\io_2(\bs z))$ for all $\bs z\in\De_k$. 

Then $X$ is a {\it simplicial complex\/} of dimension $n$, as in Example \ref{sm3ex2} (although we do not require the simplices $\De_k\ra X$ to be injective except on interiors $\De_k^\ci$), where the $k$-simplices $\De_k\ra X$ are parametrized by maps $\si_a\ci\io:\De_k\ra Y$ for $a\in A$ and $\io:\De_k\hookra\De_n$ an inclusion of a $k$-dimensional face, and two such $k$-simplices $\si_a\ci\io,\si_{a'}\ci\io'$ are the same if $\si_a\ci\io=\si_{a'}\ci\io'$ as maps $\De_k\ra Y$. Thus, Example \ref{sm3ex2} makes $X$ into a compact s-manifold $\bX$ of dimension~$n$.

Define a map $f:X\ra Y$ by $f([a,\bs z])=\si_a(\bs z)$, where $[a,\bs z]\in X$ is the $\sim$-equivalence class of $(a,\bs z)\in A\t\De_n$. Then $f$ is continuous, and smooth on each stratum $\De_k^\ci\subset X$. Thus, regarding the manifold $Y$ as an s-manifold $\bY$, $f:\bX\ra\bY$ is a smooth map of s-manifolds.

As there are no nontrivial equivalences $\sim$ on $A\t\De_n^\ci$, we see that $X_0$ in \eq{sm3eq1} is $X_0=A\t\De_n^\ci$, and the dimension $n$ strata of $\bX$ are $\{a\}\t\De_n^\ci$ for $a\in A$. Define an $R$-weighting $w$ on $\bX$ by
\begin{equation*}
w:\bigl\{\{a\}\t\De_n^\ci:a\in A\bigr\}\longra R,\quad w:\{a\}\t\De_n^\ci\longmapsto\rho_a.
\end{equation*}
Define an orientation on $(\bX,w)$ by taking the orientation on $\{a\}\t\De_n^\ci$ to be the standard orientation on $\De_n^\ci$ for each $a\in A$. Then equation \eq{sm3eq36} is equivalent to $\pd\bigl(\sum_{a\in A}\rho_a\,\si_a\bigr)=0$ in $C_{n-1}^\ssi(Y,R)$, so this is indeed an orientation. Thus, we have a weighted fundamental class $[\bX]_\fund^w\in H_n^\St(X,R)$.

As $\bX$ is compact, $f:\bX\ra\bY$ is proper, so $f_*([\bX]_\fund^w)\in H_n^\St(Y,R)$. If $Y$ is compact then $f_*([\bX]_\fund^w)$ is identified with $\al\in H_n^\ssi(Y,R)$ under the natural isomorphism $H_n^\St(Y,R)\cong H_n^\ssi(Y,R)$.
\end{ex}

We will return to this in Example \ref{sm4ex6}, using s-manifolds with boundary.

\section{S-manifolds with corners}
\label{sm4}

\subsection{The definition of s-manifolds with corners}
\label{sm41}

\begin{dfn}
\label{sm4def1}
An {\it s-manifold with corners\/} $\bfX$ of {\it dimension\/} $\dim\bfX=n\ge 0$, is data
\e
\bfX=\bigl(\bX,\; \Pi_k:C_k(\bX)\ra\bX,\; 0\le k\le n, \; D_{kl}(\bX),\; 0\le k\le l\le n\bigr),
\label{sm4eq1}
\e
satisfying the following conditions (a)--(f), where:
\begin{itemize}
\setlength{\itemsep}{0pt}
\setlength{\parsep}{0pt}
\item[(a)] $\bX$ is an s-manifold of dimension $n$, and for each $k=0,\ldots,n$, $C_k(\bX)$ is an s-manifold of dimension $n-k$, and $\Pi_k:C_k(\bX)\ra\bX$ is a proper morphism of s-manifolds, which is {\it locally a closed embedding}, that is, we can cover $C_k(X)$ by open $V\subseteq C_k(X)$ such that $\Pi_k\vert_V:V\ra X$ is a homeomorphism with a locally closed subset $\Pi_k(V)\subseteq X$. 

We write $C(\bX)=\coprod_{k=0}^nC_k(\bX)$, regarded as an object of the category $\cSMan$ of s-manifolds with mixed dimension from Definition \ref{sm3def1}. We also write $\Pi_\bX=\coprod_{k=0}^n\Pi_k:C(\bX)\ra\bX$, as a morphism in $\cSMan$.
\item[(b)] $\Pi_0:C_0(\bX)\ra\bX$ is an isomorphism of s-manifolds (so we can just take~$C_0(\bX)=\bX$).
\item[(c)] For each stratum $C_k(X)^i$ of $C_k(\bX)$, there is a stratum $X^j$ of $\bX$ such that $\Pi_k(C_k(X)^i)=X^j$, and $\Pi_k\vert_{C_k(X)^i}:C_k(X)^i\ra X^j$ is a {\it proper covering map\/} of manifolds (a proper local diffeomorphism). 
\item[(d)] For all $0\le k\le l\le n$, $D_{kl}(\bX)$ is an open and closed subset of the topological fibre product
\begin{equation*}
C_k(X)\t_{\Pi_k,X,\Pi_l}C_l(X)=\bigl\{(v,w)\in C_k(X)\t C_l(X):\Pi_k(v)=\Pi_l(w)\bigr\}.
\end{equation*}
It is a disjoint union of fibre products $C_k(X)^i\t_{X^j}C_l(X)^{i'}$. That is, if $(v,w)\in D_{kl}(\bX)$ and $v,w$ and $x=\Pi_k(v)=\Pi_l(w)$ lie in the strata $C_k(X)^i,C_l(X)^{i'},X^j$ of $C_k(\bX),C_l(\bX),\bX$, then $D_{kl}(\bX)$ contains the entire fibre product $C_k(X)^i\t_{X^j}C_l(X)^{i'}$. Thus $D_{kl}(\bX)$ is discrete information.

We write $D(\bX)=\coprod_{0\le k\le l\le n}D_{kl}(\bX)$, considered as an open and closed subset of~$C(X)\t_{\Pi_X,X,\Pi_X}C(X)$.
\item[(e)] For all $0\le k\le l\le n$, the projection $\Pi_{C_l(X)}:D_{kl}(\bX)\ra C_l(X)$ is a {\it proper covering map\/} of topological spaces (a proper local homeomorphism). Together with (c),(d), this implies that we can give $D_{kl}(\bX)$ the natural structure of an s-manifold of dimension $n-l$, with strata the fibre products $C_k(X)^i\t_{X^j}C_l(X)^{i'}$ in (d), which are transverse fibre products of manifolds, such that $\Pi_{C_l(\bX)}:D_{kl}(\bX)\ra C_l(\bX)$ is a local isomorphism.
\item[(f)] For $x\in X$ write $C_k(X)_x=\Pi_k^{-1}(x)\subseteq C_k(X)$ and $C(X)_x=\coprod_{k=0}^nC_k(X)_x$. They are finite sets, as $\Pi_k$ is a proper locally closed embedding. Define a binary relation $\tl$ on $C(X)_x$ by $u\tl v$ if $(u,v)\in D(\bX)$. We require that:
\begin{itemize}
\setlength{\itemsep}{0pt}
\setlength{\parsep}{0pt}
\item[(i)] $\tl$ is a partial order on $C(X)_x$.
\item[(ii)] If $u,v\in C_k(X)_x$ then $u\tl v$ if and only if $u=v$.
\item[(iii)] Every nonempty subset $\es\ne S\subseteq C(X)_x$ has a (unique) {\it least upper bound\/} $\lub(S)$, and {\it greatest lower bound\/} $\glb(S)$, in~$(C(X)_x,\tl)$.

Thus $C(X)_x$ has a {\it maximum\/} $\max C(X)_x$ and {\it minimum\/} $\min C(X)_x$, where $\Pi_0^{-1}(x)=\{\min C(X)_x\}\subseteq C_0(X)$.
\end{itemize}
\end{itemize}
We call an s-manifold with corners $\bfX$ an {\it s-manifold with boundary\/} if $C_k(\bX)=\es$ for $k\ge 2$, and an {\it s-manifold without boundary\/} if $C_k(\bX)=\es$ for~$k\ge 1$. 
\end{dfn}


\begin{dfn}
\label{sm4def2}
Let $\bfX$ be an s-manifold with corners of dimension $n$. Define a partial order $\bl$ on $\{0,1\}^d$ by $(u_1,\ldots,u_d)\bl(v_1,\ldots,v_d)$ if $u_i\le v_i$ for $i=1,\ldots,d$. We say that $\bfX$ {\it has simplicial corners\/} if for all $x\in X$, the partial order $(C(X)_x,\tl)$ in Definition \ref{sm4def1}(f) is isomorphic to $(\{0,1\}^d,\ab\bl)$ for some $d=0,\ldots,n$, where the isomorphism identifies $u\in C_k(X)_x$ with $(u_1,\ldots,u_d)\in \{0,1\}^d$ with~$u_1+\cdots+u_d=k$. 
\end{dfn}

Simplicial corners are the `usual' kind of corners. Manifolds with corners, when regarded as s-manifolds with corners as in Example \ref{sm4ex1} below, have simplicial corners, but manifolds with g-corners in \S\ref{sm271}, such as $X_P$ in \eq{sm2eq23}, can have non-simplicial corners. Some properties such as \eq{sm2eq10}--\eq{sm2eq13} can be true of s-manifolds with simplicial corners, but false for non-simplicial corners. S-manifolds with boundary, or without boundary, are always simplicial.

\begin{dfn}
\label{sm4def3}
Let $\bfX$ be an s-manifold with corners of dimension $n$. For $0\le k\le l\le n$, as $\Pi_l$ is proper, $\Pi_{C_k(\bX)}:D_{kl}(\bX)\ra C_k(\bX)$ is proper, and has image a closed set in $C_k(\bX)$. We define the {\it interior\/} $C_k(\bX)^\ci\subseteq C_k(\bX)$ to be the open subset $C_k(\bX)^\ci=C_k(\bX)\sm \bigl(\bigcup_{l=k+1}^n\Pi_{C_k(\bX)}\bigl(D_{kl}(\bX)\bigr)\bigr)$, as an open s-submanifold of $C_k(\bX)$, of dimension $n-k$, for $k=0,\ldots,n$. We think of $C_k(\bX)^\ci$ as an {\it s-manifold}, not an s-manifold with corners. We write $C(\bX)^\ci=\coprod_{k=0}^nC_k(\bX)^\ci$, an open s-submanifold of $C(\bX)$ of mixed dimension.

For $x,C(X)_x,\tl$ as in Definition \ref{sm4def1}(f), $u\in C_k(X)_x$ lies in $C_k(X)^\ci$ if and only if there does not exist $v\in C(X)_x$ with $u\ne v$ and $u\tl v$. That is, $u$ must be maximal in $(C(X)_x,\tl)$. But $C(X)_x$ has a unique maximum $\max C(X)_x$ by Definition \ref{sm4def1}(f)(iii), so $u\in C_k(\bX)^\ci$ if and only if $u=\max C(X)_x$. Therefore $\Pi_\bX\vert_{C(\bX)^\ci}:C(\bX)^\ci\ra\bX$ is a bijection.

Define the $k^{\rm th}$ {\it corner stratum\/} of $\bfX$ to be $S^k(X)=\Pi_k(C_k(\bX)^\ci)\subseteq X$. As $\Pi_\bX\vert_{C(\bX)^\ci}$ is a bijection, $X$ decomposes as the disjoint union $X=\coprod_{k=0}^nS^k(X)$. Also $\Pi_k\vert_{C_k(\bX)^\ci}:C_k(\bX)^\ci\ra S^k(X)$ is a bijection, and a homeomorphism, as $\Pi_k$ is a proper locally closed embedding. An alternative expression is
\begin{equation*}
S^k(X)=\Pi_k(C_k(\bX))\sm \bigl(\ts\bigcup_{l=k+1}^n\Pi_l(C_l(\bX))\bigr).
\end{equation*}
This implies that for $k=0,\ldots,n$
\begin{equation*}
\ts\coprod_{l=k}^nS^l(X)=\ts\bigcup_{l=k}^n\Pi_l(C_l(\bX)).
\end{equation*}
Now $\Pi_l(C_l(\bX))$ is closed in $X$ as $\Pi_l$ is proper. So the last two equations imply that $S^k(X)$ is locally closed in $X$, and $\ov{S^k(X)}\subseteq\ts\coprod_{l=k}^nS^l(X)$, where $\ov{S^k(X)}$ is the closure of $S^k(X)$ in~$X$. 

In particular, this implies that $S^0(X)$ is an open subset of $X$ homeomorphic to $C_0(X)^\ci$. We write $\bX^\ci$ for the corresponding open s-submanifold of $\bX$, isomorphic to $C_0(\bX)^\ci$, and call $\bX^\ci$ the {\it interior\/} of $\bfX$.
\end{dfn}

\begin{dfn}
\label{sm4def4}
Let $\bfX,\bfY$ be s-manifolds with corners. A {\it morphism}, or {\it smooth map}, $\bff:\bfX\ra\bfY$ is a pair $\bff=(f,C(f))$ where $f:\bX\ra\bY$ is a morphism in $\SMan$, and $C(f):C(\bX)\ra C(\bY)$ is a morphism in $\cSMan$, such that the following commutes in~$\cSMan$
\e
\begin{gathered}
\xymatrix@C=150pt@R=15pt{ *+[r]{C(\bX)} \ar[r]_{C(f)} \ar[d]^{\Pi_\bX} & *+[l]{C(\bY)} \ar[d]_{\Pi_\bY} \\ 
 *+[r]{\bX} \ar[r]^{f}  & *+[l]{\bY,\!} }
\end{gathered}
\label{sm4eq2}
\e
and (noting that $(C(f)\!\t\! C(f))(D(\bX))\!\subseteq\! C(Y)\!\t_{\Pi_Y,Y,\Pi_Y}\!C(Y)$ by \eq{sm4eq2}), we~have
\ea
\bigl(C(f)\t C(f)\bigr)\bigl(D(\bX)\bigr)&\subseteq D(\bY)\subseteq C(Y)\t_{\Pi_Y,Y,\Pi_Y}C(Y),
\label{sm4eq3}\\
C(f)(C(\bX)^\ci)&\subseteq C(\bY)^\ci.
\label{sm4eq4}
\ea

We call $\bff$ {\it interior\/} if also $C(f)(C_0(\bX))\subseteq C_0(\bY)\subseteq C(\bY)$. Then~$f(\bX^\ci)\subseteq\bY^\ci$. This corresponds to interior maps of manifolds with corners in Definition~\ref{sm2def4}.

If $\bfg:\bfY\ra\bfZ$ is another morphism of s-manifolds with corners, the {\it composition\/} is $\bfg\ci\bff=(g\ci f,C(g)\ci C(f))$. It is a morphism, and is interior if $\bff,\bfg$ are interior. The {\it identity\/} $\id_\bfX$ is $\id_\bfX=(\id_\bX,\id_{C(\bX)})$, an interior morphism. With these definitions, s-manifolds with corners and smooth maps form a category $\SManc$, with a subcategory $\SMancin\subset\SManc$ with morphisms interior maps. We write $\SManb\subset\SManc$ and $\SManbin\subset\SMancin$ for the full subcategories with objects s-manifolds with boundary.

We call $\bff:\bfX\ra\bfY$ a {\it diffeomorphism\/} if it is smooth with smooth inverse.

If $\bX$ is an s-manifold, we may make it into an s-manifold without boundary $\bfX$ by taking $C_0(\bX)=\bX$, $\Pi_0=\id_\bX$, $D_{00}(\bX)=\De_X$, and $C_k(\bX)=D_{kl}(\bX)=\es$ if $k\ne 0$, $(k,l)\ne(0,0)$. This defines an equivalence between s-manifolds $\bfX$ without boundary and s-manifolds in Definition \ref{sm3def1}. It also identifies smooth maps $f:\bX\ra\bY$ of s-manifolds with smooth maps $\bff:\bfX\ra\bfY$ of the corresponding s-manifolds without boundary, which are automatically interior. Thus we may identify $\SMan$ with the full subcategory of $\SManc$ (equivalently, of $\SMancin$) with objects s-manifolds without boundary. 

As for $\cSMan$ in Definition \ref{sm3def1}, we write $\cSManc$ for the category whose objects are disjoint unions $\coprod_{m=0}^\iy\bfX_m$, where $\bfX_m$ is an s-manifold with corners, of dimension $m$, allowing $\bfX_m=\es$, and whose morphisms $\bs f=(f,C(f)):\coprod_{m=0}^\iy\bfX_m\ra\coprod_{n=0}^\iy\bfY_n$ are pairs of morphisms $f:\coprod_{m=0}^\iy\bX_m\ra\coprod_{n=0}^\iy\bY_n$ and $C(f):\coprod_{m=0}^\iy C(\bX_m)\ra\coprod_{n=0}^\iy C(\bY_n)$ in $\cSMan$ satisfying \eq{sm4eq2}--\eq{sm4eq4}. Objects of $\cSManc$ will be called {\it s-manifolds with corners of mixed dimension}. We write $\cSMan,\cSManbin,\cSManb,\cSMancin$ for the subcategories of $\cSManc$ corresponding to $\SMan,\ldots,\SMancin$.

If $\bfX\in\SManc$ and $x\in X$ we write $x\in\bfX$, and we define the {\it tangent space\/} $T_x\bfX=T_x\bX$ for $T_x\bX$ as in Definition \ref{sm3def1}. If $\bff=(f,C(f)):\bfX\ra\bfY$ is smooth and $x\in\bfX$ with $f(x)=y\in\bfY$ (written $\bff(x)=y\in\bfY$) we write $T_x\bff:T_x\bfX\ra T_y\bfY$ for the tangent map $T_xf:T_x\bX\ra T_y\bY$ in Definition~\ref{sm3def1}.
\end{dfn}

\begin{dfn}
\label{sm4def5}
Let $\bff:\bfX\ra\bfY$ be a morphism of s-manifolds with corners. We call $\bff$ {\it b-normal\/} if:
\begin{itemize}
\setlength{\itemsep}{0pt}
\setlength{\parsep}{0pt}
\item[(i)] Let $x\in X$ with $f(x)=y\in Y$. Then $C(f)\vert_{C(X)_x}$, as a morphism of partial orders $(C(X)_x,\tl)\ra(C(Y)_y,\tl)$, preserves least upper bounds and greatest lower bounds. That is, if $\es\ne S\subseteq C(X)_x$ then $C(f)(\lub(S))=\lub(C(f)(S))$ and $C(f)(\glb(S))=\glb(C(f)(S))$.
\item[(ii)] $C(f)(C_k(\bX))\subseteq\coprod_{k'=0}^kC_{k'}(\bY)\subseteq C(\bY)$.
\item[(iii)] $(C(f)\t C(f))(D_{kl}(\bX))\subseteq\coprod_{ 0\le k'\le k, \;
0\le l'-k'\le l-k}D_{k'l'}(\bY)\subseteq D(\bY)$.
\end{itemize}
This generalizes b-normal maps of manifolds with corners in Definition \ref{sm2def4}. Note that (ii) for $k=0$ gives $C(f)(C_0(\bX))\subseteq C_0(\bY)$, so $\bff$ is interior.

Then compositions of b-normal maps, and identities, are b-normal. Thus we have subcategories $\SManbbn\subset\SManbin$, $\SMancbn\subset\SMancin$, $\cSManbbn\subset\cSManbin$ and $\cSMancbn\subset\cSMancin$ with morphisms b-normal maps. This gives a diagram of inclusions of subcategories
\begin{equation*}
\xymatrix@C=50pt@R=15pt{ 
& \,\,\SManbbn\,\,\, \ar@{^{(}->}[r] \ar@{^{(}->}[d] & \,\,\SMancbn \ar@{^{(}->}[d] \\
\SMan\,\,\, \ar@{^{(}->}[ur] \ar@{^{(}->}[r] \ar@{_{(}->}[dr] & \,\,\SManbin\,\,\, \ar@{^{(}->}[r] \ar@{^{(}->}[d] & \,\,\SMancin \ar@{^{(}->}[d] \\
& \,\,\SManb\,\,\, \ar@{^{(}->}[r] & \,\,\SManc. }	
\end{equation*}
\end{dfn}

\begin{rem}
\label{sm4rem1}
{\bf(a)} As in Example \ref{sm4ex1} below, Definitions \ref{sm4def1}--\ref{sm4def5} are modelled on properties of the categories $\Manc,\ab\Mangc,\ab\Manac$ from \S\ref{sm24}--\S\ref{sm27}. For manifolds with (g- or a-)corners $X$ we did not need to include the corners $C_k(X)$ in the definition of $X$, since we could construct $C_k(X)$ from the depth stratification $X=\coprod_{l=0}^nS^l(X)$ as in \S\ref{sm25}. There are several reasons why a similar construction would not work for s-manifolds with corners. 

Firstly, for manifolds with (g- or a-)corners $X$ we have $\ov{S^k(X)}=\coprod_{l=k}^nS^l(X)$, but for s-manifolds with corners $\bfX$ we know only that $\ov{S^k(X)}\subseteq\coprod_{l=k}^nS^l(X)$, so for example for $l>1$ we cannot assume a point of $S^l(X)$ is the limit of points of $S^1(X)$. Similarly, \eq{sm4eq5} below may not hold for s-manifolds with corners. Secondly, near a point $x\in X$ the strata $X^i$ of $X$ may have infinitely many connected components (e.g.\ see Example \ref{sm3ex4} at $x=(0,\ldots,0)$). These mean that the notions of `local boundary component' and `local $k$-corner component' in Definition \ref{sm2def7} are badly behaved for s-manifolds with corners.

Since we want our s-manifolds with corners $\bfX$ to have good notions of boundary $\pd\bfX$ and $k$-corners $C_k(\bfX)$, our solution is to include enough data in $\bfX$ to make the definitions of $\pd\bfX$ and  $C_k(\bfX)$ in \S\ref{sm43} easy.
\smallskip

\noindent{\bf(b)} If $\bfX=(\bX,\ldots)$ is an s-manifold with corners then smooth maps $\bff=(f,C(f)):\bfX\ra\R^n$ are equivalent to smooth maps of s-manifolds $f:\bX\ra\R^n$, with $C(f)=f\ci\Pi_\bX$. Therefore the material of \S\ref{sm35} on partitions of unity, embeddings, and the Whitney Embedding Theorem for s-manifolds, generalizes immediately to s-manifolds with corners, and we will not discuss it again.
\end{rem}

\subsection{Examples of s-manifolds with corners}
\label{sm42}

\begin{ex}
\label{sm4ex1}
{\bf(Manifolds with corners, g-corners, or a-corners.)} Let $X$ be a manifold with corners of dimension $n$, as in \S\ref{sm24}--\S\ref{sm25}. We will construct an s-manifold with corners $\bfX$ of dimension~$n$.

Write $\bX$ and $C_k(\bX)$ for $k=0,\ldots,n$ for the s-manifolds corresponding in Example \ref{sm3ex1} to the manifolds with corners $X,C_k(X)$, and $\Pi_k:C_k(\bX)\ra\bX$ for the s-manifold morphism corresponding in Example \ref{sm3ex1} to~$\Pi_k:C_k(X)\ra X$.

For $0\le k\le l\le n$, define a subset $D_{kl}(X)\subseteq C_k(X)\t_{\Pi_k,X,\Pi_l}C_l(X)$ by
\e
D_{kl}(X)=\ov{\bigl\{(v,w)\in C_k(X)\t_{\Pi_k,X,\Pi_l}C_l(X):w\in C_l(X)^\ci\bigr\}},
\label{sm4eq5}
\e
where the closure is in $C_k(X)\t_{\Pi_k,X,\Pi_l}C_l(X)$. By considering local models it is not difficult to show that $D_{kl}(X)$ is open and closed in $C_k(X)\t_{\Pi_k,X,\Pi_l}C_l(X)$, and is a disjoint union of fibre products of strata. In fact one can show that there is a unique isomorphism
\begin{equation*}
D_{kl}(X)\cong C_{l-k}(C_k(X))
\end{equation*}
which identifies $\Pi_{C_k(X)}:D_{kl}(X)\ra C_k(X)$ with $\Pi_{l-k}:C_{l-k}(C_k(X))\ra C_k(X)$.

Set $D_{kl}(\bX)=D_{kl}(X)$, and define $\bfX$ as in \eq{sm4eq1}. It is straightforward to check from properties of manifolds with corners that $\bfX$ satisfies Definition \ref{sm4def1}(a)--(f), so $\bfX$ is an s-manifold with corners. It has {\it simplicial corners\/} in the sense of Definition~\ref{sm4def2}. 

Now let $f:X\ra Y$ be a smooth map of manifolds with corners, and $\bfX,\bfY$ the s-manifolds with corners corresponding to $X,Y$. Then \S\ref{sm25} defines a morphism $C(f):C(X)\ra C(Y)$ in $\cMancin$ with $\Pi_Y\ci C(f)=f\ci\Pi_X$. By Example \ref{sm3ex1}, $f,C(f)$ give morphisms $f:\bX\ra\bY$ in $\SMan$ and $C(f):C(\bX)\ra C(\bY)$ in $\cSMan$. Define $\bff=(f,C(f))$. Then it is easy to see that $\bff:\bfX\ra\bfY$ is a morphism in $\SManc$, which is interior, or b-normal, if $f$ is interior, or b-normal.

This induces an inclusion functor $\Manc\hookra\SManc$, which is faithful, but not full. (Note that this does {\it not\/} coincide with the composition of $\Manc\hookra\SMan$ in Example \ref{sm3ex1} and $\SMan\hookra\SManc$ in Definition \ref{sm4def1}.)

All the above works without change for manifolds with g-corners in \S\ref{sm271} and manifolds with a-corners in \S\ref{sm271}, except that manifolds with g-corners $X$ such as $X_P$ in \eq{sm2eq23} can yield s-manifolds $\bfX$ with {\it non-simplicial\/} corners. This gives inclusion functors $\Mangc\hookra\SManc$ and $\Manac\hookra\SManc$, which are faithful, but not full.
\end{ex}

\begin{ex}
\label{sm4ex2}
{\bf(Orbifolds with corners.)} Generalizing orbifolds in Example \ref{sm3ex3}, there is a category (or 2-category) $\Orbc$ of {\it orbifolds with corners}, spaces locally modelled on $\R^n_k/G$ for $G\subset\GL(n,\R)$ a finite group preserving $\R^n_k\subseteq\R^n$. By combining Examples \ref{sm3ex3} and \ref{sm4ex1} we can define an faithful but not full inclusion $\Orbc\hookra\SManc$, such that for an orbifold with corners $X$ the corresponding s-manifold with corners $\bfX$ has stratification combining both the orbifold strata $X^{[G]}$ and the depth stratification $X=\coprod_{k=0}^{\dim X}S^k(X)$, and the corner structure $\bigl(C_k(\bX)_{0\le k}$, $D_{kl}(\bX)_{0\le k\le l}\bigr)$ is built using the depth stratification.
\end{ex}

\begin{ex}
\label{sm4ex3}
{\bf(Example \ref{sm3ex4} with $n=1$.)}
Let $\bX$ be as in Example \ref{sm3ex4} with $n=1$. Define $C_0(\bX)=\bX$, $\Pi_0=\id_\bX$, $C_1(\bX)=\{(0,0)\}$ considered as an s-manifold of dimension 0, and $\Pi_1$ to be the inclusion $\{(0,0)\}\hookra X$.
Set $D_{kl}(X)=C_k(X)\t_{\Pi_k,X,\Pi_l}C_l(X)$ for $0\le k\le 1$. Then $\bfX=\bigl(\bX,$ $\Pi_k:C_k(\bX)\ra\bX$, $k=0,1$, $D_{kl}(\bX),\; 0\le k\le l\le 1\bigr)$ is an s-manifold with boundary, of dimension~1.
\end{ex}

\begin{ex}
\label{sm4ex4}
{\bf(Example \ref{sm3ex5}.)}
Let $\bY$ be as in Example \ref{sm3ex5}. Define $C_0(\bY)=\bY$, $\Pi_0=\id_\bY$, $C_1(\bY)=\{(-1,0),(0,0)\}$ considered as an s-manifold of dimension 0, and $\Pi_1$ to be the inclusion $\{(-1,0),(0,0)\}\hookra Y$.
Set $D_{kl}(Y)=C_k(Y)\t_{\Pi_k,Y,\Pi_l}C_l(Y)$ for $0\le k\le 1$. Then $\bfY=\bigl(\bY,$ $\Pi_k:C_k(\bY)\ra\bY$, $k=0,1$, $D_{kl}(\bY),\; 0\le k\le l\le 1\bigr)$ is an s-manifold with boundary, of dimension~1.
\end{ex}

\subsection{Boundaries and corners of s-manifolds with corners}
\label{sm43}

We extend the material of \S\ref{sm25} to s-manifolds with corners.

\begin{dfn}
\label{sm4def6}
Let $\bfX$ be an s-manifold with corners of dimension $n$, and use the notation of Definition \ref{sm4def1}. For each $k=0,1,\ldots,n$ we will define the $k$-{\it corners\/} $C_k(\bfX)$ of $\bfX$, which is an s-manifold with corners of dimension $n-k$. If $n\ge 1$ we also define the {\it boundary\/} $\pd\bfX$ to be $\pd\bfX=C_1(\bfX)$, an s-manifold with corners of dimension $n-1$. As in \eq{sm4eq1} we write
\begin{align*}
C_k(\bfX)=\bigl(C_k(\bX),\; &\Pi_l:C_l(C_k(\bX))\ra\bX,\; 0\le l\le n-k, \\
&D_{lm}(C_k(\bX)),\; 0\le l\le m\le n-k\bigr),
\end{align*}
where for $D_{k(k+l)}(\bX)\subseteq C_k(X)\t_XC_{k+l}(X)$ as in \eq{sm4eq1} with the s-manifold structure of Definition \ref{sm4def1}(e), we define
\begin{equation*}
C_l(C_k(\bX)):=D_{k(k+l)}(\bX), \quad \Pi_l:=\Pi_{C_k(\bX)}:D_{k(k+l)}(\bX)\longra C_k(\bX).
\end{equation*}
Then for $0\le l\le m\le n-k$ we define
\ea
&D_{lm}(C_k(\bX))\subseteq D_{k(k+l)}(\bX)\t_{\Pi_{C_k(\bX)},C_k(X),\Pi_{C_k(\bX)}}D_{k(k+m)}(\bX)
\quad\text{by}
\label{sm4eq6}\\
&D_{lm}(C_k(\bX))=\bigl\{\bigl((u,v),(u,w)\bigr):u\in C_k(\bX),\; v\in C_{k+l}(\bX),\; w\in C_{k+m}(\bX),
\nonumber\\
&(u,v)\in D_{k(k+l)}(\bX),\; (u,w)\in D_{k(k+m)}(\bX),\; (v,w)\in D_{(k+l)(k+m)}(\bX)\bigr\}.
\nonumber
\ea

Observe that by mapping $\bigl((u,v),(u,w)\bigr)\mapsto\bigl((u,w),(v,w)\bigr)$ we may identify $D_{lm}(C_k(\bX))$ with 
\ea
&D_{lm}(C_k(\bX))\cong
\label{sm4eq7}\\
&\bigl\{\bigl((u,w),(v,w)\bigr)\in D_{k(k+m)}(\bX)\t_{\Pi_{C_{k+m}(\bX)},C_{k+m}(X),\Pi_{C_{k+m}(\bX)}}D_{(k+l)(k+m)}(\bX):
\nonumber\\
&(u,v)\in D_{k(k+l)}(\bX)\bigr\}.
\nonumber
\ea
If $\bigl((u,w),(v,w)\bigr)$ lies in the fibre product on the second line then $\Pi_k(u)=\Pi_{k+m}(w)=\Pi_{k+l}(v)$, so $(u,v)$ lies in $C_k(X)\t_XC_{k+l}(X)$. As $D_{k(k+l)}(\bX)$ is an open and closed subset of $C_k(X)\t_XC_{k+l}(X)$, we see that \eq{sm4eq7} is an open and closed subset of $D_{k(k+m)}(\bX)\t_{C_{k+m}(X)}D_{(k+l)(k+m)}(\bX)$.

We will verify that $C_k(\bfX)$ satisfies Definition \ref{sm4def1}(a)--(f). For (a), $C_k(\bX),\ab C_l(C_k(\bX))$ are s-manifolds of dimensions $n-k,n-k-l$ by Definition \ref{sm4def1}(a),(e), and $\Pi_l$ is a proper closed embedding as $D_{k(k+l)}(\bX)\subseteq C_k(\bX)\t_{\Pi_k,\bX,\Pi_{k+l}}C_{k+l}(\bX)$ is closed and $\Pi_{k+l}:C_{k+l}(\bX)\ra\bX$ is a proper closed embedding.

For (b), Definition \ref{sm4def1}(f)(ii) implies that $D_{kk}(\bX)=\De_{C_k(\bX)}=\bigl\{(u,u):u\in C_k(\bX)\bigr\}$. Thus $\Pi_2:D_{kk}(\bX)\ra C_k(\bX)$ is a bijection. As it is also a proper covering map of s-manifolds, it is an isomorphism, and (b) follows.

Part (c) holds as $D_{k(k+l)}(\bX)$ is a disjoint union of transverse fibre products $C_k(X)^i\t_{X^j}C_{k+l}(X)^{i'}$ by Definition \ref{sm4def1}(d), which are the strata of $D_{k(k+l)}(\bX)$. Since $\Pi_{k+l}:C_{k+l}(X)^{i'}\ra X^j$ is a proper covering map by Definition \ref{sm4def1}(a), the projection $C_k(X)^i\t_{X^j}C_{k+l}(X)^{i'}\ra C_k(X)^i$ is a proper covering map of manifolds.

For (d), $D_{lm}(C_k(\bX))$ is open and closed in $D_{k(k+l)}(\bX)\t_{C_k(\bX)}D_{k(k+m)}(\bX)$ as $D_{(k+l)(k+m)}(\bX)$ is open and closed in $C_{k+l}(X)\t_XC_{l+m}(X)$. It is a disjoint union of fibre products of strata as $D_{k(k+l)}(\bX),D_{k(k+m)}(\bX),D_{(k+l)(k+m)}(\bX)$ are. For (e), the projection 
\begin{equation*}
\Pi_{D_{k(k+m)}(\bX)}:D_{lm}(C_k(\bX))\longra D_{k(k+m)}(\bX)=:C_m(C_k(\bX))
\end{equation*}
is a proper covering map as in the presentation \eq{sm4eq7} of $D_{lm}(C_k(\bX))$, it is open and closed in $D_{k(k+m)}(\bX)\t_{C_{k+m}(X)}D_{(k+l)(k+m)}(\bX)$, and the projection
\begin{equation*}
D_{k(k+m)}(\bX)\t_{C_{k+m}(X)}D_{(k+l)(k+m)}(\bX)\longra D_{k(k+m)}(\bX)
\end{equation*}
is a proper covering map as $\Pi_{C_{k+m}(X)}:D_{(k+l)(k+m)}(\bX)\ra C_{k+m}(X)$ is a proper covering map by Definition \ref{sm4def1}(e) for~$\bfX$.

For (f), let $u\in C_k(\bX)$ with $\Pi_k(u)=x\in X$, and write $\tl$ for the partial order on $C(X)_x$ in Definition \ref{sm4def1}(f) for $\bfX$, and $\bl$ for the corresponding binary relation on $C(C_k(X))_u$. Then points of $C(C_k(X))_u$ are pairs $(u,v)$ such that $v\in C(X)_x$ with $u\tl v$, and for $(u,v),(u,w)\in C(C_k(X))_u$, we have $(u,v)\bl(u,w)$ if and only if $v\tl w$. Parts (f)(i)--(iii) for $\bl$ are now more-or-less immediate from the corresponding properties of $\tl$ applied to the second factor $v$ in~$(u,v)$. Hence Definition \ref{sm4def1}(a)--(f) hold, and $C_k(\bfX)$ is an s-manifold with corners.

If $v\in C(\bX)$ with $\Pi_\bX(v)=x$ then $v\in C(\bX)^\ci$ if and only if $v=\max(C(X)_x)$ in $(C(X)_x,\tl)$. Similarly, if $(u,v)\in C(C_k(\bX))$ with $\Pi_\bX(u)=\Pi_\bX(v)=x$ then $(u,v)\in C(C_k(\bX))^\ci$ if and only if $(u,v)$ is the maximum in $(C(C_k(X))_u,\bl)$, which is equivalent to $v=\max(C(X)_x)$ in $(C(X)_x,\tl)$, that is, $v\in C(\bX)^\ci$. Thus we see that
\e
C_l(C_k(\bX))^\ci=\bigl\{(u,v)\in D_{k(k+l)}(\bX):v\in C_{k+l}(\bX)^\ci\bigr\}.
\label{sm4eq8}
\e

Define a morphism $\bs\Pi_k:C_k(\bfX)\ra\bfX$ by
\begin{equation*}
\bs\Pi_k=\Bigl(\Pi_k:C_k(\bX)\ra\bX,\; \ts\coprod\limits_{l=k}^n\Pi_{C_{k+l}(\bX)}:\coprod\limits_{l=k}^nD_{kl}(\bX)\longra\coprod\limits_{l=0}^nC_l(\bX)\Bigr).
\end{equation*}
Here $\Pi_{C_{k+l}(\bX)}:D_{kl}(\bX)\ra C_{k+l}(\bX)$ is a morphism of s-manifolds by Definition \ref{sm4def1}(e). Equation \eq{sm4eq2} for $\bs\Pi_k$ is obvious. In \eq{sm4eq3}, $C(\Pi_k)\t C(\Pi_k)$ maps $((u,v),(u,w))$ in \eq{sm4eq6} to $(v,w)$, so \eq{sm4eq3} for $\bs\Pi_k$ follows from the condition $(v,w)\in D_{(k+l)(k+m)}(\bX)$ in \eq{sm4eq6}. Equation \eq{sm4eq4} for $\bs\Pi_k$ follows from \eq{sm4eq8}. Hence $\bs\Pi_k$ is a smooth map of s-manifolds with corners.

We also write $\bs i_\bfX:\pd\bfX\ra\bfX$ for $\bs\Pi_1:C_1(\bfX)\ra\bfX$. We can iterate boundaries to get $\bfX,\pd\bfX,\pd^2\bfX,\ldots,$ where $\pd^k\bfX$ is an s-manifold with corners of dimension $n-k$ for $k=0,1,\ldots,n$. It has a natural proper covering map $\pd^k\bfX\ra C_k(\bfX)$. We write $\bs i_\bfX^k:\pd^k\bfX\ra\bfX$ for the composition~$\bs i_\bfX\ci \bs i_{\pd\bfX}\ci\cdots\ci\bs i_{\pd^{k-1}\bfX}$.

We define the {\it corners\/} of $\bfX$ to be $C(\bfX)=\coprod_{k=0}^nC_k(\bfX)$, as an object of $\cSManc$. Then $\bs\Pi_\bfX=\coprod_{k=0}^n\bs\Pi_k$ is a morphism $\bs\Pi_\bfX:C(\bfX)\ra\bfX$ in $\cSManc$.

Now let $\bff=(f,C(f)):\bfX\ra\bfY$ be a morphism of s-manifolds with corners. Define a morphism in $\cSManc$ by
\begin{equation*}
C(\bff)=\bigl(C(f),C(f)\t C(f)\bigr):C(\bfX)\longra C(\bfY).
\end{equation*}
Here $C(f)$ is as in $\bff$, and $C(f)\t C(f):C(X)\t_XC(X)\ra C(Y)\t_YC(Y)$ maps $D(\bX)=C(C(\bX))\ra D(\bY)=C(C(\bY))$ by \eq{sm4eq3}. Equation \eq{sm4eq2} is obvious, and \eq{sm4eq3} follows from \eq{sm4eq6} for $C(\bfX),C(\bfY)$ and \eq{sm4eq3} for $\bff$. Equation \eq{sm4eq4} for $C(\bff)$ follows from \eq{sm4eq8} for $\bfX,\bfY$ and \eq{sm4eq4} for $\bff$. Thus $C(\bff)$ is a morphism in $\cSManc$. From Definition \ref{sm4def1}(f)(ii) we see that $C_0(C(\bX))=\De_{C(\bX)}$, so $C(C(f))=C(f)\t C(f)$ maps $C_0(C(\bX))\ra C_0(C(\bY))$, and $C(\bff)$ is interior. 

If $\bff$ is b-normal, we can show that Definition \ref{sm4def5}(i)--(iii) for $C(\bff)$ follow from Definition \ref{sm4def5}(i)--(iii) for $\bff$, so $C(\bff)$ is b-normal. It is easy to see that $\bff\mapsto C(\bff)$ is compatible with identities and compositions, so we have defined a functor
\begin{equation*}
C:\SManc\longra\cSMancin\subset\cSManc,
\end{equation*}
which also maps $\SMancbn\ra\cSMancbn$. As the constructions of corners for manifolds with (g-) or (a-)corners and for s-manifolds with corners are compatible, the following diagrams commute up to natural isomorphism, where the inclusions are as in Example \ref{sm4ex1}:
\begin{align*}
\xymatrix@C=2pt@R=15pt{
*+[r]{\Manc\,\,\,\,\,} \ar[d]^C & \ar@{^{(}->}[rrrrr] &&&&& *+[l]{\SManc} \ar[d]_C &
*+[r]{\!\!\!\!\!\!\!\Mangc\,\,\,} \ar@<-10pt>[d]^C & \ar@{^{(}->}[rrrrr] &&&&& *+[l]{\SManc} \ar[d]_C &
*+[r]{\!\!\!\!\!\!\!\Manac\,\,\,} \ar@<-10pt>[d]^C & \ar@{^{(}->}[rrrrr] &&&&& *+[l]{\SManc} \ar[d]_C
\\  
*+[r]{\cManc\,\,} & \ar@{^{(}->}[rrrrr] &&&&& *+[l]{\cSManc,\!} &
*+[r]{\!\!\!\!\!\!\!\cMangc\,\,\,} & \ar@{^{(}->}[rrrrr] &&&&& *+[l]{\cSManc,\!} &
*+[r]{\!\!\!\!\!\!\!\cManac\,\,\,} & \ar@{^{(}->}[rrrrr] &&&&& *+[l]{\cSManc.\!} 
}
\end{align*}
As \eq{sm2eq10}--\eq{sm2eq13} are false in $\Mangc$ by \S\ref{sm271}(h), they are also false in~$\SManc$.
\end{dfn}

As in \cite[\S 4]{Joyc1}, we can make {\it b-normal\/} morphisms act on boundaries:

\begin{dfn}
\label{sm4def7}
Let $\bff:\bfX\ra\bfY$ be a b-normal morphism in $\SManc$. Then Definition \ref{sm4def5}(ii) implies that $C(f)\vert_{C_1(\bX)}$ maps $C_1(\bX)\ra C_0(\bY)\amalg C_1(\bY)$. Definition \ref{sm4def6} makes $C_1(\bX)$ into an s-manifold with corners $\pd\bfX$. Define $\pd^\bff_0\bfX$ and $\pd^\bff_1\bfX$ to be the open and closed s-submanifolds of $\pd\bfX$ mapped to $C_0(\bY)$ and $C_1(\bY)$ respectively, so that $\pd\bfX=\pd^\bff_0\bfX\amalg\pd^\bff_1\bfX$. Define morphisms in $\SManc$
\begin{align*}
\pd_0\bff&:\pd^\bff_0\bfX\longra\bfY,& \pd_1\bff&:\pd^\bff_1\bfX\longra\pd\bfY\quad\text{by}\\
\pd_0\bff&=\Pi_\bfY\ci C(\bff)\vert_{\pd^\bff_0\bfX}=\bff\ci\Pi_\bfX\vert_{\pd^\bff_0\bfX},
& \pd^\bff_1\bfX&=C(\bff)\vert_{\pd^\bff_1\bfX}.
\end{align*}
As $C_0(\bfY)\cong\bfY$, we see that $\pd_0\bff\amalg\pd_1\bff:\pd^\bff_0\bfX\amalg\pd^\bff_1\bfX\ra\bfY\amalg\pd\bfY$ is identified with $C(\bff)\vert_{C_1(\bfX)}:C_1(\bfX)\ra C_0(\bfY)\amalg C_1(\bfY)$. We use this notation in Theorem~\ref{sm4thm2}.	
\end{dfn}

\subsection{\texorpdfstring{Orientations and `fundamental classes'}{Orientations and ‘fundamental classes’}}
\label{sm44}

Section \ref{sm26} explained our point of view on `fundamental classes' of manifolds with corners. We now extend the material of \S\ref{sm33} on orientations and fundamental classes of s-manifolds to the corners case, such that \S\ref{sm26}(i)--(vii) hold.

\begin{dfn}
\label{sm4def8}
Let $\bfX$ be an s-manifold with corners of dimension $n$. Then $\pd^k\bfX$ is an s-manifold with corners of dimension $n-k$, so its interior $(\pd^k\bfX)^\ci$ is an s-manifold of dimension $n-k$, for $k=0,\ldots,n$. An {\it orientation\/} $\bs o_\bfX=(o_\bfX^0,o_\bfX^1,\ldots,o_\bfX^n)$ on $\bfX$ is an orientation $o_\bfX^k$ on $(\pd^k\bfX)^\ci$ for $k=0,\ldots,n$ in the sense of Definition \ref{sm3def2}, satisfying the following condition~\eq{sm4eq10}.

For each $k=0,\ldots,n-1$, as in \eq{sm3eq1} we have the open subsets $(\pd^k\bfX)_0^\ci\subseteq(\pd^k\bfX)^\ci\subseteq\pd^k\bfX$ and $(\pd^{k+1}\bfX)_0^\ci\subseteq(\pd^{k+1}\bfX)^\ci\subseteq\pd^{k+1}\bfX$, which are manifolds of dimensions $n-k,n-k-1$ with orientations $o_\bfX^k$, $o_\bfX^{k+1}$, giving fundamental classes $[(\pd^k\bfX)_0^\ci]_\fund\in H_{n-k}^\St((\pd^k\bfX)_0^\ci,\Z)$ and $[(\pd^{k+1}\bfX)_0^\ci]_\fund\in H_{n-k-1}^\St((\pd^{k+1}\bfX)_0^\ci,\Z)$. By Definition \ref{sm4def3} the morphism $\bs i_{\pd^k\bfX}\vert_{(\pd^{k+1}\bfX)_0^\ci}:(\pd^{k+1}\bfX)_0^\ci\ra\pd^k\bfX$ is a homeomorphism with a locally closed subset of $\pd^k\bfX$, which thus has the structure of an oriented manifold, and is disjoint from $(\pd^k\bfX)_0^\ci$. Define
\begin{equation*}
(\pd^k\bfX)_{0,\le 1}=(\pd^k\bfX)_0^\ci\amalg \bs i_{\pd^k\bfX}\bigl((\pd^{k+1}\bfX)_0^\ci\bigr)\subseteq\pd^k\bfX.
\end{equation*}
Then as for \eq{sm3eq2}, Property \ref{sm2pr2}(g) gives a long exact sequence
\e
\begin{gathered}
\xymatrix@C=25pt@R=15pt{  & 0=H_{n-k}^\St(\bs i_{\pd^k\bfX}\bigl((\pd^{k+1}\bfX)_0^\ci\bigr),\Z) \ar[r]_(0.57){\inc_*} & H_{n-k}^\St((\pd^k\bfX)_{0,\le 1},\Z) \ar[d]^{\vert_{(\pd^k\bfX)_0^\ci}} \\
\cdots  & H_{n-k-1}^\St(\bs i_{\pd^k\bfX}\bigl((\pd^{k+1}\bfX)_0^\ci\bigr),\Z) \ar[l] & H_{n-k}^\St((\pd^k\bfX)_0^\ci,\Z). \ar[l]_(0.39)\pd }
\end{gathered}
\label{sm4eq9}
\e
In a similar way to \eq{sm3eq3}, we require that for $k=0,\ldots,n-1$ we have
\e
\pd\bigl([(\pd^k\bfX)_0^\ci]_\fund\bigr)=(\bs i_{\pd^k\bfX}\vert_{(\pd^{k+1}\bfX)_0^\ci})_*\bigl([(\pd^{k+1}\bfX)_0^\ci]_\fund\bigr)
\label{sm4eq10}
\e
in $H_{n-k-1}^\St(\bs i_{\pd^k\bfX}\bigl((\pd^{k+1}\bfX)_0^\ci\bigr),\Z)$.

Observe that if \eq{sm4eq10} holds, then the orientation $o_\bfX^k$ on $(\pd^k\bfX)_0^\ci$ determines the fundamental class of $(\pd^{k+1}\bfX)_0^\ci$, and thus the orientation $o_\bfX^{k+1}$ on $(\pd^{k+1}\bfX)_0^\ci$. Hence by induction on $k$, all of $o_\bfX^0,o_\bfX^1,\ldots,o_\bfX^n$ are determined by the orientation $o_\bfX^0$ on $\bfX^\ci=(\pd^0\bfX)^\ci$. So we can equivalently define an orientation on $\bfX$ to be just an orientation $o_\bfX=o_\bfX^0$ on $\bfX^\ci$, such that there exist (necessarily unique) orientations $o_\bfX^k$ on $(\pd^k\bfX)^\ci$ for $k=1,\ldots,n$ satisfying~\eq{sm4eq10}.

If $\bs o_\bfX=(o_\bfX^0,o_\bfX^1,\ldots,o_\bfX^n)$ is an orientation on $\bfX$, then $\bs o_{\pd^k\bfX}=(o_\bfX^k,\ldots,o_\bfX^n)$ is an orientation on $\pd^k\bfX$ for $k=0,\ldots,n$, as in \S\ref{sm26}(i). We call $\bfX$ {\it oriented\/} if it has a particular choice of orientation, and {\it orientable\/} if it admits an orientation, and {\it locally orientable\/} if $\bfX$ can be covered with open s-submanifolds with corners $\bfU\subseteq\bfX$ with $\bfU$ orientable. If $\bfX$ is an oriented s-manifold with corners, we write $-\bfX$ for $\bfX$ with the opposite orientation.
\end{dfn}

\begin{dfn}
\label{sm4def9}
Let $\bfX$ be an s-manifold with corners of dimension $n$. An {\it orientation bundle\/} $(\Or_\bfX,\bs\om_\bfX)$ on $\bX$ consists of a principal $\Z_2$-bundle $\pi:\Or_\bfX\ra\bfX$ over $\bfX$, so that $(i_\bfX^k)^*(\Or_\bfX)\ra\pd^k\bfX$ is a principal $\Z_2$-bundle over $\pd^k\bfX$ for $k=0,\ldots,n$, and an $(n+1)$-tuple $\bs\om_\bfX=(\om_\bfX^0,\om_\bfX^1,\ldots,\om_\bfX^n)$, such that:
\begin{itemize}
\setlength{\itemsep}{0pt}
\setlength{\parsep}{0pt}
\item[(a)] $\om_\bfX^k:\Or_{(\pd^k\bfX)_0^\ci}\ra(i_\bfX^k)^*(\Or_\bfX)\vert_{(\pd^k\bfX)_0^\ci}$ is an isomorphism of principal $\Z_2$-bundles over $(\pd^k\bfX)_0^\ci\subseteq \pd^k\bfX$, such that $\bigl((i_\bfX^k)^*(\Or_\bfX)\vert_{(\pd^k\bfX)^\ci},\om_\bfX^k\bigr)$ is an orientation bundle on the s-manifold $(\pd^k\bfX)^\ci$ for $k=0,\ldots,n$.
\item[(b)] Twisting \eq{sm4eq9} by $(i_\bfX^k)^*(\Or_\bfX)$ gives a long exact sequence
\begin{equation*}
\!\!\!\!\!\xymatrix@C=120pt@R=15pt{  *+[r]{H_{n-k}^\St((\pd^k\bfX)_{0,\le 1},(i_\bfX^k)^*(\Or_\bfX)\vert_{\cdots},\Z)}\ar[r]_(0.54){\raisebox{-9pt}{$\vert_{(\pd^k\bfX)_0^\ci}$}} & *+[l]{H_{n-k}^\St((\pd^k\bfX)_0^\ci,(i_\bfX^k)^*(\Or_\bfX)\vert_{\cdots},\Z)} \ar@<-2ex>[d]_\pd
\\
*+[r]{\cdots}  & *+[l]{H_{n-k-1}^\St(\bs i_{\pd^k\bfX}\bigl((\pd^{k+1}\bfX)_0^\ci\bigr),(i_\bfX^k)^*(\Or_\bfX)\vert_{\cdots},\Z).} \ar[l] }
\end{equation*}
Then as for \eq{sm4eq10}, we require that for $k=0,\ldots,n-1$ we have
\end{itemize}
\e
\pd\bigl((\om_\bfX^k)_*\bigl([(\pd^k\bfX)_0^\ci]_\fund\bigr)\bigr)=(\bs i_{\pd^k\bfX}\vert_{(\pd^{k+1}\bfX)_0^\ci})_*\bigl((\om_\bfX^{k+1})_*\bigl([(\pd^{k+1}\bfX)_0^\ci]_\fund\bigr)\bigr)
\label{sm4eq11}
\e
\begin{itemize}
\setlength{\itemsep}{0pt}
\setlength{\parsep}{0pt}
\item[] in $H_{n-k-1}^\St(\bs i_{\pd^k\bfX}\bigl((\pd^{k+1}\bfX)_0^\ci\bigr),(i_\bfX^k)^*(\Or_\bfX)\vert_{\cdots},\Z)$.
\end{itemize}

Observe that if \eq{sm4eq11} holds, then $\om_\bfX^k$ determines $(\om_\bfX^{k+1})_*([(\pd^{k+1}\bfX)_0^\ci]_\fund)$, and so determines $\om_\bfX^{k+1}$. Hence by induction on $k$, all of $\om_\bfX^0,\om_\bfX^1,\ldots,\om_\bfX^n$ are determined by $\om_\bfX^0$. So we can equivalently define an orientation bundle on $\bfX$ to be a pair $(\Or_\bfX,\om_\bfX^0)$, such that there exist (necessarily unique) $\om_\bfX^k$ as in (i) for $k=1,\ldots,n$ satisfying~\eq{sm4eq11}.

If $(\Or_\bfX,(\om_\bfX^0,\om_\bfX^1,\ldots,\om_\bfX^n)\bigr)$ is an orientation bundle on $\bfX$, then $\bs i_{\pd^k\bfX}^*(\Or_\bfX,\ab\bs\om_\bfX)\ab:=\bigl((i^k_\bfX)^*(\Or_\bfX),\ab(\om_\bfX^k,\ldots,\om_\bfX^n)\bigr)$ is an orientation bundle on $\pd^k\bfX$ for $k=0,\ldots,n$. As in Definition \ref{sm3def3}(i)--(iii), it is easy to see that:
\begin{itemize}
\setlength{\itemsep}{0pt}
\setlength{\parsep}{0pt}
\item[(i)] If an s-manifold with corners $\bfX$ has an orientation bundle, it is locally orientable. 
\item[(ii)] If $\bfX$ is oriented then it has an orientation bundle $(\Or_\bfX,\bs\om_\bfX)$ with $\Or_\bfX=X\t\Z_2$, and $\om_\bfX^k:\Or_{(\pd^k\bfX)_0^\ci}\ra(i_\bfX^k)^*(\Or_\bfX)\vert_{(\pd^k\bfX)_0^\ci}=(\pd^k\bfX)_0^\ci\t\Z_2$ the trivialization of $\Or_{(\pd^k\bfX)_0^\ci}$ induced by the orientation on $(\pd^k\bfX)_0^\ci$.
\item[(iii)] If $\bfX$ has an orientation bundle $(\Or_\bfX,\bs\om_\bfX)$, any trivialization $\Or_\bfX\cong X\t\Z_2$ induces an orientation on $\bfX$.
\end{itemize}
We say that an orientation on $\bfX$ is {\it compatible with\/} $(\Or_\bfX,\bs\om_\bfX)$ if it corresponds to a trivialization~$\Or_\bfX\cong X\t\Z_2$. 
\end{dfn}

Here is the analogue of \S\ref{sm26}(ii)--(vi) for s-manifolds with corners. Note that part (c) is not obvious, and uses deep properties of Steenrod homology, since $X$ could be a pathological topological space, as in Examples \ref{sm3ex5}, \ref{sm4ex4} and \ref{sm4ex5}(d). 

\begin{thm}
\label{sm4thm1}
Let\/ $\bfX$ be an oriented s-manifold with corners of dimension $n,$ and use the notation of Definitions\/ {\rm\ref{sm4def1}} and\/ {\rm\ref{sm4def3}}. Write $S^{\le 1}(X)=S^0(X)\amalg S^1(X)\subset X,$ so that\/ $S^{\le 1}(X)$ is a locally closed subset of\/ $X,$ with\/ $S^0(X)=\bfX^\ci$ an open subset of $S^{\le 1}(X),$ whose complement\/ $S^1(X)$ is homeomorphic to $(\pd\bfX)^\ci$ under $i_\bfX:\pd\bfX\ra\bfX$. Thus by Property\/ {\rm\ref{sm2pr2}(g)} we have an exact sequence
\e
\xymatrix@C=24pt{ H_n^\St(S^0(X),\Z) \ar[r]^(0.48)\pd & H_{n-1}^\St(S^1(X),\Z) \ar[r]^(0.48){\inc_*} & H_{n-1}^\St\bigl(S^{\le 1}(X),\Z\bigr). }
\label{sm4eq12}
\e

\noindent{\bf(a)} Suppose $\bfX$ is oriented, so that\/ $\pd\bfX$ is also oriented. Then the s-manifolds $\bfX^\ci,(\pd\bfX)^\ci$ are oriented and have fundamental classes in Steenrod homology by Theorem\/ {\rm\ref{sm3thm1}(a)}. For $\pd$ as in {\rm\eq{sm4eq12},} these satisfy
\e
\pd\bigl([\bfX^\ci]_\fund\bigr)=(i_\bfX)_*\bigl([(\pd\bfX)^\ci]_\fund\bigr).
\label{sm4eq13}
\e
As $i_\bfX:(\pd\bfX)^\ci\ra S^1(X)$ is a homeomorphism, this implies that\/ $[\bfX^\ci]_\fund$ determines $[(\pd\bfX)^\ci]_\fund,$ and so by induction $[\bfX^\ci]_\fund$ determines $[(\pd^k\bfX)^\ci]_\fund$ for $k=0,\ldots,n$. Also exactness of\/ \eq{sm4eq12} implies that
\e
\inc_*\ci(i_\bfX)_*\bigl([(\pd\bfX)^\ci]_\fund\bigr)=0
\label{sm4eq14}
\e
in $H_{n-1}^\St\bigl(S^{\le 1}(X),\Z\bigr)$. If\/ $\bfX$ is an s-manifold with boundary then $S^{\le 1}(X)=X,$ and\/ $(\pd\bfX)^\ci=\pd\bfX$ is an s-manifold without boundary, and\/ {\rm\eq{sm4eq13}--\eq{sm4eq14}} become
\ea
\pd\bigl([\bfX^\ci]_\fund\bigr)&=(i_\bfX)_*\bigl([\pd\bfX]_\fund\bigr),
\label{sm4eq15}\\
\inc_*\ci(i_\bfX)_*\bigl([\pd\bfX]_\fund\bigr)&=0,
\label{sm4eq16}
\ea
where \eq{sm4eq16} holds in $H_{n-1}^\St(X,\Z)$.
\smallskip

\noindent{\bf(b)} Suppose $\bfX$ has an orientation bundle $(\Or_\bfX,\bs\om_\bfX)$. Then we have a twisted version of\/ \eq{sm4eq12}
\begin{equation*}
\text{\begin{small}$\displaystyle\xymatrix@C=15pt{ H_n^\St(S^0(X),\Or_\bfX\vert_{\cdots},\Z) \ar[r]^(0.48)\pd & H_{n-1}^\St(S^1(X),\Or_\bfX\vert_{\cdots},\Z) \ar[r]^(0.48){\inc_*} & H_{n-1}^\St\bigl(S^{\le 1}(X),\Or_\bfX\vert_{\cdots},\Z\bigr), }$\end{small}}
\end{equation*}
and under this the fundamental classes of\/ $\bfX^\ci,(\pd\bfX)^\ci$ from Theorem\/ {\rm\ref{sm3thm1}(b)} satisfy {\rm\eq{sm4eq13}--\eq{sm4eq14}} in Steenrod homology twisted by $\Or_\bfX$. If\/ $\bfX$ is an s-manifold with boundary then {\rm\eq{sm4eq15}--\eq{sm4eq16}} hold.
\smallskip

\noindent{\bf(c)} Now suppose that $\bfX$ is a compact oriented s-manifold with boundary of dimension $n=1$. Then $\pd\bfX$ is an ordinary compact oriented\/ $0$-manifold, that is, a finite set with signs. Writing $\#(\pd\bfX)\in\Z$ for the signed count, we have
\begin{equation*}
\#(\pd\bfX)=0.
\end{equation*}
\end{thm}

\begin{proof}
For (a), consider the commutative diagram with exact rows and columns
\e
\begin{gathered}
\xymatrix@C=20pt@R=15pt{
0 \ar[d] & 0 \ar[d]  \\ 
H_n^\St(S^0(X),\Z) \ar[d]^{\vert_{S^0(X)_0}} \ar[r]^(0.48)\pd & H_{n-1}^\St(S^1(X),\Z) \ar[r]^(0.48){\inc_*} \ar[d]^{\vert_{S^1(X)_0}} & H_{n-1}^\St\bigl(S^{\le 1}(X),\Z\bigr) \ar[d]^{\vert_{S^{\le 1}(X)_0}} \\
H_n^\St(S^0(X)_0,\Z)  \ar[r]^(0.48)\pd & H_{n-1}^\St(S^1(X)_0,\Z) \ar[r]^(0.48){\inc_*} & H_{n-1}^\St\bigl(S^{\le 1}(X)_0,\Z\bigr).   }
\end{gathered}
\label{sm4eq17}
\e
Here $S^i(X)_0$ are the subsets $S^i(\bX)_0$ in \eq{sm3eq1} for the s-manifolds $S^i(\bX)$, and $S^{\le 1}(X)_0=S^0(X)_0\amalg S^1(X)_0$, as a subset of $S^{\le 1}(X)$. The columns of \eq{sm4eq17} are from \eq{sm3eq12} for $S^0(\bX),S^1(\bX),S^{\le 1}(\bX)$. The rows are the exact sequences \eq{sm4eq12} for $\bfX,\bfX_0$, see also \eq{sm4eq9} for~$k=0$.

By Theorem \ref{sm3thm1}(a) we have $[\bfX^\ci]_\fund\in H_n^\St(S^0(X),\Z)$ with $[\bfX^\ci]_\fund\vert_{S^0(X)_0}\ab=[S^0(X)_0]_\fund$ in $H_n^\St(S^0(X)_0,\Z)$ and $(i_\bfX)_*\bigl([(\pd\bfX)^\ci]_\fund\bigr)\in H_{n-1}^\St(S^1(X),\Z)$ with $(i_\bfX)_*\bigl([(\pd\bfX)^\ci]_\fund\bigr)\vert_{S^1(X)_0}=[S^1(X)_0]_\fund$ in $H_{n-1}^\St(S^1(X)_0,\Z)$. But \eq{sm4eq10} for $k=0$ gives $\pd\bigl([S^0(X)_0]_\fund\bigr)=[S^1(X)_0]_\fund$, so from \eq{sm4eq17} we see that \eq{sm4eq13} holds, and \eq{sm4eq14} is immediate. 

Part (b) holds by the same argument in Steenrod homology twisted by $\Or_\bfX$. For (c), as $X$ is compact the projection $\pi:X\ra *$ is proper, so $\pi_*:H_0^\St(X,\Z)\ra H_0^\St(*,\Z)=\Z$ is defined by Property \ref{sm2pr2}(b). Applying $\pi_*$ to \eq{sm4eq16} gives
\begin{equation*}
\#(\pd\bfX)=\pi_*\ci\inc_*\ci(i_\bfX)_*\bigl([\pd\bfX]_\fund\bigr)=0\quad\text{in $H_0^\St(*,\Z)=\Z$.}\qedhere
\end{equation*}
\end{proof}

\begin{ex}
\label{sm4ex5}
{\bf(a) (Manifolds with (g- or a-)corners.)} As in Example \ref{sm4ex1}, manifolds with (g- or a-)corners can be made into s-manifolds with corners. This identifies the usual notions of orientations and orientation bundles for manifolds with (g- or a-)corners with those for s-manifolds with corners.
\smallskip

\noindent{\bf(b) (Orbifolds with corners.)} As in Example \ref{sm4ex2}, orbifolds with corners $X$ can be made into s-manifolds with corners $\bfX$. Generalizing Example \ref{sm3ex6}(c), $\bfX$ is locally orientable if and only if $X$ is, and then orientations and orientation bundles for orbifolds with corners $X$ and s-manifolds with corners $\bfX$ agree.
\smallskip

\noindent{\bf(c)} The s-manifold with boundary $\bfX$ in Examples \ref{sm3ex4} and \ref{sm4ex3} does not admit an orientation, since in \eq{sm4eq10} we always have $\pd\bigl([(\pd^0\bfX)_0^\ci]_\fund\bigr)=0$ by the proof in Example \ref{sm3ex6}(d) that 
$\bX$ in Example \ref{sm3ex4} admits an orientation.
\smallskip

\noindent{\bf(d)} The s-manifold with boundary $\bfX$ in Examples \ref{sm3ex5} and \ref{sm4ex4} does admit an orientation. Take the interval $[-1,0]\t\{0\}$ Examples \ref{sm3ex5} to have the usual orientation of $[-1,0]$, and each circle $\cS^1_{1/k}$ to have an arbitrary orientation, and for the orientation of $\pd\bfX$ we give $(-1,0)$ the negative orientation and $(0,0)$ the positive orientation. Then we can show \eq{sm4eq10} holds for $k=0$ by the argument of Example~\ref{sm3ex6}(e).
\end{ex}

\subsection{\texorpdfstring{$\Z_2$-fundamental classes}{ℤ₂-fundamental classes}}
\label{sm45}

The material of \S\ref{sm34} on $\Z_2$-fundamental classes of s-manifolds generalizes easily to s-manifolds with corners.

\begin{dfn}
\label{sm4def10}
Let $\bfX$ be an s-manifold with corners of dimension $n$. We say that $\bfX$ {\it admits a $\Z_2$-fundamental class\/} if:
\begin{itemize}
\setlength{\itemsep}{0pt}
\setlength{\parsep}{0pt}
\item[{\bf(i)}] The s-manifold $(\pd^k\bfX)^\ci$ admits a $\Z_2$-fundamental class in the sense of Definition \ref{sm3def4} for $k=0,\ldots,n$.
\item[{\bf(ii)}] We have an analogue of \eq{sm4eq9} over $\Z_2$:
\begin{equation*}
\xymatrix@C=20pt@R=15pt{  & 0=H_{n-k}^\St(\bs i_{\pd^k\bfX}\bigl((\pd^{k+1}\bfX)_0^\ci\bigr),\Z_2) \ar[r]_(0.57){\inc_*} & H_{n-k}^\St((\pd^k\bfX)_{0,\le 1},\Z_2) \ar[d]^{\vert_{(\pd^k\bfX)_0^\ci}} \\
\cdots  & H_{n-k-1}^\St(\bs i_{\pd^k\bfX}\bigl((\pd^{k+1}\bfX)_0^\ci\bigr),\Z_2) \ar[l] & H_{n-k}^\St((\pd^k\bfX)_0^\ci,\Z_2). \ar[l]_(0.39)\pd }
\end{equation*}
As for \eq{sm4eq10}, we require that for $k=0,\ldots,n-1$ we have
\begin{equation*}
\pd\bigl([(\pd^k\bfX)_0^\ci]_\fund\bigr)=(\bs i_{\pd^k\bfX}\vert_{(\pd^{k+1}\bfX)^\ci})_*\bigl([(\pd^{k+1}\bfX)_0^\ci]_\fund\bigr),
\end{equation*}
where the fundamental classes are of manifolds in $\Z_2$-homology, and so do not need orientations.
\end{itemize}

If $\bfX$ admits a $\Z_2$-fundamental class then so does $\pd^k\bfX$ for~$k=0,\ldots,n$.

The analogue of Theorem \ref{sm4thm1}(a) then holds for the $\Z_2$-fundamental classes of $\bfX^\ci,\pd\bfX^\ci$, in Steenrod homology over $\Z_2$. The analogue of Theorem \ref{sm4thm1}(c) holds with $\#(\pd\bfX)=0$ in $\Z_2$.
\end{dfn}

\subsection{Transverse fibre products of s-manifolds with corners}
\label{sm46}

Here are analogues of Definitions \ref{sm3def6}, \ref{sm3def7}, Remark \ref{sm3rem4}, and Theorem~\ref{sm3thm4}.

\begin{dfn}
\label{sm4def11}
Let $\bfX,\bfY,\bfZ$ be s-manifolds with corners with $\dim\bfX=l$, $\dim\bfY=m$, $\dim\bfZ=n$, and $\bfg=(g,C(g)):\bfX\ra\bfZ$, $\bfh=(h,C(h)):\bfY\ra\bfZ$ be b-normal morphisms. We say that $\bfg,\bfh$ are {\it transverse\/} if:
\begin{itemize}
\setlength{\itemsep}{0pt}
\setlength{\parsep}{0pt}
\item[(a)] $g:\bX\ra\bZ$ and $h:\bY\ra\bZ$ are transverse in $\SMan$ in the sense of Definition~\ref{sm3def6}.
\item[(b)] $C(g):C(\bX)\ra C(\bZ)$ and $C(h):C(\bY)\ra C(\bZ)$ are transverse in $\cSMan$ in the sense of Remark~\ref{sm3rem5}.
\end{itemize}
Actually (b) implies (a) by Definition \ref{sm4def1}(b), as $\bfg,\bfh$ are b-normal, so interior.

We call a b-normal morphism $\bfh:\bfY\ra\bfZ$ a {\it submersion\/} if $h:\bY\ra\bZ$ and $C(h):C(\bY)\ra C(\bZ)$ are submersions in the sense of Definition \ref{sm3def6} (actually $C(h)$ a submersion implies $h$ is). If $\bfh$ is a submersion then $\bfg,\bfh$ are transverse for any b-normal~$\bfg:\bfX\ra\bfZ$.
\end{dfn}

\begin{rem}
\label{sm4rem2}
{\bf(a)} Let $\bfg:\bfX\ra\bfZ$ be a morphism in $\SManc$ and $x'\in C_k(\bX)$ with $\Pi_k(x')=x\in\bX$, $C(g)(x')=z'\in C_l(\bZ)$, and $g(x)=\Pi_l(z')=z\in\bZ$. Applying tangent maps to equation \eq{sm4eq2} for $\bfg$ gives a commutative diagram
\e
\begin{gathered}
\xymatrix@R=13pt@C=115pt{ *+[r]{T_{x'}C_k(\bX)} \ar[r]_{T_{x'}C(g)} \ar[d]^{T_{x'}\Pi_k}_\cong & *+[l]{T_{z'}C_l(\bZ)}
\ar[d]^{T_{z'}\Pi_l}_\cong \\ *+[r]{T_x\bX} \ar[r]^{T_xg} & *+[l]{T_z\bZ.\!} }
\end{gathered}
\label{sm4eq18}
\e
The columns $T_{x'}\Pi_k,T_{z'}\Pi_l$ are isomorphisms as by Definition \ref{sm4def1}(c), $\Pi_k,\Pi_l$ are covering maps of manifolds on strata.

In Definition \ref{sm4def11}(a),(b) we require $g,h$ and $C(g),C(h)$ to be transverse. This involves two conditions: a surjectivity condition on tangent spaces in Definition \ref{sm3def6}(a), and dimension conditions in Definition \ref{sm3def6}(b).

Because of the commuting diagrams \eq{sm4eq18} for $\bfX,\bfY,\bfZ$ with columns isomorphisms, Definition \ref{sm3def6}(a) for $g,h$ and $C(g),C(h)$ are {\it equivalent}, and we can impose only the conditions for $g,h$. But Definition \ref{sm3def6}(b) for $C(g),C(h)$ is generally a stronger restriction than Definition \ref{sm3def6}(b) for~$g,h$.
\smallskip

\noindent{\bf(b)} As in Remark \ref{sm3rem4}, an important special case of Definition \ref{sm4def11} is when $\bfZ$ is an ordinary $n$-manifold without boundary. Then Definition \ref{sm3def6}(b) for $g,h$ and $C(g),C(h)$ is automatic, so by {\bf(a)}, $\bfg:\bfX\ra\bfZ$, $\bfh:\bfY\ra\bfZ$ are {\it transverse\/} if $T_x\bfg\op T_y\bfh:T_x\bfX\op T_y\bfY\ra T_z\bfZ$ is surjective whenever $x\in\bfX$ and $y\in\bfY$ with $\bfg(x)=\bfh(y)=z\in\bfZ$, and $\bfh:\bfY\ra\bfZ$ is a {\it submersion\/} if $T_y\bfh:T_y\bfY\ra T_z\bfZ$ is surjective whenever $y\in\bfY$ with $\bfh(y)=z\in\bfZ$, exactly as in Definition~\ref{sm2def3}.
\smallskip

\noindent{\bf(c)} In Definition \ref{sm4def11} we require $\bfg,\bfh$ to be {\it b-normal}. This is necessary for \eq{sm4eq25} and \eq{sm4eq26} in Theorem \ref{sm4thm2} to hold.
\end{rem}

\begin{dfn}
\label{sm4def12}
Work in the situation of Definition \ref{sm4def11}, and write $\bfg=(g,C(g))$ and $\bfh=(h,C(h))$. Then Theorem \ref{sm3rem4} and Remark \ref{sm3rem5} give Cartesian squares in $\SMan$ and $\cSMan$: 
\begin{gather}
\begin{gathered}
\xymatrix@R=13pt@C=115pt{ *+[r]{\bW} \ar[r]_f \ar[d]^e & *+[l]{\bY}
\ar[d]_h \\ *+[r]{\bX} \ar[r]^g & *+[l]{\bZ,\!} }
\end{gathered}
\label{sm4eq19}\\
\begin{gathered}
\xymatrix@R=13pt@C=100pt{ *+[r]{C(\bW)} \ar[r]_{C(f)} \ar[d]^{C(e)} & *+[l]{C(\bY)}
\ar[d]_{C(h)} \\ *+[r]{C(\bX)} \ar[r]^{C(g)} & *+[l]{C(\bZ),\!} }
\end{gathered}
\label{sm4eq20}
\end{gather}
where $\bX,\bY,\bZ,C(\bX),C(\bY),C(\bZ),g,h,C(g),C(h)$ are part of the data of $\bfX,\ab\bfY,\ab\bfZ,\ab\bfg,\ab\bfh$, and \eq{sm4eq19}--\eq{sm4eq20} are the definitions of $\bW,C(\bW),e,f,C(e),C(f)$.

As $C(\bW)$ is an object of $\cSMan$ it is a disjoint union of s-manifolds of different dimensions. The contribution from $C^i(\bX)\t_{C^k(\bZ)}C^j(\bY)$ has dimension $l+m-n-i-j+k$. As $\bfg,\bfh$ are b-normal this is only nonempty if $i,j\ge k$, so the dimension is at most $d:=l+m-n$. Define $C_h(\bW)$ to be the part of $C(\bW)$ which is an s-manifold of dimension $d-h$. Then $C(\bW)=\coprod_{h=0}^dC_h(\bW)$. 

Consider the diagram in $\cSMan$:
\e
\begin{gathered}
\xymatrix@R=13pt@C=50pt{ *+[r]{C(\bW)} \ar[rr]_{C(f)} \ar@{..>}[dr]^{\Pi_\bW} \ar[dd]^{C(e)} && C(\bY)
\ar[dd]_(0.25){C(h)} \ar[dr]^{\Pi_\bY} \\ 
& \bW \ar[rr]_(0.25)f \ar[dd]^(0.25)e && *+[l]{\bY} \ar[dd]_h \\
*+[r]{C(\bX)} \ar[dr]_{\Pi_\bX} \ar[rr]^(0.25){C(g)} && C(\bZ) \ar[dr]^{\Pi_\bZ}\\
& \bX \ar[rr]^g && *+[l]{\bZ.\!} }
\end{gathered}
\label{sm4eq21}
\e
The two bottom right parallelograms commute by \eq{sm4eq2}, and the two rectangles commute by \eq{sm4eq19}--\eq{sm4eq20}. As \eq{sm4eq19} is Cartesian in $\SMan$, and so in $\cSMan$, there exists a unique morphism $\Pi_\bW:C(\bW)\ra\bW$ in $\cSMan$ making the diagram commute. Define $\Pi_h:C_h(\bW)\ra\bW$ to be the restriction of $\Pi_\bW$ to $C_h(\bW)\subseteq C(\bW)$ for $h=0,\ldots,d$. 

From \eq{sm4eq21} we see we have a commutative diagram of topological spaces:
\e
\begin{gathered}
\xymatrix@R=13pt@C=180pt{ *+[r]{C(W)\t_{\Pi_W,W,\Pi_W}C(W)} \ar[r]_{C(f)\t C(f)} \ar[d]^{C(e)\t C(e)} & *+[l]{C(Y)\t_{\Pi_Y,Y,\Pi_Y}C(Y)}
\ar[d]_{C(h)\t C(h)} \\*+[r]{C(X)\t_{\Pi_X,X,\Pi_X}C(X)} \ar[r]^{C(g)\t C(g)} & *+[l]{C(Z)\t_{\Pi_Z,Z,\Pi_Z}C(Z),\!} }
\end{gathered}
\label{sm4eq22}
\e
which is in fact a Cartesian square. As in Definition \ref{sm4def1} we have open and closed $D(\bX)\subseteq C(X)\t_{\Pi_X,X,\Pi_X}C(X)$, and similarly for $D(\bY),D(\bZ)$, and \eq{sm4eq3} gives
$\bigl(C(g)\t C(g)\bigr)\bigl(D(\bX)\bigr)\subseteq D(\bZ)$ and $\bigl(C(h)\t C(h)\bigr)\bigl(D(\bY)\bigr)\subseteq D(\bZ)$. Define
\e
D(\bW)=D(\bX)\t_{C(g)\t C(g),D(\bZ),C(h)\t C(h)}D(\bY),
\label{sm4eq23}
\e
as a subset of $C(W)\t_{\Pi_W,W,\Pi_W}C(W)$ regarded as a fibre product as in \eq{sm4eq22}. It is open and closed, as $D(\bX),D(\bY),D(\bZ)$ are. Define $D_{kl}(\bW)$ to be the intersection of $D(\bW)$ with~$C_k(W)\t_{\Pi_k,W,\Pi_l}C_l(W)$. 

Suppose $(w',w'')\in D_{h_1h_2}(\bW)$ with $(C(e)\t C(e))(w',w'')=(x',x'')$ in $D_{i_1i_2}(\bX)$, $(C(f)\t C(f))(w',w'')=(y',y'')\!\in\! D_{j_1j_2}(\bY)$, and $(C(g)\t C(g))(x',x'')\ab=(C(h)\t C(h))(y',y'')=(z',z'')\in D_{k_1k_2}(\bZ)$. As $\bfg,\bfh$ are b-normal, Definition \ref{sm4def5}(iii) implies that $i_2-i_1,j_2-j_1\ge k_2-k_1\ge 0$. But $h_2-h_1=i_2-i_1+j_2-j_1-k_2+k_1$, so $h_1\le h_2$. Thus $D_{kl}(\bW)=\es$ unless $k\le l$, so $D(\bW)=\coprod_{0\le k\le l\le d}D_{kl}(\bW)$. Define 
\begin{equation*}
\bfW=\bigl(\bW,\; \Pi_k:C_k(\bW)\ra\bW,\; 0\le k\le d, \; D_{kl}(\bW),\; 0\le k\le l\le d\bigr),
\end{equation*}
and write $\bfe=(e,C(e))$ and $\bff=(f,C(f))$.
\end{dfn}

\begin{thm}
\label{sm4thm2}
In Definition\/ {\rm\ref{sm4def12},} $\bfW$ is an s-manifold with corners with\/ $\dim\bfW=d:=l+m-n,$ where $d\ge 0$ if\/ $\bfW\ne\es,$ and\/ $\bfe:\bfW\ra\bfX,$ $\bff:\bfW\ra\bfY$ are b-normal (hence interior) morphisms with\/ $\bfg\ci\bfe=\bfh\ci\bff$. The following is a Cartesian square in both\/ $\SManc$ and\/~$\SMancin\!:$
\e
\begin{gathered}
\xymatrix@R=13pt@C=90pt{ *+[r]{\bfW} \ar[r]_\bff \ar[d]^\bfe & *+[l]{\bfY}
\ar[d]_\bfh \\ *+[r]{\bfX} \ar[r]^\bfg & *+[l]{\bfZ.\!} }
\end{gathered}
\label{sm4eq24}
\e
Thus\/ $\bfW$ is a fibre product\/ $\bfX\t_{\bfg,\bfZ,\bfh}\bfY$ in $\SManc$ and\/~$\SMancin$.

The following is also a Cartesian square in both\/ $\cSManc$ and\/ $\cSMancin,$ with\/ $C(\bfg),C(\bfh)$ transverse:
\e
\begin{gathered}
\xymatrix@R=13pt@C=90pt{ *+[r]{C(\bfW)} \ar[r]_{C(\bff)} \ar[d]^{C(\bfe)} & *+[l]{C(\bfY)}
\ar[d]_{C(\bfh)} \\ *+[r]{C(\bfX)} \ar[r]^{C(\bfg)} & *+[l]{C(\bfZ).\!} }
\end{gathered}
\label{sm4eq25}
\e

Using the notation of Definition\/ {\rm\ref{sm4def7},} there is a canonical isomorphism
\e
\pd\bfW\cong \bigl(\pd_0^\bfg\bfX\t_{\pd_0\bfg,\bfZ,\bfh}\bfY\bigr)\amalg\bigl(\bfX\t_{\bfg,\bfZ,\pd_0\bfh}\pd_0^\bfh\bfY\bigr)\amalg\bigl(\pd_1^\bfg\bfX\t_{\pd_1\bfg,\pd\bfZ,\pd_1\bfh}\pd_1^\bfh\bfY\bigr),
\label{sm4eq26}
\e
where the three fibre products on the right hand side are transverse in\/~$\SManc$.
\end{thm}

\begin{proof}
First we verify Definition \ref{sm4def1}(a)--(f) for $\bfW$. For (a), $\bW,C_k(\bW)$ are s-manifolds of the required dimensions by Definition~\ref{sm4def12}. 

To show $\Pi_\bW$ is proper, let $K\subseteq W$ be compact. Then $e(K)\subseteq X$ and $f(K)\subseteq Y$ are compact as $e,f$ are continuous, so $\Pi_\bX^{-1}(e(K))$ and $\Pi_\bY^{-1}(f(K))$ are compact as $\Pi_\bX,\Pi_\bY$ are proper. Thus $\Pi_\bX^{-1}(e(K))\t\Pi_\bY^{-1}(f(K))$ is a compact subset of $C(X)\t C(Y)$. Now $(C(e),C(f)):C(W)\ra C(X)\t C(Y)$ is a closed embedding (and so proper) as $C(W)=C(X)\t_{C(Z)}C(Y)$ is a topological fibre product and $C(Z)$ is Hausdorff. Hence $(C(e),C(f))^{-1}\bigl(\Pi_\bX^{-1}(e(K))\t\Pi_\bY^{-1}(f(K))\bigr)$ is compact. But from \eq{sm4eq21} we see this is $\Pi_W^{-1}(K)$. Hence $\Pi_\bW$ is proper, and $\Pi_k$ is proper for  $k=0,\ldots,d$. We deduce that $\Pi_\bW$ is locally a closed embedding from \eq{sm4eq19}--\eq{sm4eq21} and the fact that $\Pi_\bX,\Pi_\bY,\Pi_\bZ$ are locally closed embeddings. Thus $\Pi_k$ is locally a closed embedding, completing~(a).

For (b), as $\bfg,\bfh$ are b-normal they are interior, $C(g),C(h)$ map $C_0(\bX),C_0(\bY)\ab\ra C_0(\bZ)$, so $C_0(\bX)\t_{C_0(\bZ)}C_0(\bY)\subseteq C_0(\bW)\subseteq C(\bW)=C(\bX)\t_{C(\bZ)}C(\bY)$. Also, no other $C_i(\bX)\t_{C_k(\bZ)}C_j(\bY)$ can contribute to $C_0(\bW)$, as for this to be nonempty we must have $i,j\ge k$ by Definition \ref{sm4def5}(ii) since $\bfg,\bfh$ are b-normal, so $C_0(\bW)=C_0(\bX)\t_{C_0(\bZ)}C_0(\bY)$. As $\Pi_0:C_0(\bX)\ra\bX$, $\Pi_0:C_0(\bY)\ra\bY$, $\Pi_0:C_0(\bZ)\ra\bZ$ are isomorphisms, $\Pi_0:C_0(\bW)\ra\bW$ is an isomorphism.

For (c), in \eq{sm4eq21} the morphisms $\Pi_\bX,\Pi_\bY,\Pi_\bZ$ are proper covering maps on strata by Definition \ref{sm4def1}(c) for $\bfX,\bfY,\bfZ$. By the Cartesian properties of \eq{sm4eq19}--\eq{sm4eq21}, $\Pi_\bW$ is also a proper covering map on strata.

The first part of (d) holds by Definition \ref{sm4def12}, and being a disjoint union of fibre products of strata follows from \eq{sm4eq23}, Definition \ref{sm4def1}(d) for $\bfX,\bfY,\bfZ$, and~\eq{sm4eq19}--\eq{sm4eq21}.

For (e), we have $D(\bW)\subseteq C(W)\t_WC(W)$. Write $\Pi_2^\bW:D(\bW)\ra C(W)$ for projection to the second $C(W)$ factor. As $C(W)=C(X)\t_{C(Z)}C(Y)$, we see from \eq{sm4eq22}--\eq{sm4eq23} that $\Pi_2^\bW$ is identified with the morphism of fibre products 
\begin{equation*}
\Pi_2^\bX\t_{\Pi_2^\bZ}\Pi_2^\bY:D(\bX)\t_{D(\bZ)}D(\bY)\longra C(W)\t_{C(Z)}C(Y).	
\end{equation*}
Now $\Pi_2^\bX,\Pi_2^\bY,\Pi_2^\bZ$ are proper covering maps by Definition \ref{sm4def1}(e) for $\bfX,\bfY,\bfZ$. Hence $\Pi_2^\bW$ is a proper covering map, and (e) follows.

For (f), let $(x,y)\in W$ with $g(x)=h(y)=z\in Z$. Write $\tl_X,\tl_Y,\tl_Z$ for the partial orders on $C(X)_x,C(Y)_y,C(Z)_z$ in Definition \ref{sm4def1}(f) for $\bfX,\bfY,\bfZ$, and $\tl_W$ for the corresponding binary relation on $C(W)_{(x,y)}$. Suppose $(t_1,u_2),(t_2,u_2)\in C(W)_{(x,y)}$. Then $t_a\in C(X)_x$, $u_a\in C(Y)_y$ with $C(g)(t_a)=C(h)(u_a)=v_a\in C(Z)_z$ for $a=1,2$. From \eq{sm4eq23} we see that $(t_1,u_2)\tl_W(t_2,u_2)$ if $t_1\tl_Xt_2$, $u_1\tl_Yu_2$, and $v_1\tl_Zv_2$. Part (f)(i) is immediate from $\tl_X,\tl_Y,\tl_Z$ being partial orders. For (f)(ii), let $t_a\in C_{i_a}(X)_x$, $u_a\in C_{j_a}(Y)_y$, $v_a\in C_{k_a}(Z)_z$ for $a=1,2$. Then $(t_a,u_a)\in C_{h_a}(W)_{(x,y)}$ where $h_a=i_a+j_a-k_a$. We have $i_1\le i_2$ with $i_1=i_2$ if and only if $t_1=t_2$, and $j_1\le j_2$ with $j_1=j_2$ if and only if $u_1=u_2$, and $k_1\le k_2$ with $k_1=k_2$ if and only if $v_1=v_2$. Also $i_2-i_1,j_2-j_1\ge k_2-k_1$ by Definition \ref{sm4def5}(iii) as $\bfg,\bfh$ are b-normal. Noting that $t_1=t_2$ or $u_1=u_2$ force $v_1=v_2$, all these imply that if $h_1=h_2$ then $(t_1,u_1)=(t_2,u_2)$, proving~(f)(ii). 

For (f)(iii), let $S=\bigl\{(t_a,u_a):a\in A\bigr\}$ be a nonempty subset of $C(W)_{(x,y)}$, and set $v_a=C(g)(t_a)=C(h)(u_a)\in C(Z)_z$. Write $t=\lub(\{t_a:a\in A\})\in C(X)_x$, $u=\lub(\{u_a:a\in A\})\in C(Y)_y$, and $v=\lub(\{v_a:a\in A\})\in C(Z)_z$. Then $C(g)(t)=C(h)(u)=v$ by Definition \ref{sm4def5}(i), as $\bfg,\bfh$ are b-normal. So $(t,u)\in C(X)_x\t_{C(Z)_z}C(Y)_y=C(W)_{(x,y)}$, and it is easy to check from the description of $\tl_W$ above that $(t,u)=\lub(S)$. So least upper bounds exist in $(C(W)_{(x,y)},\tl_W)$, and similarly greatest lower bounds do, proving (f)(iii). Note that this also implies that $C(e)$ and $C(f)$ map preserve least upper bounds and greatest lower bounds, as required for $\bfe,\bff$ to be b-normal. Thus $\bfW$ is an s-manifold with corners.

We will show that
\e
C(\bW)^\ci=C(\bX)^\ci\t_{C(g)\vert_{\cdots},C(\bZ)^\ci,C(h)\vert_{\cdots}}C(\bY)^\ci,
\label{sm4eq27}
\e
as an open subset of the fibre product $C(\bW)=C(\bX)\t_{C(\bZ)}C(\bY)$ from \eq{sm4eq20}, noting that $C(g)(C(\bX)^\ci),C(h)(C(\bY)^\ci)\subseteq C(\bZ)^\ci\subseteq C(\bZ)$ by \eq{sm4eq4}. Let $(u_1,v_1)\ab\in C(\bW)$, so that $C(g)(u_1)=C(h)(v_1)=w_1\in C(\bZ)$, and $\Pi_X(u_1)=x\in X$, $\Pi_Y(v_1)=y\in Y$, and $\Pi_Z(w_1)=g(x)=h(y)=z\in Z$. Write $u_2=\max(C(X)_x)$, $v_2=\max(C(Y)_y)$ and $w_2=\max(C(Z)_z)$. Then $u_2\in C(\bX)^\ci$, $v_2\in C(\bY)^\ci$, $w_2\in C(\bZ)^\ci$ by Definition \ref{sm4def3}, and $C(g)(u_2)=C(h)(v_2)=w_2$ by \eq{sm4eq4}, so $(u_2,v_2)\in C(\bW)$ by \eq{sm4eq20}. Also $(u_1,u_2)\in D(\bX)$, $(v_1,v_2)\in D(\bY)$, $(v_1,v_2)\in D(\bZ)$ by Definition \ref{sm4def1}(f), so $((u_1,v_1),(u_2,v_2))\in D(\bW)$ by \eq{sm4eq23}. Hence $(u_1,v_1)\in C(\bW)^\ci$ if and only if $(u_1,v_1)=(u_2,v_2)$. This holds if and only if $(u_1,v_1)\in C(\bX)^\ci\t_{C(\bZ)^\ci}C(\bY)^\ci$, proving~\eq{sm4eq27}.

By definition $e,f$ and $C(e),C(f)$ are morphisms in $\SMan,\cSMan$, and \eq{sm4eq2} commutes for $\bfe,\bff$ by \eq{sm4eq21}. Equations \eq{sm4eq3}--\eq{sm4eq4} for $\bfe,\bff$ follow from \eq{sm4eq23} and \eq{sm4eq27}. Hence $\bfe:\bfW\ra\bfX$, $\bff:\bfW\ra\bfY$ are smooth maps. To show $\bfe$ is b-normal, we proved Definition \ref{sm4def5}(i) for $\bfe$ above, and (ii),(iii) for $\bfe$ follow easily from (ii),(iii) for $\bfh$. Thus $\bfe$ is b-normal. Similarly $\bff$ is b-normal.

Equations \eq{sm4eq19}--\eq{sm4eq20} imply that $\bfg\ci\bfe=\bfh\ci\bff$, so \eq{sm4eq24} commutes.

To prove that \eq{sm4eq24} is Cartesian in $\SManc$, suppose $\bfe':\bfW'\ra\bfX$, $\bff':\bfW'\ra\bfY$ are morphisms in $\SManc$ with $\bfg\ci\bfe'=\bfh\ci\bff'$. Then $g\ci e'=h\ci f'$ and $C(g)\ci C(e')=C(h)\ci C(f')$, so the Cartesian squares  \eq{sm4eq19}--\eq{sm4eq20} give unique morphisms $b:\bW'\ra\bW$ in $\SMan$ and $C(b):C(\bW')\ra C(\bW)$ in $\cSMan$ with
\e
e'=e\ci b,\quad 	f'=f\ci b,\quad C(e')=C(e)\ci C(b),\quad C(f')=C(f)\ci C(b).
\label{sm4eq28}
\e
Set $\bfb=(b,C(b))$. We see that \eq{sm4eq2} commutes for $\bfb$ using \eq{sm4eq21} and the Cartesian property of \eq{sm4eq19}. Equations \eq{sm4eq3}--\eq{sm4eq4} for $\bfb$ follow from \eq{sm4eq3}--\eq{sm4eq4} for $\bfe',\bff'$ and \eq{sm4eq23}, \eq{sm4eq27}. Hence $\bfb$ is a morphism in $\SManc$. Equation \eq{sm4eq28} implies that $\bfe'=\bfe\ci\bfb$, $\bff'=\bff\ci\bfb$, and $\bfb$ is unique with this property by uniqueness of $b,C(b)$. Therefore \eq{sm4eq24} is Cartesian in $\SManc$. If $\bfe',\bff'$ are interior then as $C_0(\bW)=C_0(\bX)\t_{C_0(\bZ)}C_0(\bY)$ we see that $\bfb$ is also interior. Thus \eq{sm4eq24} is also Cartesian in $\SMancin$.

Next we show that $C(\bfg),C(\bfh)$ are transverse in $\cSManc$. That is, we must show that $C(g),C(h)$ are transverse and $C(C(g)),C(C(h))$ are transverse in the sense of Definition \ref{sm3def6}. But $C(g),C(h)$ are transverse by Definition \ref{sm4def11} as $\bfg,\bfh$ are transverse. Definition \ref{sm4def6} defines $C_l(C_k(\bX))=D_{k(k+l)}(\bX)$, where $D_{k(k+l)}(\bX)\!\subseteq\! C_k(\bX)\!\t_\bX\! C_{k+l}(\bX)$, and the projection $\Pi_{C_{k+l}(\bX)}:D_{k(k+l)}(\bX)\ra C_{k+l}(\bX)$ is a proper covering map by Definition \ref{sm4def1}(e). In a similar way to Remark \ref{sm4rem2}(a), it is easy to check that the conditions in Definition \ref{sm3def6} for $C(C(g)),C(C(h))$ to be transverse are mapped under $\Pi_{C_{k+l}(\bX)}:C_l(C_k(\bX))\ra C_{k+l}(\bX)$ and the analogues for $\bfY,\bfZ$ to the conditions in Definition \ref{sm3def6} for $C(g),C(h)$ to be transverse. Hence $C(\bfg),C(\bfh)$ are transverse.

Applying $C:\SMancbn\ra\cSMancbn$ to \eq{sm4eq24} shows that \eq{sm4eq25} is a commutative square in $\cSMancbn$. To prove that \eq{sm4eq25} is Cartesian in $\cSManc$ and $\cSMancin$, we can check that our explicit constructions of the corner functor $C$ in Definition \ref{sm4def6}, and the fibre product $\bfW=\bfX\t_{\bfg,\bfZ,\bfh}\bfY$ in Definition \ref{sm4def12}, commute, so that $C(\bfW)$ in Definition \ref{sm4def6} is equivalent to the fibre product $C(\bfW)=C(\bfX)\t_{C(\bfg),C(\bfZ),C(\bfh)}C(\bfY)$ constructed as in Definition \ref{sm4def12} (though in $\cSManc$ rather than~$\SManc$).

As $\bfe,\bff,\bfg,\bfh$ are b-normal, by Definition \ref{sm4def5}(ii) we may restrict \eq{sm4eq25} to
\begin{equation*}
\xymatrix@R=13pt@C=170pt{ *+[r]{C_1(\bfW)} \ar[r]_{C(\bff)\vert_{C_1(\bfW)}} \ar[d]^{C(\bfe)\vert_{C_1(\bfW)}} & *+[l]{C_0(\bfY)\amalg C_1(\bfY)}
\ar[d]_{C(\bfh)\vert_{C_{\le 1}(\bfY)}} \\ *+[r]{C_0(\bfX)\amalg C_1(\bfX)} \ar[r]^{C(\bfg)\vert_{C_{\le 1}(\bfX)}} & *+[l]{C_0(\bfZ)\amalg C_1(\bfZ).\!} }
\end{equation*}
Using $C_0(\bfX)\cong\bfX$, $C_1(\bfX)=\pd\bfX$, etc., and the notation of Definition \ref{sm4def7}, we identify this with the commutative diagram
\e
\begin{gathered}
\xymatrix@R=13pt@C=160pt{ *+[r]{\pd\bfW=\pd^\bfe_0\bfW\amalg\pd^\bfe_1\bfW=\pd^\bff_0\bfW\amalg\pd^\bff_1\bfW} \ar[r]_(0.65){\pd_0\bff\amalg\pd_1\bff} \ar[d]^{\pd_0\bfe\amalg\pd_1\bfe} & *+[l]{\bfY\amalg \pd^\bfh_0\bfY\amalg\pd^\bfh_1\bfY}
\ar[d]_{\bfh\amalg\pd_0\bfh\amalg\pd_1\bfh} \\ *+[r]{\bfX\amalg \pd^\bfg_0\bfX\amalg\pd^\bfg_1\bfX} \ar[r]^{\bfg\amalg\pd_0\bfg\amalg\pd_1\bfg} & *+[l]{\bfZ\amalg \pd\bfZ.\!} }
\end{gathered}
\label{sm4eq29}
\e
The fibre product of the three bottom right hand terms in \eq{sm4eq29} is a Cartesian subsquare of \eq{sm4eq25}, and is the disjoint union of five individual fibre products, the three in \eq{sm4eq26} together with $\bfX\t_\bfZ\bfY$, which corresponds to $C_0(\bfW)$ in \eq{sm4eq25}, and $\pd^\bfg_0\bfX\t_{\pd_0\bfg,\bfZ,\pd_0\bfh}\pd^\bfh_0\bfY$, which corresponds to part of $C_2(\bfW)$ in \eq{sm4eq25}. Equation \eq{sm4eq26} follows, and the proof of Theorem \ref{sm4thm2} is complete.
\end{proof}

\subsection{\texorpdfstring{Oriented transverse fibre products in $\SManc$}{Oriented transverse fibre products in SManᶜ}}
\label{sm47}

As in \S\ref{sm37}, we consider when the transverse fibre products in $\SManc$ in Theorem \ref{sm4thm2} are compatible with orientations or orientation bundles. Example \ref{sm3ex7} identifies a problem with this, which we solved in Definition \ref{sm3def8} by defining {\it strong transversality\/} to exclude the case Definition~\ref{sm3def6}(b)(iv).

For s-manifolds with corners, the analogue of Definition \ref{sm3def6}(b)(iv) is a problem for orientations whenever $\bfZ$ has a boundary. In particular, it is difficult to ensure \eq{sm4eq10} and \eq{sm4eq11} hold for $\bfW$ on the component $\pd_1^\bfg\bfX\t_{\pd\bfZ}\pd_1^\bfh\bfY$ in \eq{sm4eq26}. We deal with this by assuming $\bfZ$ is {\it without boundary}.

\begin{dfn}
\label{sm4def13}
Let $\bfX,\bfY,\bfZ$ be s-manifolds with corners with $\dim\bfX=l$, $\dim\bfY=m$, $\dim\bfZ=n$, where $\bfZ$ is without boundary, and $\bfg=(g,C(g)):\bfX\ra\bfZ$, $\bfh=(h,C(h)):\bfY\ra\bfZ$ be morphisms (which are automatically b-normal as $\bfZ$ is without boundary). We say that $\bfg,\bfh$ are {\it strongly transverse\/} if:
\begin{itemize}
\setlength{\itemsep}{0pt}
\setlength{\parsep}{0pt}
\item[(a)] $g:\bX\ra\bZ$ and $h:\bY\ra\bZ$ are strongly transverse in $\SMan$ in the sense of Definition~\ref{sm3def8}.
\item[(b)] $C(g):C(\bX)\ra C(\bZ)$ and $C(h):C(\bY)\ra C(\bZ)$ are strongly transverse in $\cSMan$.
\end{itemize}
Actually (b) implies (a). Also $\bfg,\bfh$ strongly transverse implies $\bfg,\bfh$ transverse.	

As in Remark \ref{sm4rem2}(b), an important special case is when $\bfZ$ is an ordinary $n$-manifold without boundary. Then $\bfg:\bfX\ra\bfZ$, $\bfh:\bfY\ra\bfZ$ are {\it strongly transverse\/} if $T_x\bfg\op T_y\bfh:T_x\bfX\op T_y\bfY\ra T_z\bfZ$ is surjective whenever $x\in\bfX$ and $y\in\bfY$ with $\bfg(x)=\bfh(y)=z\in\bfZ$.
\end{dfn}

Observe that if $\bfZ$ is without boundary in Theorem \ref{sm4thm2} then in equation \eq{sm4eq26} we have $\pd_0^\bfg\bfX=\pd\bfX$, $\pd_0\bfg=\bfg\ci\bs i_\bfX$, $\pd_0^\bfh\bfY=\pd\bfY$, $\pd_0\bfh=\bfh\ci\bs i_\bfY$, and the final term $\pd_1^\bfg\bfX\t_{\pd\bfZ}\pd_1^\bfh\bfY$ is empty, so \eq{sm4eq26} becomes
\begin{equation*}
\pd\bfW\cong \bigl(\pd\bfX\t_{\bfg\ci\bs i_\bfX,\bfZ,\bfh}\bfY\bigr)\amalg \bigl(\bfX\t_{\bfg,\bfZ,\bfh\ci\bs i_\bfY}\pd\bfY\bigr).
\end{equation*}
Iterating this gives for $c=0,\ldots,d$
\e
\pd^c\bfW\cong \coprod_{a,b\ge 0:a+b=c}\bigl(\pd^a\bfX\t_{\bfg\ci\bs i^a_\bfX,\bfZ,\bfh\ci\bs i^b_\bfY}\pd^b\bfY\bigr)^{\coprod^{\binom{c}{a}}},
\label{sm4eq30}
\e
where $(\cdots)^{\coprod^{\binom{c}{a}}}$ means the disjoint union of $\binom{c}{a}$ copies of $(\cdots)$. The binomial coefficients arise as the $c$ ordered applications of $\pd$ in $\pd^c\bfW$ can each act on either $\bfX$ or $\bfY$. Here is the analogue of Theorem~\ref{sm3thm5}.

\begin{thm}
\label{sm4thm3}
{\bf(a)} In Theorem\/ {\rm\ref{sm4thm2},} suppose\/ $\bfZ$ is without boundary, $\bfg,\bfh$ are strongly transverse, and\/ $\bfX,\bfY,\bfZ$ are oriented. Then the transverse fibre product\/ $\bfW=\bfX\t_{\bfg,\bfZ,\bfh}\bfY$ has a natural orientation. This depends on an orientation convention, which we take to be that in Akaho--Joyce\/ {\rm\cite[\S 2.4]{AkJo}} and Fukaya--Oh--Ohta--Ono\/ {\rm\cite[\S 8.2]{FOOO}}. It depends on the order of\/ $\bfX,\bfY,$ and in oriented s-manifolds with corners we have
\ea
\bfX\t_{\bfg,\bfZ,\bfh}\bfY&\cong (-1)^{(l-n)(m-n)}\bfY\t_{\bfh,\bfZ,\bfg}\bfX,
\label{sm4eq31}\\
\pd\bfW&\cong \bigl(\pd\bfX\t_{\bfg\ci\bs i_\bfX,\bfZ,\bfh}\bfY\bigr)\amalg (-1)^{l-n}\bigl(\bfX\t_{\bfg,\bfZ,\bfh\ci\bs i_\bfY}\pd\bfY\bigr).
\label{sm4eq32}
\ea

\noindent{\bf(b)} In Theorem\/ {\rm\ref{sm4thm2},} suppose\/ $\bfZ$ is without boundary, $\bfg,\bfh$ are strongly transverse, and\/ $\bfX,\bfY,\bfZ$ have orientation bundles $(\Or_\bfX,\bs\om_\bfX),\ab(\Or_\bfY,\bs\om_\bfY),\ab(\Or_\bfZ,\bs\om_\bfZ)$. Then we may define an orientation bundle $(\Or_\bfW,\bs\om_\bfW=(\om_\bfW^0,\ldots,\om_\bfW^d))$ on\/ $\bfW,$ where $\Or_\bfW$ and\/ $\om_\bfW^0$ are defined as in {\rm\eq{sm3eq22}--\eq{sm3eq25},} and\/ $\om_\bfW^k$ for $k>0$ are determined by\/ $\om_\bfW^0$ as in Definition\/~{\rm\ref{sm4def9}}.
\end{thm}

\begin{proof}
As orientations on s-manifolds with corners $\bfX$ correspond to orientation bundles $(\Or_\bfX,\bs\om_\bfX)$ with $\Or_\bfX$ trivial, the first part of (a) follows from (b) in the case when $\Or_\bfX,\ab\Or_\bfY,\ab\Or_\bfZ$ are trivial principal $\Z_2$-bundles, so that $\Or_\bfW$ in \eq{sm3eq22} is also trivial. Having proved the first part of (a), equations \eq{sm4eq31}--\eq{sm4eq32} follow from the corresponding equations \cite[Prop.~2.10]{AkJo}, \cite[Lem.~8.2.3]{FOOO} for fibre products of the oriented manifolds $X_0,Y_0,Z_0$, and their boundaries. So it is enough to prove~(b).

Work in the situation of (b). Taking interiors in \eq{sm4eq30} gives
\e
(\pd^c\bfW)^\ci\cong \coprod_{a,b\ge 0:a+b=c}\bigl((\pd^a\bfX)^\ci\t_{\bfg\ci\bs i^a_\bfX,\bfZ,\bfh\ci\bs i^b_\bfY}(\pd^b\bfY)^\ci\bigr)^{\coprod^{\binom{c}{a}}},
\label{sm4eq33}
\e
where the fibre products are strongly transverse fibre products of s-manifolds, not s-manifolds with corners. By Definition \ref{sm4def9}, the s-manifolds $(\pd^a\bfX)^\ci,\ab(\pd^b\bfY)^\ci,\ab\bfZ$ have orientation bundles with $\Z_2$-bundles $(\bs i^a_\bfX)^*(\Or_\bfX),(\bs i^b_\bfY)^*(\Or_\bfY),\ab\Or_\bfZ$. Thus Theorem \ref{sm3thm5} gives an orientation bundle on $(\pd^a\bfX)^\ci\t_{\bfZ}(\pd^b\bfY)^\ci$ with $\Z_2$-bundle $\Pi_X^*(\Or_\bfX)\ot_{\Z_2}\Pi_Y^*(\Or_\bfY)\ot_{\Z_2}\Pi_Z^*(\Or_\bfZ)$. 

In Definition \ref{sm4def9}, $\om_\bfW^c$ is data on $(\pd^c\bfW)_0^\ci\subseteq(\pd^c\bfW)^\ci$, and so is determined by the isomorphism \eq{sm4eq33} and these orientation bundles on $(\pd^a\bfX)^\ci\t_{\bfZ}(\pd^b\bfY)^\ci$ from Theorem \ref{sm3thm5}. We define $\om_\bfW^c$ to be the isomorphism in the orientation bundle given by Theorem \ref{sm3thm5} on each copy $(\pd^a\bfX)^\ci\t_{\bfZ}(\pd^b\bfY)^\ci$ in \eq{sm4eq33}, {\it except that\/} we must change the signs: iterating \eq{sm4eq32} $c$ times assigns a sign to each term in \eq{sm4eq30}, which depends not just on $a,b,l,n$ but on which of the $\binom{c}{a}$ copies of the $(\pd^a\bfX)^\ci\t_{\bfZ}(\pd^b\bfY)^\ci$ we are on, and we multiply $\om_\bfW^c$ by this sign.

This completes the definition of $(\Or_\bfW,\bs\om_\bfW=(\om_\bfW^0,\ldots,\om_\bfW^d))$, and Definition \ref{sm4def9}(a) for $(\Or_\bfW,\bs\om_\bfW)$ follows from Theorem \ref{sm3thm5}. We prove Definition \ref{sm4def9}(b) for $(\Or_\bfW,\bs\om_\bfW)$ by a similar argument to the proof of Theorem \ref{sm3thm5}. For the case $k=0$ in Definition \ref{sm4def9}(b), suppressing the isomorphism $\bs i_\bfW\vert_{(\pd\bfW)_0^\ci}$, equation \eq{sm4eq11} becomes
\begin{align*}
\pd\bigl((\om_\bfW^0)_*\bigl([(\bfX\t_\bfZ\bfY)_0^\ci]_\fund\bigr)\bigr)=\,&(\om_\bfW^1)_*\bigl([(\pd\bfX\t_\bfZ\bfY)_0^\ci]_\fund\bigr)\op\\
&(\om_\bfW^1)_*\bigl([(\bfX\t_\bfZ\pd\bfY)_0^\ci]_\fund\bigr)
\end{align*}
in $H_{d-1}^\St((\pd\bfX\t_\bfZ\bfY)_0^\ci,\Pi_W^*(\Or_\bfW),\Z)\op H_{d-1}^\St((\bfX\t_\bfZ\pd\bfY)_0^\ci,\Pi_W^*(\Or_\bfW),\Z)$. This is two separate equations, one for $\pd\bfX\t_\bfZ\bfY$ and one for $\bfX\t_\bfZ\pd\bfY$. We prove the one for $\pd\bfX\t_\bfZ\bfY$ by starting with equation \eq{sm4eq11} in Definition \ref{sm4def9}(b) for $\bfX$ with $k=0$, taking exterior tensor products $-\bt[(\bfY)_0^\ci]_\fund$ to get an equation on $X\t Y$, then taking cap products with $-\cap\Pd((\De_{Z_0})_*([Z_0]_\fund))$, and using a limiting argument as in the proof of Theorem \ref{sm3thm5}. The cases $k>0$ use the same idea, but with more notation. We leave the details to the reader.
\end{proof}

\subsection{Vector bundles and sections}
\label{sm48}

The material of \S\ref{sm38} on vector bundles and sections on s-manifolds generalizes immediately to s-manifolds with corners, with no new ideas.

\subsection{\texorpdfstring{$R$-weighted s-manifolds with corners}{R-weighted s-manifolds with corners}}
\label{sm49}

The material of \S\ref{sm39} on $R$-weighted s-manifolds also generalizes easily to $R$-weighted s-manifolds with corners. 

\begin{dfn}
\label{sm4def14}
Let $\bfX$ be an s-manifold with corners of dimension $n$, and $R$ a commutative ring. An $R$-{\it weighting\/} $\bs w$ for $\bfX$ is $\bs w=(w^0,w^1,\ldots,w^n)$, where $w^k$ is an $R$-weighting on the s-manifold $(\pd^k\bfX)^\ci$ in the sense of Definition \ref{sm3def11} for $k=0,\ldots,n$. Then we call $(\bfX,\bs w)$ an $R$-{\it weighted s-manifold with corners}. An unweighted s-manifold with corners $\bfX$ may be regarded as a weighted s-manifold with corners $(\bfX,\bs 1)$ with all weights~1.

We can now define {\it orientations\/} and {\it orientation bundles\/} on $(\bfX,\bs w)$ as in Definitions \ref{sm4def8} and \ref{sm4def9}, except that orientations or orientation bundles for $(\pd^k\bfX)^\ci$ are replaced by $R$-weighted versions for $((\pd^k\bfX)^\ci,w^k)$, and fundamental classes in \eq{sm4eq10} and \eq{sm4eq11} are replaced by $w^k$-weighted fundamental classes defined as in \eq{sm3eq35}. Theorem \ref{sm4thm1} then generalizes immediately to $R$-weighted fundamental classes, for example, \eq{sm4eq13} becomes
\begin{equation*}
\smash{\pd\bigl([\bfX^\ci]^{w^0}_\fund\bigr)=(i_\bfX)_*\bigl([(\pd\bfX)^\ci]^{w^1}_\fund\bigr).}
\end{equation*}
\end{dfn}

\begin{ex}
\label{sm4ex6}
Continue in the situation of Example \ref{sm3ex9}, so that $Y$ is a smooth manifold, $R$ a commutative ring, and $\al\in H_n^\ssi(Y,R)$ is a class in smooth singular homology.

Suppose $C_0,C_1\in C_n^\ssi(Y,R)$ are cycles with $\pd C_0=\pd C_1=0$ and $[C_0]=[C_1]=\al$. Then Example \ref{sm3ex9} defines compact, oriented $R$-weighted s-manifolds $(\bX_i,w_i)$ and smooth maps $f_i:\bX_i\ra\bY=Y$  from $C_i$ for $i=0,1$.

As $C_0,C_1$ are homologous, we can choose $B\in C_{n+1}^\ssi\bigl(Y\t[0,1],R\bigr)$ with $\pd B=-C_0\t\{0\}+C_1\t\{1\}$. Then by generalizing Example \ref{sm3ex9} to include boundaries over $Y\t\{0,1\}$, we can define a compact, oriented, $R$-weighted s-manifold with boundary $(\bfW,v)$ with a smooth map $\bfe:\bfW\ra \bfY=Y$, such that $\pd(\bfW,v)\cong -(\bX_0,w_0)\amalg (\bX_1,w_1)$ in compact, oriented, $R$-weighted s-manifolds without boundary, and this identifies $\bfe\vert_{\pd\bfW}$ with $f_0\amalg f_1:\bX_0\amalg\bX_1\ra\bY$.

The $R$-weighted version of \eq{sm4eq16} implies that
\begin{equation*}
-\inc_*\bigl([\bX_0]_\fund^{w_0}\bigr)+\inc_*\bigl([\bX_1]_\fund^{w_1}\bigr)=\inc_*\bigl([\pd\bfW]_\fund^{\pd v}\bigr)=0\quad\text{in $H_n^\St(W,R)$.}
\end{equation*}
Pushing this forward by $e:W\ra Y$ implies that
\begin{equation*}
(f_0)_*([\bX_0]_\fund^{w_0})=(f_1)_*([\bX_1]_\fund^{w_1})\quad \text{in $H_n^\St(Y,R)$.}
\end{equation*}
This shows that in Example \ref{sm3ex9}, as one would expect, the class $f_*([\bX]_\fund^w)$ depends only on $\al\in H_n^\ssi(Y,R)$, not on the choice of cycle $C$ with~$[C]=\al$.
\end{ex}

\section{Applying s-manifolds in Symplectic Geometry}
\label{sm5}

Many important areas of Symplectic Geometry involve mathematical structures defined by `counting' moduli spaces $\oM$ of $J$-holomorphic curves, including Gromov--Witten invariants \cite{FuOn,HWZ1}, quantum cohomology \cite{McSa1}, Lagrangian Floer cohomology \cite{AkJo,Fuka,FOOO}, Fukaya categories \cite{Auro,Seid2}, Symplectic Field Theory \cite{EGH}, Symplectic Cohomology \cite{Seid1}, and Contact Homology~\cite{EES}.

In each of these fields, one studies moduli spaces $\oM$ of $J$-holomorphic curves $u:\Si\ra X$ in a symplectic manifold $(X,\om)$, where $J$ is a (possibly domain-dependent) almost complex structure on $X$ compatible with $\om$, and the Riemann surface $\Si$ may have boundary, or nodal singularities, or marked points.

To `count' the moduli space $\oM$ one needs to give $\oM$ a geometric structure which `behaves like a compact oriented manifold', in the sense that it has a fundamental class $[\oM]_\fund$ in homology $H_*(\oM,\Q)$, or a fundamental chain in cases when $\oM$ has boundary and corners. This is difficult, because $\oM$ may be very singular, and not of the expected dimension, even for generic $J$.

There are three main approaches to this problem in the literature:
\begin{itemize}
\setlength{\itemsep}{0pt}
\setlength{\parsep}{0pt}
\item[(A)] (Fukaya's Kuranishi spaces.) Fukaya--Oh--Ohta--Ono \cite{Fuka,FOOO,FuOn} give moduli spaces $\oM$ the structure of {\it Kuranishi spaces\/} (see the author \cite{Joyc2,Joyc5,Joyc6}), and construct fundamental classes for them in $\Q$-homology.
\item[(B)] (Hofer's polyfolds.) Hofer--Wysocki--Zehnder \cite{HWZ1,HWZ2} give moduli spaces $\oM$ the structure of {\it polyfold Fredholm sections} (see \cite{FFGW} for a survey), and construct fundamental classes for them in $\Q$-homology.
\item[(C)] (The rest of the world.) Both (A),(B) are very difficult and technical, and take hundreds of pages to set up the foundations properly. Many authors simplify the problem by making restrictive assumptions on the geometry (e.g.\ by considering only symplectic manifolds $(X,\om)$ which are exact, or monotone, and only curves $\Si$ of genus 0), and perturbing the $J$-holomorphic curve equations (e.g.\ by domain-dependent almost complex structures), to ensure that the moduli spaces $\oM$ are manifolds (or pseudo-manifolds), so that defining fundamental classes/chains for $\oM$ is easy.
\end{itemize}

We wish to advocate a fourth approach:
\begin{itemize}
\setlength{\itemsep}{0pt}
\setlength{\parsep}{0pt}
\item[(D)] (S-manifolds.) Make {\it no restrictive assumptions\/} on the geometry, but perturb the $J$-holomorphic curve equations in such a way that all moduli spaces $\oM$ in the problem are compact oriented s-manifolds (possibly with corners, or $\Q$-weighted). Use fundamental classes as in \S\ref{sm33} and~\S\ref{sm44}.
\end{itemize}
Our definitions of s-manifolds, and s-manifolds with corners, above, have been very carefully designed to make this work. We conclude by giving a bit more detail on this programme:
\smallskip

\noindent{\bf(a)} Let $\oM$ be a moduli space of stable triples $(\Si,u,\bs z)$, where $\Si$ is a Deligne--Mumford prestable $J$-holomorphic curve, possibly with boundary, $u:\Si\ra X$ is a $J$-holomorphic map, possibly with $u(\pd\Si)\subset L$ for $L\subset X$ a Lagrangian, and $\bs z=(z_1,\ldots,z_k)$ is a $k$-tuple of marked points on $\Si$. We assume that the (possibly domain-dependent) almost complex structure $J$ is generic. We divide $\oM$ into strata $\oM^i$ according to:
\begin{itemize}
\setlength{\itemsep}{0pt}
\setlength{\parsep}{0pt}
\item[(i)] The topology of $\Si$, in particular its interior and boundary nodes.
\item[(ii)] Whether $u:\Si\ra X$ is a branched cover of another $J$-holomorphic map $u':\Si'\ra X$ over some component of $\Si$.
\item[(iii)] The finite symmetry group $\Aut(\Si,u,\bs z)$.
\end{itemize}

Typically there is one {\it nonsingular stratum\/} $\oM^{\rm ns}$ in which $\Si$ has no nodes, $u:\Si\ra X$ is not a branched cover, and $\Aut(\Si,u,\bs z)=\{1\}$. If $J$ is generic then each stratum $\oM^i$ is a smooth manifold, as in Definition \ref{sm3def1}(c). Most of Definition \ref{sm3def1}(a)--(e) is straightforward. For $\oM$ to be an s-manifold, or s-manifold with corners, the {\it key question\/} is whether the conditions on $\dim\oM^i$ hold. The easy case is when $\dim\oM^{\rm ns}=n$ and $\dim\oM^i\le n-2$ for~$i\ne{\rm ns}$.
\smallskip

\noindent{\bf(b)} The reason this is a difficult problem is that in many cases, with the most obvious definition of $\oM$, there can be singular strata $\oM^i$ with $\dim\oM^i\ge\dim\oM^{\rm ns}$. These mean that $\oM$ is not an s-manifold, and make defining $[\oM]_\fund$ much more difficult: we need a deeper geometric structure, such as Kuranishi spaces or polyfold Fredholm sections, and virtual techniques.

Our solution is to perturb the $J$-holomorphic equations --- to use a special kind of domain-dependent almost complex structure $\bs J$, which can depend on choices of extra marked points on $\Si$ --- which ensure that the conditions on $\dim\oM^i$ for $\oM$ to be an s-manifold (with corners) do hold, and orientations and fundamental classes for $\oM$ work.
\smallskip

\noindent{\bf(c)} In the Kuranishi space and polyfold theories (A),(B) above, there are significant technical difficulties in deciding what is the appropriate smooth structure to put on $\oM$ in the normal directions to a  singular stratum $\oM^i\subset\oM$. In polyfold theory \cite{HWZ1,HWZ2} this is done using a `gluing profile'. We avoid this problem completely by having {\it no smooth structure\/} normal to strata in s-manifolds. Therefore it is simpler, and less work, to prove a moduli space is an s-manifold, than a Kuranishi space or polyfold Fredholm section.
\smallskip

\noindent{\bf(d)} We divide moduli spaces $\oM$ into two types:
\begin{itemize}
\setlength{\itemsep}{0pt}
\setlength{\parsep}{0pt}
\item[(i)] Those for which a fundamental class $[\oM]_\fund$ is expected to exist in homology over $\Z$. These include moduli spaces $\oM_{0,k}(\al,J)$ of genus 0 curves with $k\ge 3$ marked points in a semipositive symplectic manifold $(X,\om)$, which may be used to define Gromov--Witten invariants and Quantum Cohomology for semipositive $(X,\om)$, and moduli spaces $\oM_{0,k,l}(\al,J)$ of genus 0 discs with $k\ge 0$ interior marked points and $l\ge 1$ boundary marked points, in a semipositive $(X,\om)$, with boundary in a Lagrangian $L\subset X$, which may be used to define Lagrangian Floer cohomology and Fukaya categories for semipositive $(X,\om)$.
\item[(ii)] The rest, for which a fundamental class $[\oM]_\fund$ is expected to exist in homology over $\Q$, but not over $\Z$.
\end{itemize}
The important difference is that for type (i), the virtual dimensions of moduli subspaces $\oM^\Ga\subset\oM$ of triples $(\Si,u,\bs z)$ with $\Aut(\Si,u,\bs z)\cong\Ga$ satisfy $\vdim\oM^\Ga<\vdim\oM$ for $\Ga\ne\{1\}$, but in case (ii) we can have $\vdim\oM^\Ga\ge\vdim\oM$. As triples $(\Si,u,\bs z)$ with $\Aut(\Si,u,\bs z)\cong\Ga$ should be `counted' with weight $1/\md{\Ga}$, fundamental classes can only be defined over $\Q$ in type~(ii).
\smallskip

\noindent{\bf(e)} For type (i), in future work with Guillem Cazassus and Alex Ritter the author hopes to give a systematic method of using domain-dependent almost complex structures $\bs J$ to make moduli spaces $\oM$ into compact, oriented s-manifolds $\bs\oM$ (possibly with corners), for generic~$\bs J$.

For type (ii), this method on its own is not sufficient. We use some additional geometric data on $X$ to determine, given $u:\Si\ra X$, some extra marked points $\{y_1,\ldots,y_N\}$ on $\Si$, which are naturally unordered, and on which $\Aut(\Si,u,\bs z)$ acts effectively. We choose an arbitrary ordering $(y_1,\ldots,y_N)$ of these extra marked points, and use a domain-dependent almost complex structure $\bs J$ which depends on the ordered set $(y_1,\ldots,y_N)$. In this way we {\it break the symmetry group\/} $\Aut(\Si,u,\bs z)$, so that $\oM$ contains only points $(\Si,u,\bs y,\bs z)$ with $\Aut(\Si,u,\bs y,\bs z)=\{1\}$. As there are $N!$ possible orderings of $\{y_1,\ldots,y_N\}$, we must weight strata with $N$ extra marked points by $1/N!$. So we make $\oM$ into a $\Q$-{\it weighted\/} s-manifold (with corners) $(\bs\oM,w)$, as in \S\ref{sm39} and \S\ref{sm49}, for generic $\bs J$. This is why $[\bs\oM]_\fund^w$ is defined in homology over~$\Q$.

\medskip

\noindent{\small\sc The Mathematical Institute, Radcliffe
Observatory Quarter, Woodstock Road, Oxford, OX2 6GG, U.K.

\noindent E-mail: {\tt joyce@maths.ox.ac.uk.}}

\end{document}